\documentclass[10pt]{article}
\usepackage{xcolor}
\usepackage{amsmath}
\usepackage{amssymb}
\usepackage{amsthm}
\usepackage{fullpage}
\usepackage{graphics}
\usepackage{graphicx}
\usepackage{subcaption}
\usepackage{caption}
\usepackage[font=footnotesize,labelfont=bf]{caption}
\usepackage[draft]{todonotes} 
\usepackage{mathtools}
\usepackage{bbm}
\usepackage{multicol}
\usepackage[shortlabels]{enumitem}

\usepackage[normalem]{ulem}
\newcommand{\rev}{}

\newlist{abbrv}{itemize}{1}
\setlist[abbrv,1]{label=,labelwidth=1in,align=parleft,itemsep=0.1\baselineskip,leftmargin=!}

\makeatletter
\renewcommand\paragraph{\@startsection{paragraph}{4}{\z@}%
	{-2.5ex\@plus -1ex \@minus -.25ex}%
	{1.25ex \@plus .25ex}%
	{\normalfont\normalsize\bfseries}}
\makeatother
\newtheorem{thm}{Theorem}
\newtheorem{lm}[thm]{Lemma}
\newtheorem{defn}[thm]{Definition}
\newtheorem{prop}[thm]{Proposition}
\newtheorem{rmk}[thm]{Remark}
\newtheorem{cor}[thm]{Corollary}
\newtheorem*{claim}{Claim}

\newcommand{\cH}{\mathcal{H}}

\newcommand{\prob}{\mathbb{P}}
\newcommand{\intZ}{\mathbb{Z}}
\newcommand{\realR}{\mathbb{R}}
\newcommand{\complexC}{\mathbb{C}}

\newcommand{\polylog}{\mathrm{Li}}
\newcommand{\sign}{\mathrm{sgn}}

\newcommand{\FS}{\mathbb{F}} 
\newcommand{\FSm}{\FS^{(m)}}
\newcommand{\FSmo}{\FS^{(m-1)}}
\newcommand{\FgSm}{\tilde{\mathbb{F}}^{(m)}}
\newcommand{\FgSmo}{\tilde{\mathbb{F}}^{(m-1)}}
\newcommand{\FgSmoo}{\tilde{\mathbb{F}}^{(1)}}

\newcommand{\fslm}{\hat{\sfs}} 

\newcommand{\conf}{\mathcal{X}}

\newcommand{\beq}{ \begin{equation} }
\newcommand{\eeq}{ \end{equation} }
\newcommand{\bes}{ \begin{split} }
\newcommand{\beqq}{ \begin{equation*} }
\newcommand{\eeqq}{ \end{equation*} }

\newcommand{\dd}{{\mathrm d}}
\newcommand{\ii}{\mathrm{i}}
\newcommand{\ddbar}[1]{\frac{{\mathrm d}#1}{2\pi {\mathrm i}#1}}
\newcommand{\ddbarr}[1]{\frac{{\mathrm d}#1}{2\pi {\mathrm i}}}

\newcommand{\LL}{{\rm L}}
\newcommand{\RR}{{\rm R}}

\newcommand{\zz}{Z}

\newcommand{\rr}{\mathbbm{r}}

\newcommand{\DL}{\mathcal{L}}
\newcommand{\DR}{\mathcal{R}}
\newcommand{\DRL}{\mathcal{H}}
\newcommand{\PQS}{\tilde{\mathcal{G}}}
\newcommand{\PQ}{\mathcal{G}}
\newcommand{\consto}{\mathcal{Q}}

\newcommand{\B}{B}

\newcommand{\DP}{\mathcal{P}}

\newcommand{\cI}{\mathcal{I}}
\newcommand{\ccI}{\tilde{\mathcal{I}}}
\newcommand{\cJ}{\mathcal{J}}

\newcommand{\height}{\mathbbm{h}}
\newcommand{\region}{\mathbf{R}}

\newcommand{\id}{1}

\newcommand{\roots}{\mathcal{R}}

\newcommand{\rootsL}{\mathrm{L}}
\newcommand{\rootsR}{\mathrm{R}}
\newcommand{\ql}{\mathrm{l}}
\newcommand{\qr}{\mathrm{r}}
\newcommand{\qol}{\ql} 
\newcommand{\qor}{\qr} 
\newcommand{\fftn}{F}
\newcommand{\fs}{f}
\newcommand{\hfs}{\mathtt{f}}
\newcommand{\hatf}{\hat{\fs}}

\newcommand{\const}{C}
\newcommand{\gdet}{D}

\newcommand{\UU}{\mathrm{U}}
\newcommand{\VV}{\mathrm{V}}
\newcommand{\bUU}{\mathbf{\UU}}
\newcommand{\bVV}{\mathbf{\VV}}
\newcommand{\mrootL}{\mathrm{S_1}}
\newcommand{\mrootR}{\mathrm{S_2}}
\newcommand{\KsL}{K_1} 
\newcommand{\KsR}{K_2} 
\newcommand{\kk}{k}
\newcommand{\tKsL}{\tilde{\KsL}} 
\newcommand{\tKsR}{\tilde{\KsR}}
\newcommand{\CA}{\mathrm{E}}
\newcommand{\sH}{H} 
\newcommand{\sQL}{\mathrm{Q}_1}
\newcommand{\sQR}{\mathrm{Q}_2}
\newcommand{\jac}{\mathrm{J}}
\newcommand{\Jb}{\jac}

\newcommand{\Dt}{\tilde{D}}

\newcommand{\xx}{\mathbbm{x}}

\newcommand{\isk}{\Omega}

\newcommand{\htta}{\mathfrak{H}} 

\newcommand{\bfx}{\mathbf{x}}
\newcommand{\bftau}{\boldsymbol{\tau}}
\newcommand{\bfgamma}{\boldsymbol{\gamma}}

\newcommand{\constc}{\mathcal{C}}

\newcommand{\ccc}{\mathsf{C}}
\newcommand{\gdetlm}{\mathsf{D}}
\newcommand{\cccm}{\ccc^{(m)}}
\newcommand{\gdetlmm}{\gdetlm^{(m)}}
\newcommand{\cccmo}{\ccc^{(m-1)}}
\newcommand{\gdetlmmo}{\gdetlm^{(m-1)}}
\newcommand{\gdetlmdd}{\mathsf{d}}
\newcommand{\gdetlmddm}{\mathsf{d}^{(m)}}
\newcommand{\gdetlmddmo}{\mathsf{d}^{(m-1)}}

\newcommand{\lmk}{\mathsf{k}}
\newcommand{\lmKL}{\mathsf{K_1}}
\newcommand{\lmKR}{\mathsf{K_2}}
\newcommand{\tlmKL}{\mathsf{\tilde{K}_1}}
\newcommand{\tlmKR}{\mathsf{\tilde{K}_2}}

\newcommand{\inodesL}{\mathsf{L}}
\newcommand{\inodesR}{\mathsf{R}}

\newcommand{\SSSL}{\mathsf{S_1}}
\newcommand{\SSSR}{\mathsf{S_2}}
\newcommand{\hftn}{\mathsf{h}}

\newcommand{\QWL}{\mathsf{Q_1}}
\newcommand{\QWR}{\mathsf{Q_2}}

\newcommand{\sfs}{\mathsf{f}}
\newcommand{\UUlm}{\mathsf{U}} 
\newcommand{\VVlm}{\mathsf{V}} 
\newcommand{\WWlm}{\mathsf{W}} 
\newcommand{\wsw}{\mathsf{w}} 
\newcommand{\bUUlm}{\mathbf{\UUlm}}
\newcommand{\bVVlm}{\mathbf{\VVlm}}
\newcommand{\bUUlmk}{\bUUlm^{[k]}} 
\newcommand{\bVVlmk}{\bVVlm^{[k]}} 
\newcommand{\su}{\mathsf{u}} 
\newcommand{\sv}{\mathsf{v}} 
\newcommand{\sw}{\zeta} 
\newcommand{\spp}{\mathrm{p}}  
\newcommand{\bp}{\mathbf{p}}
\newcommand{\bpk}{\bp^{[k]}}
\newcommand{\bx}{\mathbf{x}}
\newcommand{\bxk}{\bx^{[k]}}

\newcommand{\detv}{\mathcal{D}}
\newcommand{\detvv}{\mathcal{E}}

\newcommand{\Cb}{}

\newcommand{\bfz}{\mathbf{z}}
\newcommand{\bfzk}{\bfz^{[k]}}

\newcommand{\bbz}{\mathbf{z}}

\newcommand{\bW}{\mathbf{W}}
\newcommand{\bz}{Z} 
\newcommand{\bzz}{\mathbf{\zz}}
\newcommand{\bk}{\mathbf{k}}
\newcommand{\bt}{\mathbf{t}}
\newcommand{\ba}{\mathbf{a}}

\newcommand{\bn}{\mathbf{n}}
\newcommand{\bnk}{\bn^{[k]}}
\newcommand{\fgic}{\mathcal{G}}

\newcommand{\gn}{g}

\newcommand{\mr}{\mathbf}

\newcommand{\cnt}{\mathcal{B}}

\newcommand{\ww}{\hat{\omega}}

\newcommand{\tW}{\tilde{W}}
\newcommand{\tU}{\tilde{U}}
\newcommand{\tV}{\tilde{V}}
\newcommand{\tS}{\tilde{S}}

\DeclareMathOperator{\st}{step}

\newcommand{\rF}{\mathcal{F}}
\newcommand{\rG}{\mathcal{G}}
\newcommand{\rI}{\mathcal{I}}
\newcommand{\rJ}{\mathcal{J}}

\newcommand{\C}{\mathbb{C}}
\newcommand{\Z}{\mathbb{Z}}
\newcommand{\R}{\mathbb{R}}

\newcommand{\fontlm}{\mathsf}

\numberwithin{equation}{section} 
\numberwithin{thm}{section}

\author{Jinho Baik\footnote{Department of Mathematics, University of Michigan,
Ann Arbor, MI, 48109. Email: \texttt{baik@umich.edu}} 
 and Zhipeng Liu\footnote{Department of Mathematics, University of Kansas, Lawrence, KS 66045. Email: \texttt{zhipeng@ku.edu}}}
\date{\today}

\begin{document}
	\title{Multi-point distribution of periodic TASEP} 
	
	\date{\today}

	\maketitle
	
\begin{abstract}
The height fluctuations of the models in the KPZ class are expected to converge to a universal  process. 
The spatial process at equal time is known to converge to the Airy process or its variations. However, the temporal process, or more generally {the two-dimensional space-time fluctuation field, is less well understood.}
We consider this question for the periodic TASEP (totally asymmetric simple exclusion process). 
For a particular initial condition, we evaluate the multi-time and multi-location 
distribution explicitly in terms of a multiple integral involving a Fredholm determinant.  
We then evaluate the large time limit in the so-called relaxation time scale. 
\end{abstract}

\section{Introduction}\label{sec:int}

{The models in the KPZ universality class are expected to have the 1:2:3 \rev{scaling} for the 
height fluctuations, spatial correlations, and time correlations as time $t\to \infty$.
This means that the scaled two-dimensional fluctuation field 
\beq
h_t(\gamma, \tau) := \frac{H(c_1\gamma(\tau t)^{2/3}, \tau t) - \left(c_2(\tau t)+c_3(\tau t)^{2/3}\right)}{c_4(\tau t)^{1/3}}
\eeq
of the height} function $H(\ell, t)$, where $\ell$ is the spatial variable and $t$ is time, is believed to converge  to a universal field which  depends only on the initial condition.\footnote{
For some initial conditions, such as the stationary initial condition, one may need to translate the space in the characteristic direction.} 
{Here $c_1, c_2, c_3, c_4$ are model-dependent constants.}
Determining the limiting {two-dimensional fluctuation field}
\beq \label{eq:twfid}
	(\gamma, \tau) \mapsto h(\gamma, \tau):= \lim_{t\to \infty} h_t(\gamma, \tau)
\eeq
is an outstanding question.

By now there are several results for the one-point distribution.
The one-point distribution of $h(\gamma, \tau)$ for fixed $(\gamma, \tau)$ is given by 
random matrix distributions (Tracy-Widom distributions) or their generalizations.
The convergence 
is proved for a quite long list of models including PNG, TASEP, ASEP, $q$-TASEP, random tilings, last passage percolations, directed polymers, the KPZ equation, and so on. See, for example, \cite{Baik-Deift-Johansson99, Johansson00, Tracy-Widom09, Amir-Corwin-Quastel11, Borodin-Corwin-Ferrari14}, and the review article \cite{Corwin11}.
These models were studied using various integrable methods under standard initial conditions.
See also the recent papers \cite{Corwin-Liu-Wang16, Quastel-Remenik16} for general initial conditions.

The spatial one-dimensional process, $\gamma \mapsto h(\gamma, \tau)$ for fixed $\tau$, is also well understood. 
This process is given by the Airy process and its variations.
However, the convergence 
is proved rigorously only for a smaller number of models.
It was proved for the determinantal models like PNG, TASEP, last passage percolation,\footnote{See, for example, \cite{Prahofer-Spohn02, Johansson03, Imamura-Sasamoto04, Borodin-Ferrari-Prahofer-Sasamoto07, Borodin-Ferrari-Prahofer07, Borodin-Ferrari-Sasamoto08a, Baik-Ferrari-Peche10} for special initial conditions. See the recent paper \cite{Matetski-Quastel-Remenik17} for general initial conditions for TASEP.} but not yet for other integrable models such as ASEP, $q$-TASEP, finite-temperature directed polymers, and the KPZ equation.

{The two-dimensional fluctuation field,} $(\gamma, \tau) \mapsto h(\gamma, \tau)$, on the other hand, is less well understood. 
The joint distribution is known only for the two-point distribution. 
In 2015, Johansson \cite{JohanssonTwoTime} considered the zero temperature Brownian semi-discrete directed polymer and computed the limit of the two-point (in time and location) distribution.\footnote{There are non-rigorous physics papers for the two-time distribution of directed polymers \cite{DotsenkoTime1, DotsenkoTime2, DotsenkoTime3}. However, another physics paper \cite{DeNardisLeDoussalTwoTime} indicates that the formulas in these papers are not correct.}  
The limit is obtained in terms of rather complicated 
series involving the determinants of matrices whose entries contain the Airy kernel. 
\Cb{The formula is simplified more recently in terms a contour integral of a Fredholm determinant in \cite{Johansson18} in which the author also extended his work to the directed last passage percolation model with geometric weights.}
Two other papers studied qualitative behaviors of the temporal correlations. 
Using a variational problem involving two independent Airy processes, Ferrari and Spohn \cite{FerrariSpohnTimeCorrel} proved in 2016 the power law of the covariance in the time direction in the large and small time limits, $\tau_1/\tau_2\to 0$ and $\tau_1/\tau_2\to 1$. 
Here, $\tau_1$ and $\tau_2$ denote the scaled time parameters. 
De Nardis and Le Doussal \cite{DeNardisLeDoussalTwoTime} extended this work further and also augmented by other physics arguments to compute the similar 
limits of the two-time distribution when one of the arguments is large. 
It is yet to be seen if one can deduce these results from the formula of Johansson.

\bigskip

The objective of this paper is to study the {two-dimensional fluctuation field of \emph{spatially periodic} KPZ models.}
Specifically, we evaluate the multi-point distribution of the periodic TASEP (totally asymmetric simple exclusion process) and compute a large time limit in a certain critical regime.

We denote by $L$ the period and by $N$ the number of particles per period. 
Set $\rho=N/L$, the average density of particles. 
The periodic TASEP (of period $L$ and density $\rho$) is defined by the occupation function $\eta_j(t)$ satisfying the spatial periodicity:
\begin{equation}
	\eta_j(t)=\eta_{j+L}(t),\qquad j\in \intZ, \quad t\ge 0.
\end{equation}
Apart from this condition, the particles follow the usual TASEP rules. 

Consider the limit as $t, L, N\to\infty$ with fixed $\rho=N/L$. 
Since the spatial fluctuations of the usual infinite TASEP is $O(t^{2/3})$, 
{all of the particles in the periodic TASEP are correlated when}
$t^{2/3}=O(L)$. 
We say that the periodic TASEP is in the \emph{relaxation time scale} if 
\beq
	t=O(L^{3/2}).
\eeq 
If $t\ll L^{3/2}$, we expect that {the system size has negligible effect}  and, therefore, the system follows the KPZ dynamics.
See, for example, \cite{Baik-Liu16b}.
On the other hand, if $t\gg L^{3/2}$, then the system is basically in a finite system, and hence we expect the  \rev{stationary} dynamics. See, for example, \cite{Derrida-Lebowitz98}.
Therefore, in the relaxation time scale,  we predict that the KPZ dynamics and the  \rev{stationary} dynamics 
are both present.

Even though the periodic TASEP is as natural as the infinite TASEP, the one-point distribution was obtained only recently. 
Over the last two years, in a physics paper \cite{Prolhac16} and, independently, in mathematics papers \cite{Baik-Liu16, Liu16}, the authors evaluated the one-point function of the height function in finite time and computed the large time limit in the relaxation time scale. 
The one-point function follows the the KPZ scaling $O(t^{1/3})$ but the limiting distribution is different from \rev{that of} the infinite TASEP.\footnote{The formulas obtained in \cite{Baik-Liu16, Liu16} and \cite{Prolhac16} are similar, but different. It is yet to be checked that these formulas are the same.} 
This result was obtained for the three initial conditions of periodic step, flat, and stationary.
{Some  \rev{earlier related studies} can be found in physics papers \cite{Gwa-Spohn92, Derrida-Lebowitz98, Priezzhev2003, Golinelli-Mallick04, Golinelli-Mallick05, Povolotsky-Priezzhev07, Gupta-Majumdar-Godreche-Barma07, Poghosyan-Priezzhev08}}\rev{, including results on the large deviation and spectral properties of the system.} 

In this paper, we {extend the analysis of the papers} \cite{Baik-Liu16, Liu16} and compute the multi-point (in time and location) distribution {of the periodic TASEP with a special initial condition called the periodic step initial condition.}
Here we allow any number of points unlike the previous work of Johansson on the infinite TASEP.
It appears that the periodicity of the model simplifies the algebraic computation compared with the infinite TASEP.  
In a separate paper we will consider flat and stationary initial conditions.
The main results are the following: 
\begin{enumerate}
\item {For arbitrary initial conditions, we evaluate finite-time joint distribution functions of the periodic TASEP at multiple points in the space-time coordinates}  in terms of a multiple integral involving a determinant of size $N$. See Theorem~\ref{prop:multipoint_distribution_origin} and Corollary~\ref{cor:pTASEPgenIC}.
\item For the periodic step initial condition, we simplify the determinant to a Fredholm determinant. See Theorem~\ref{thm:multi_point_formula}  and Corollary~\ref{cor:pTASEPstep}.
\item We compute the large time limit of the multi-point (in the space-time coordinates) distribution in the relaxation time scale for the periodic TASEP with the periodic step initial condition. 
See Theorem~\ref{thm:htfntasy}.
\end{enumerate}

\bigskip

One way of studying the usual infinite TASEP is the following. 
First, one computes the transition probability using the coordinate Bethe ansatz method. 
This means that we solve the Kolmogorov forward equation explicitly after replacing it (which contains complicated interactions between the particles) by the free evolution equation with certain boundary conditions. 
In \cite{Schutz97}, Sch\"utz obtained the transition probability of the infinite TASEP. 
Second, one evaluates the marginal or joint distribution by taking a sum of the transition probabilities.
It is important that the resulting expression should be suitable for the asymptotic analysis.
This is achieved typically by obtaining a Fredholm determinant formula. 
In \cite{Rakos-Schutz05},  R\'akos and Sch\"utz re-derived the famous finite-time Fredholm determinant formula of Johansson \cite{Johansson00} for the one-point distribution in the case of the step initial condition using this procedure. 
Subsequently, Sasamoto \cite{Sasamoto05} and Borodin, Ferrari, Pr\"ahofer, and Sasamoto \cite{Borodin-Ferrari-Prahofer-Sasamoto07}  obtained a Fredholm determinant formula for the joint distribution of multiple points with equal time.
This was further extended by Borodin and Ferrari \cite{Borodin-Ferrari08} to the points in spatial directions of each other. 
However, it was not extended to the case when the points are temporal directions of each other. 
The third step is to analyze the finite-time formula asymptotically using the method of steepest-descent. 
See \cite{Johansson00,Sasamoto05,Borodin-Ferrari-Prahofer-Sasamoto07,Borodin-Ferrari08} and also a more recent paper \cite{Matetski-Quastel-Remenik17}.
In the KPZ 1:2:3 scaling limit, the above algebraic formulas  give only the spatial process 
 $\gamma \mapsto h(\gamma, \tau)$. 

We applied the above procedure to the one-point distribution of the periodic TASEP in \cite{Baik-Liu16}.  
We obtained a formula for the transition probability, which is a periodic analogue of the formula of Sch\"utz.
Using that, we computed the finite-time one-point distribution for arbitrary initial condition.
The distribution was given by an integral of a determinant of size $N$. 
We then simplified the determinant to a Fredholm determinant for the cases of the step and flat initial conditions. 
The resulting expression was suitable for the asymptotic analysis. 
A similar computation for the stationary initial condition was carried out in \cite{Liu16}. 

In this paper, we extend the analysis of \cite{Baik-Liu16, Liu16} to multi-point distributions. 
For general initial conditions, we evaluate the joint distribution by taking a multiple sum of the transition probabilities obtained in \cite{Baik-Liu16}. 
The computation can be reduced to an evaluation of a sum involving only two arbitrary points in the space-time coordinates (with different time coordinates.)
The main technical result of this paper, presented in Proposition~\ref{thm:DRL_simplification}, is the evaluation of this sum in a compact form. 
The key point, compared with the infinite TASEP \cite{Borodin-Ferrari-Prahofer-Sasamoto07, Borodin-Ferrari08,Matetski-Quastel-Remenik17}, is that the points do not need to be restricted to the spatial directions.\footnote{In the large time limit, we add a certain restriction when the re-scaled times are equal. See Theorem~\ref{thm:htfntasy}.
The outcome of the above computation 
is that  we find the joint distribution in terms of a multiple integral involving a determinant of size $N$. 
For the periodic step initial condition, we simplify the determinant further to a Fredholm determinant.}
The final formula is suitable for the large-time asymptotic analysis in relaxation time scale.

\bigskip

If we take the period $L$ to infinity while keeping other parameters fixed, the periodic TASEP becomes the infinite TASEP. 
Moreover, it is easy to check (see Section~\ref{sec:inftasep} below) that the joint distributions of the periodic TASEP and the infinite TASEP are equal even for fixed $L$ if $L$ is large enough compared with the times. 
Hence, the finite-time joint distribution formula obtained in this paper (Theorem~\ref{prop:multipoint_distribution_origin} and Corollary~\ref{cor:pTASEPgenIC}) 
in fact, gives a formula of the joint distribution of the infinite TASEP; see the equations~\eqref{eq:aux_2017_07_01_10} and~\eqref{eq:aux_2017_08_14_01}.  
This formula contains an auxiliary parameter $L$ which has no meaning in the infinite TASEP. 
From this observation, we find that if we take the large time limit of our formula in the sub-relaxation time scale, $t\ll L^{3/2}$, 
then the limit, if it exists, is the joint distribution of the two-dimensional process $h(\gamma, \tau)$ in~\eqref{eq:twfid}. 
However, it is not  
clear at this moment if our formula is suitable for the asymptotic analysis {in the sub-relaxation time scale; 
the kernel of the operator} in the Fredholm determinant does not seem to converge in the sub-relaxation time scale while it converges in the relaxation time scale. 
The question of computing the limit in the sub-relaxation time scale, and hence the multi-point distribution of the infinite TASEP, will be left as a later project.

\bigskip

This paper is organized as follows. 
We state the limit theorem in Section~\ref{sec:limitdisformula}.
The finite time formula for general initial conditions is in Section~\ref{sec:tran}. 
Its simplification for the periodic step initial condition is obtained in Section~\ref{sec:stepicn}. 
{In Section~\ref{sec:proof_DRL}, we prove Proposition~\ref{thm:DRL_simplification}, the key algebraic computation.} 
The asymptotic analysis of the formula obtained in Section~\ref{sec:stepicn} is carried out in Section~\ref{sec:pfastp}, proving the result in Section~\ref{sec:limitdisformula}.  
We discuss some properties of the limit of the joint distribution in Section~\ref{sec:consistency}. 
In Section~\ref{sec:inftasep} we show that the finite-time formulas obtained in Sections~\ref{sec:tran} and~\ref{sec:stepicn} are also valid for infinite TASEP for all large enough $L$. 

\subsubsection*{Acknowledgments}
The work of Jinho Baik was supported in part by NSF grants DMS-1361782, DMS-1664531  and DMS-1664692, and the Simons Fellows program. The work was done in part when Zhipeng Liu was at Courant Institute, New York University.

\section{Limit theorem for multi-point distribution} \label{sec:limitdisformula}

\subsection{Limit theorem}

Consider the periodic TASEP of period $L$ with $N$ particles per period. 
We set $\rho=N/L$, the average particle density. 
We assume that the particles move to the right. 
Let $\eta_j(t)$ be the occupation function of periodic TASEP: $\eta_j(t)=1$ if the site $j$ is occupied at time $t$, otherwise $\eta_j(t)=0$, and it satisfies the periodicity $\eta_j(t)=\eta_{j+L}(t)$. 
We consider the \emph{periodic step initial condition} defined by 
\beq \label{eq:icst}
\bes
	& \eta_{j}(0) = \begin{dcases} 1 \qquad \text{for $-N+1\le j\le 0$,}\\
	0 \qquad \text{for $1\le j\le L-N$,} \end{dcases} \\
\end{split}
\eeq
and $\eta_{j+L}(0) =\eta_{j}(0)$. 

We state the results in terms of the height function
\begin{equation}
\label{eq:def_height_2}
	\height(\mr p)
	\qquad \text{where $\mr p=\ell \mr e_1+t\mr e_2=(\ell,t)\in\intZ\times\realR_{\ge 0}$.}
\end{equation} 
Here $\mr e_1=(1,0)$ and $\mr e_2=(0,1)$ are the unit coordinate vectors in the spatial and time directions, respectively.
The height function is defined by
\begin{equation}
\label{eq:def_height_1}
	\height(\ell \mr e_1+t\mr e_2) 
	= \begin{dcases}
2J_0(t)+\sum_{j=1}^\ell\left(1-2\eta_j(t)\right),&\ell\ge 1,\\
2J_0(t),										 &\ell=0,\\
2J_0(t)-\sum_{j=\ell+1}^0\left(1-2\eta_j(t)\right),&\ell\le -1,
\end{dcases}
\end{equation}
where $J_0(t)$ counts the number of particles jumping through the bond from $0$ to $1$ during the time interval $[0,t]$. 
The periodicity implies that 
\beq
	\height((\ell +nL)\mr e_1+t\mr e_2)=\height(\ell \mr e_1+t\mr e_2)+n(L-2N)
\eeq
for integers $n$.

See Figure~\ref{fig:density_profile} for the evolution of the density profile and Figure~\ref{fig:heightevaoluation} for the limiting height function. \rev{Note that the step initial condition~\eqref{eq:icst} generates shocks\footnote{These shocks are generated when faster particles from lower density region enter higher density region and are forced to slow down. \Cb{See, for example, \cite{Corwin-Ferrari-Peche10,Ferrari-Nejjar15} for the study of similar behaviors in infinite TASEP.}}.
By solving the Burgers' equation in a periodic domain, one could derive the explicit formulas of the density profile, the limiting height function and the shock location. These computations were done in \cite{Baik-Liu16b}.} 

\begin{figure}
	\centering
	\includegraphics[scale=0.5]{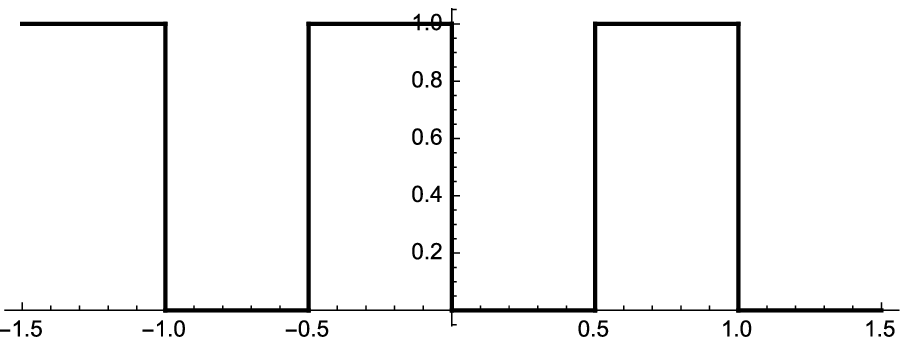} \quad 
	\includegraphics[scale=0.5]{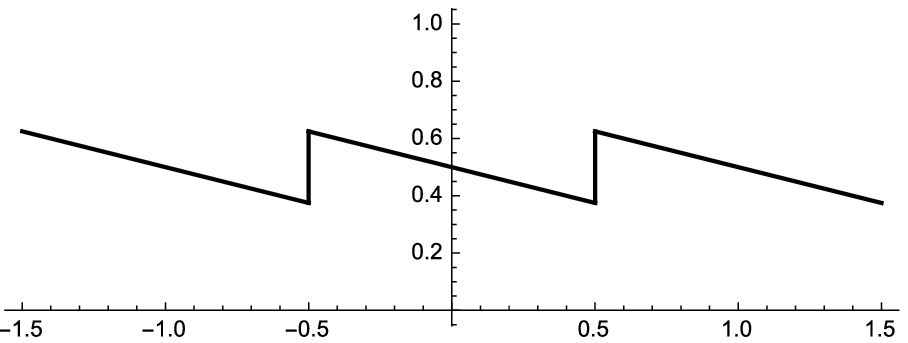} \quad 
	\includegraphics[scale=0.5]{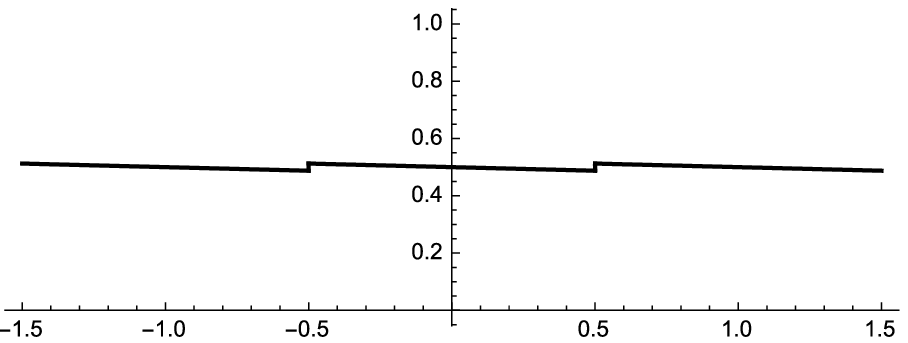}
	\caption{{The pictures represent the density profile at time $t=0$, $t=L$ and $t=10L$ when $\rho=1/2$. The horizontal axis is scaled down by $L$.}}
\label{fig:density_profile}
\end{figure}

\begin{figure}
	\centering
	\includegraphics[scale=0.6]{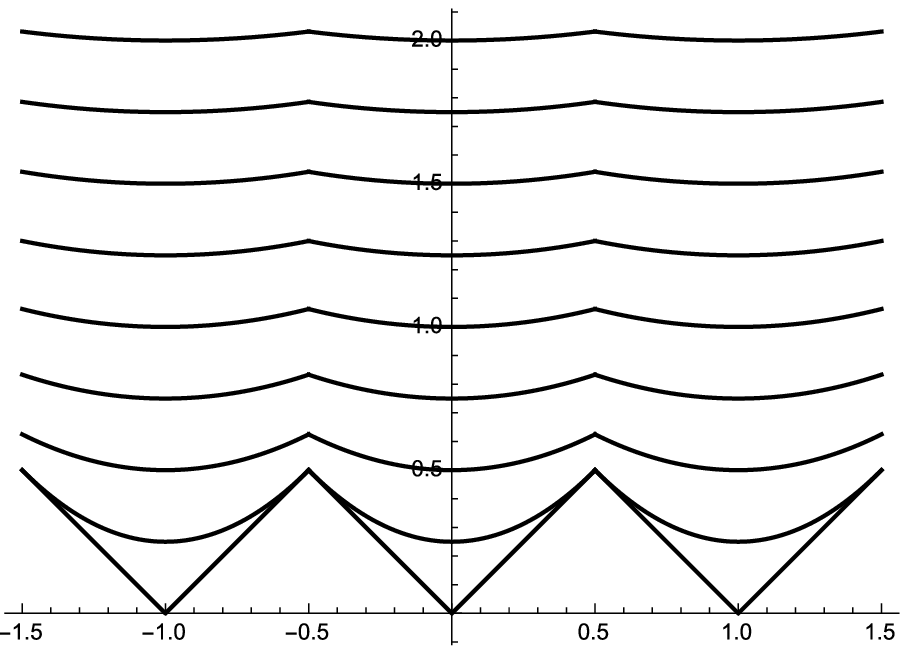} \qquad 
	\includegraphics[scale=0.6]{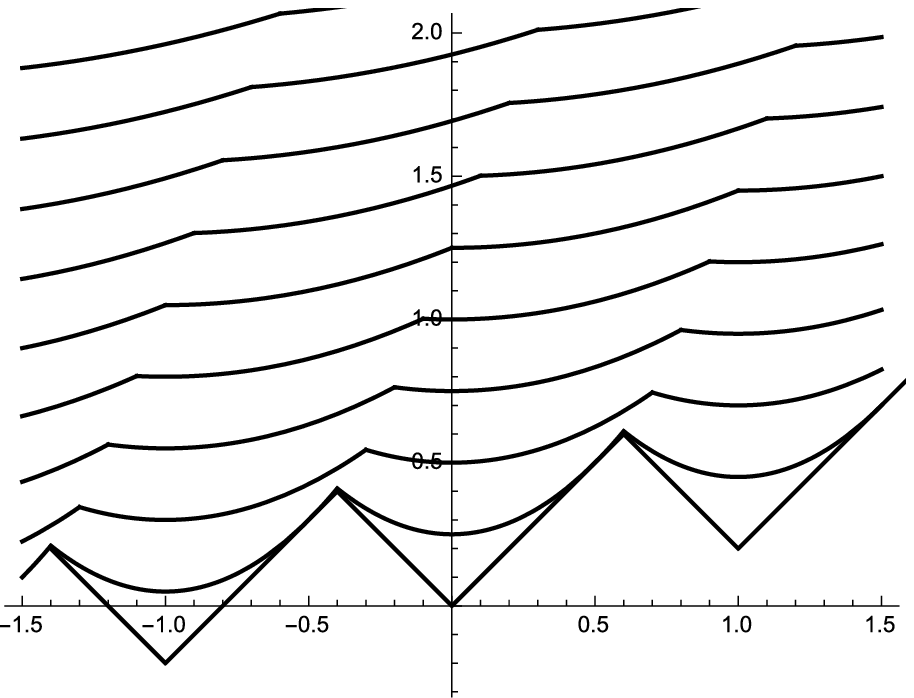} 
	\caption{{The pictures represent the limiting height function at times $t= 0.5nL$ for $n=0,1,2,\cdots$. The left picture is when $\rho=1/2$ and the right picture is when $\rho=2/5$. Both horizontal axis (location) and vertical axis (height) are scaled down by $L$.
	}}
\label{fig:heightevaoluation}
\end{figure}

\bigskip

We represent the space-time position in new coordinates.
Let 
\begin{equation}
	\mr e_c:=(1-2\rho)\mr e_1+\mr e_2 
\end{equation}
be a vector parallel to the characteristic directions. 
If we represent $\mr p= \ell \mr e_1+t\mr e_2$ in terms of $\mr e_1$ and $\mr e_c$, then
\begin{equation}
	\mr p = s \mr e_1+t \mr e_c \qquad \text{where $s=\ell- t(1-2\rho)$.}
\end{equation}
Consider the region 
\begin{equation}
	\region:=\{\ell \mr e_1+t\mr e_2\in\intZ\times\realR_{\ge 0}: 0\le \ell-(1-2\rho) t\le L\}
	= \{s \mr e_1+t \mr e_c \in\intZ\times\realR_{\ge 0}: 0\le s\le L\}.
\end{equation}
See Figure~\ref{fig:points_space_time}. 
Due to the periodicity, the height function in $\region$ determines the height function in the whole space-time plane. 
\begin{figure}
	\centering
	\includegraphics[scale=0.25]{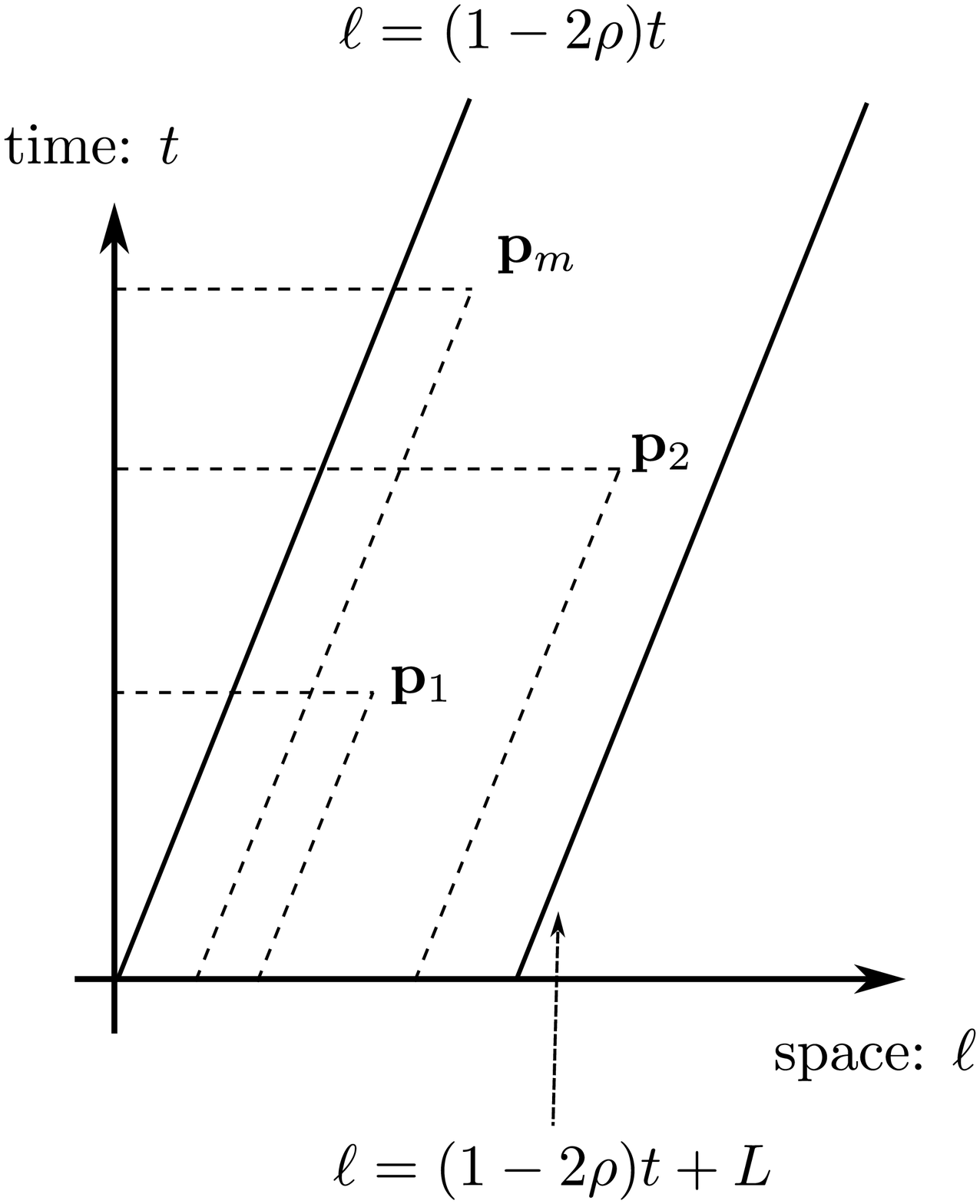}
	\caption{Illustration of the points $\mr p_j$, $j=1,\cdots,m$, in the region $\region$.}
	\label{fig:points_space_time}
\end{figure}

The following theorem is the main asymptotic result. We take the limit as follows. 
We take $L,N\to\infty$ in such a way that the average density $\rho=N/L$ is fixed, or more generally $\rho$ stays in a compact subset of the interval $(0,1)$. 
We consider $m$ distinct points $\mr p_i = s_i \mr e_1+t_i \mr e_c$ in the space-time plane such that their temporal coordinate $t_j\to\infty$ and satisfy the relaxation time scale $t_j=O(L^{3/2})$. 
The relative distances of the coordinates are scaled as in the $1:2:3$ KPZ prediction:
$t_i-t_j=O(L^{3/2})=O(t_i)$, $s_i=O(L)=O(t_i^{2/3})$, and the height at each point is scaled by $O(L^{1/2})=O(t_i^{1/3})$.

\begin{thm}[Limit of multi-point joint distribution for periodic TASEP]
	\label{thm:htfntasy}
Fix two constants $c_1$ and $c_2$ satisfying $0<c_1<c_2<1$. 
Let $N=N_L$ be a sequence of integers such that $c_1L\le N\le c_2L$ for all sufficiently large $L$. 
Consider the periodic TASEP of period $L$ and average particle density $\rho=\rho_L=N/L$.
Assume the periodic step  initial condition~\eqref{eq:icst}. 
Let $m$ be a positive integer.
Fix $m$ points $\fontlm{p}_j=(\gamma_j,\tau_j)$, $j=1, \cdots, m$, in the region
\begin{equation} \label{eq:def_region_points_after}
	\fontlm{R}:=\{(\gamma,\tau)\in\realR\times\realR_{>0}:0\le \gamma\le 1\}.
\end{equation}
Assume that 
\beq \label{eq:tauassumeor}
	\tau_1<\tau_2<\cdots< \tau_m. 
\eeq
Let $\mr p_j=s_j\mr e_1+t_j\mr e_c$ be $m$ points\footnote{Since $\mr p_j$ should have an integer value for its spatial coordinate, to be precise, we need to take the integer part of $s_j+t_j(1-2\rho)$ for the spatial coordinate. This small distinction does not change the result since the limits are uniform in the parameters $\gamma_j, \tau_j$. Therefore, we suppress the integer value notations throughout this paper.} in the region $\region$ shown in Figure~\ref{fig:points_space_time}, where $\mr e_1=(1,0)$ and $\mr e_c=(1-2\rho,1)$, with
\begin{equation} \label{eq:scaling_paremeters}
\begin{split}
	s_j=\gamma_j L,\qquad 
	t_j=\tau_j \frac{L^{3/2}}{\sqrt{\rho(1-\rho)}} .
\end{split}
\end{equation}
Then, for arbitrary fixed $x_1,\cdots,x_m\in\realR$,
\begin{equation}
	\label{eq:main_theorem}
	\lim_{L\to\infty}\prob\left(\bigcap_{j=1}^m\left\{\frac{\height(\mr p_j)-(1-2\rho)s_j-(1-2\rho+2\rho^2)t_j}{-2\rho^{1/2}(1-\rho)^{1/2} L^{1/2}}\le x_j\right\}\right)=\FS (x_1,\cdots,x_m;\fontlm{p}_1,\cdots,\fontlm{p}_m )
\end{equation}
where the function $\FS$ is defined in~\eqref{eq:def_mpdist}.
The convergence is locally uniform in $x_j, \tau_j$, and $\gamma_j$. 
If $\tau_i=\tau_{i+1}$ for some $i$, then~\eqref{eq:main_theorem} still holds if we assume that $x_{i}<  x_{i+1}$.  
\end{thm}

\begin{rmk} 
Suppose that we have arbitrary $m$ distinct points $\fontlm{p}_j=(\gamma_j,\tau_j)$ in $\fontlm{R}$.
Then we may rearrange them so that $0<\tau_1\le \cdots \le \tau_m$. 
{If $\tau_j$ are all different, we can} apply the above theorem since the result holds for arbitrarily ordered $\gamma_j$. 
If some of $\tau_j$ are equal, then we may rearrange the points further so that $x_j$ are ordered with those $\tau_j$, and use the theorem if $x_j$ are distinct. 
The only case which are not covered by the above theorem is when some of $\tau_j$ are equal and the corresponding $x_j$ are also equal. 
\end{rmk}

The case when $m=1$ was essentially obtained in our previous paper \cite{Baik-Liu16} (and also \cite{Prolhac16}.) 
In that paper, we considered the location of a tagged particle instead of the height function, but it straightforward to translate the result to the height function. 

\begin{rmk}
We will check that $\FS(x_1,\cdots,x_m; \mathrm{p}_1,\cdots,\mathrm{p}_m)$ is periodic with respect to each of the space coordinates $\gamma_j$ in Subsection~\ref{sec:limtjtdis}. 
By this spatial periodicity, we can remove the restrictions $0\le \gamma_j\le 1$ in the above theorem.
\end{rmk}

\begin{rmk}
Since we expect the KPZ dynamics in the sub-relaxation scale $t_j\ll L^{3/2}$, 
we expect that the $\tau_j\to 0$ limit of the above result should give rise to a result for the usual infinite TASEP. 
Concretely, we expect that the limit 
\beq \label{eq:zizz}
\lim_{\tau\to 0} \FS((\tau_1\tau)^{1/3}x_1, \cdots, (\tau_m\tau)^{1/3}x_m; (\gamma_1(\tau_1\tau)^{2/3}, \tau_1 \tau), \cdots (\gamma_m(\tau_m\tau)^{2/3}, \tau_m \tau))
\eeq
exists and it is the limit of the multi-time, multi-location joint distribution of 
the height function $\htta(s,t)$ of the usual TASEP with step initial condition, 
\begin{equation}
\lim_{T \to\infty}\prob\left(\bigcap_{j=1}^m\left\{\frac{\htta(\gamma_j \tau_j^{2/3}T^{2/3}, 2\tau_j T) - \tau_j T}{- \tau_j^{1/3}T^{1/3}}\le x_j\right\}\right) .
\end{equation}
\rev{In particular, we expect that when $m=1$,~\eqref{eq:zizz} is $F_{GUE}(x_1+\gamma_1^2/4)$, the Tracy-Widom GUE distribution; we expect that when $\tau_1=\cdots=\tau_m$,~\eqref{eq:zizz} is equal to the corresponding joint distribution of the Airy$_2$ process \cite{Prahofer-Spohn02}  $A_2(\gamma/2)-\gamma^2/4$; and when $m=2$,~\eqref{eq:zizz} is expected to match the two-time distribution $F_{\textrm{two-time}}(\gamma_1/2,x_1+\gamma_1^2/4;\gamma_2/2,x_2+\gamma_2^2/4;\tau_1^{1/3}(\tau_2-\tau_1)^{-1/3})$ obtained by Johansson \cite{JohanssonTwoTime,Johansson18}.}
See also Section~\ref{sec:inftasep}.

\end{rmk}

\subsection{Formula of the limit of the joint distribution} \label{sec:limtjtdis}

\subsubsection{Definition of $\FS$}\label{sec:defFS}


\begin{defn} \label{def:FSdefn}
Fix a positive integer $m$.
Let $\fontlm{p}_j=(\gamma_j, \tau_j)$ for each $j=1, \cdots, m$ where $\gamma_j\in \realR$ and 
\beq
	0<\tau_1<\tau_2<\cdots< \tau_m.
\eeq
Define, for $x_1, \cdots, x_m \in \R$, 
\begin{equation}
\label{eq:def_mpdist}
\begin{split}
	&\FS\left(x_1,\cdots,x_m;\mathrm{p}_1,\cdots,\mathrm{p}_m\right)
	=	\oint\cdots\oint \ccc(\bfz) \gdetlm(\bfz) \ddbar{z_m}\cdots\ddbar{z_1}
\end{split}
\end{equation}
where $\bfz=(z_1, \cdots, z_m)$ and the contours are nested circles in the complex plane satisfying $0<|z_m|<\cdots<|z_1|<1$. 
Set $\bfx=(x_1, \cdots, x_m)$, $\bftau=(\tau_1, \cdots, \tau_m)$, and $\bfgamma=(\gamma_1, \cdots, \gamma_m)$.
The function $\ccc(\bfz)=\ccc(\bfz; \bfx, \bftau)$ is defined by~\eqref{eq:costcdefaq} and it depends on $\bfx$ and $\bftau$ but not on $\bfgamma$. 
The function $\gdetlm(\bfz)=\gdetlm(\bfz; \bfx, \bftau, \bfgamma)$ depends on all $\bfx$, $\bftau$, and $\bfgamma$, and it is given by the Fredholm determinant, $\gdetlm(\bfz) =\det(\id-\lmKL\lmKR)$ defined  in~\eqref{eq:def_gdetlm}. 
\end{defn}

The functions in the above definition satisfy the following properties. 
The proofs of (P1), (P3), and (P4) are scattered in this section while (P2) is proved later in Lemma~\ref{lem:Dgetlmmdcov} . 
\begin{enumerate}[(P1)]
\item For each $i$, $\ccc(\bfz)$ is a meromorphic function of $z_i$ in the disk $|z_i|< 1$. 
It has simple poles at $z_i=z_{i+1}$ for $i=1, \cdots, m-1$. 
\item For each $i$, $\gdetlm(\bfz)$ is analytic in the  \rev{punctured} disk $0<|z_i|<1$. 
\item For each $i$, $\gdetlm(\bfz)$ does not change if we replace $\gamma_i$ by $\gamma_i+1$. Therefore, $\FS$ is periodic, with period $1$, in the parameter $\gamma_i$ for each $i$.
\item If $\tau_i=\tau_{i+1}$, the function $\FS$ is still well-defined for $x_i< x_{i+1}$. 
\end{enumerate}

\begin{rmk}
\Cb{It is not easy to check directly from the formula that $\FS$ defines a joint distribution function. 
Nonetheless, we may check them indirectly. 
From the fact that $\FS$ is a limit of a sequence of joint distribution functions, $0\le \FS\le 1$
and $\FS$ is a non-decreasing function of $x_k$ for each  $k$. 
It also follows from the fact that the joint distribution can be majorized by a marginal distribution, $\FS$ converges to $0$ if any coordinate $x_k\to -\infty$ since it was shown in (4.10) of \cite{Baik-Liu16} that the $m=1$ case is indeed a distribution function; see Lemma~\ref{lm:Flimni} below. 
The most difficult property to prove is the consistency which $\FS$ should satisfy as a coordinate $x_k\to +\infty$. We prove this property in Section~\ref{sec:consistency} by finding a probabilistic interpretation of the formula of $\FS$ when the $z_i$-contours are not nested; see Theorem~\ref{thm:pofEpm} and Proposition~\ref{prop:consy}.}
\end{rmk}


\subsubsection{Definition of $\ccc(\bfz)$} \label{sec:cccsz}

Let $\log z$ be the principal branch of the logarithm function with cut $\realR_{\le 0}$. 
Let $\polylog_s(z)$ be the polylogarithm function defined by 
\beq
	\polylog_s(z)= \sum_{k=1}^\infty \frac{z^k}{k^s} \quad \text{for $|z|<1$ and $s\in \complexC$.}
\eeq  
It has an analytic continuation 
using the formula
\beq
	\polylog_s(z) = \frac{z}{\Gamma(s)} \int_0^\infty \frac{x^{s-1}}{e^x-z} \dd x \quad 
	\text{for $z\in \complexC\setminus [1,\infty)$ if $\Re(s)>0$.}
\eeq
Set 
\beq \label{eq:A1andA2}
	A_1(z)= -\frac1{\sqrt{2\pi}} \polylog_{3/2}(z), \qquad A_2(z)= -\frac1{\sqrt{2\pi}} \polylog_{5/2}(z).
\eeq
For $0<|z|,|z'|<1$, set
\begin{equation} \label{eq:aux_2017_03_26_10}
	\B(z, z')= \frac{zz'}{2} \iint\frac{\eta\xi \log(-\xi+\eta)}{(e^{-\xi^2/2}-z)(e^{-\eta^2/2}-z')} \ddbarr{\xi}\ddbarr{\eta}
	=\frac{1}{4\pi}\sum_{k,k'\ge 1}\frac{z^k(z')^{k'}}{(k+k')\sqrt{kk'}}
\end{equation}
where the integral contours are the vertical lines $\Re\xi=\mathfrak{a}$ and $\Re\eta=\mathfrak{b}$ with constants $\mathfrak{a}$ and $\mathfrak{b}$ satisfying $-\sqrt{-\log|z|}<\mathfrak{a}<0<\mathfrak{b}<\sqrt{-\log|z'|}$. 
The equality of the double integral and the series is easy to check (see (9.27)--(9.30) in \cite{Baik-Liu16} for a similar calculation.) 
Note that $B(z,z')=B(z',z)$. When $z=z'$, we can also check that  
\beq \label{eq:Bdenf}
	B(z):=B(z, z)=\frac{1}{4\pi}\int_0^z\frac{\left(\polylog_{1/2}(y)\right)^2}{y}\dd y. 
\eeq

\begin{defn}
Define 
\begin{equation}
\label{eq:costcdefaq}
\begin{split}
	\ccc(\bfz)
	:= \left[ \prod_{\ell=1}^{m} \frac{z_\ell}{z_\ell- z_{\ell+1}} \right] 
	\left[ \prod_{\ell=1}^{m} \frac{e^{x_\ell A_1(z_{\ell})+ \tau_\ell A_2(z_\ell)}}{e^{x_\ell A_1(z_{\ell+1}) + \tau_\ell A_2(z_{\ell+1})}}
	e^{2B(z_{\ell})-2B(z_{\ell+1},z_\ell)} \right] 
\end{split}
\end{equation}
where we set $z_{m+1}:=0$. 
\end{defn}

Since $A_1, A_2, B$ are analytic inside the unit circle, it is clear from the definition that $\ccc(\bfz)$ satisfies the property (P1) in Subsubsection~\ref{sec:defFS}.  

\subsubsection{Definition of $\gdetlm(\bfz)$}

\begin{figure}
	\centering
	\includegraphics[scale=0.45]{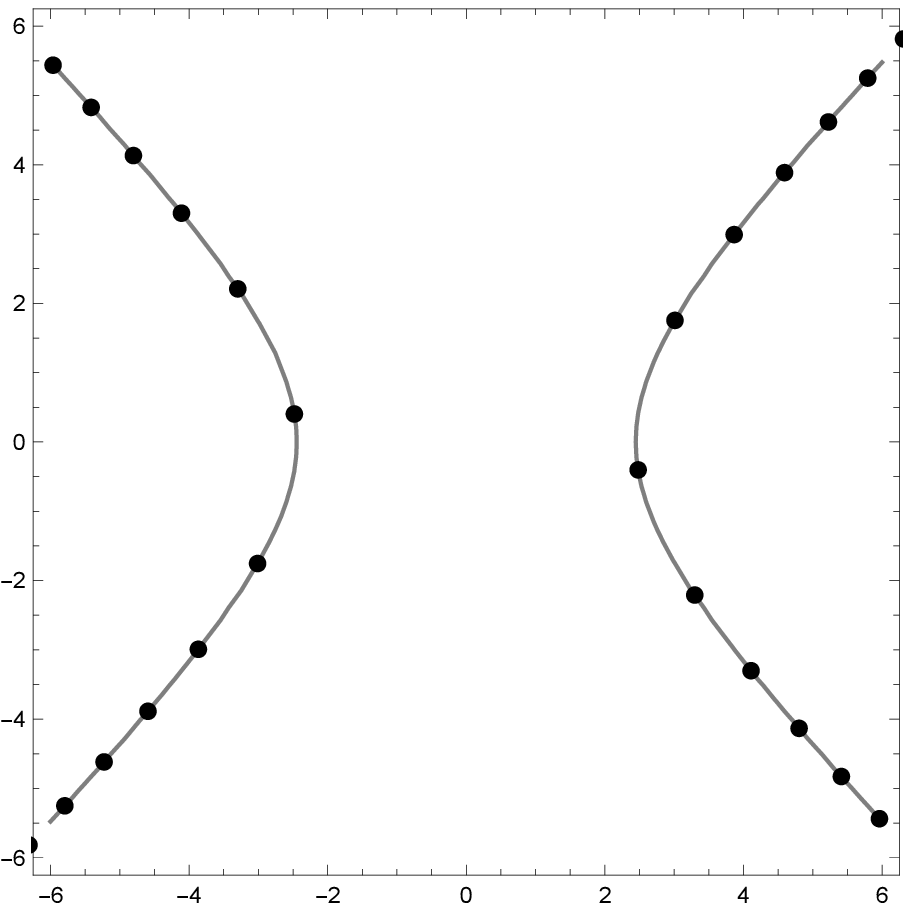} \quad 
	\includegraphics[scale=0.45]{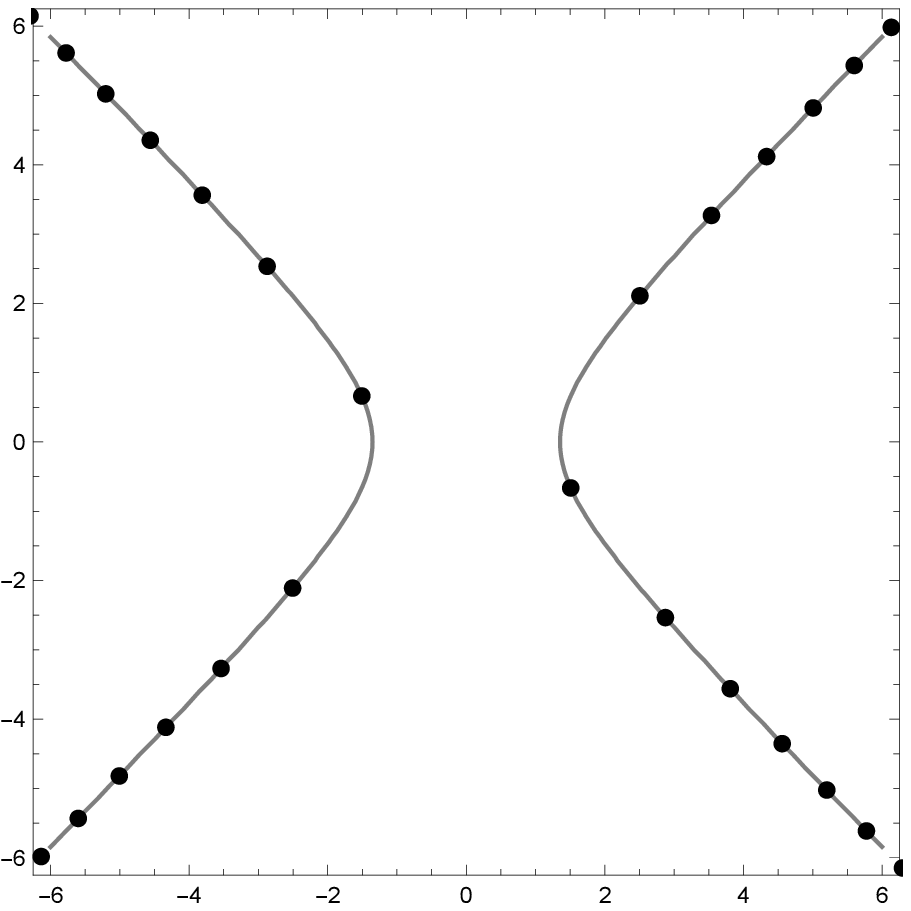} \quad 
	\includegraphics[scale=0.45]{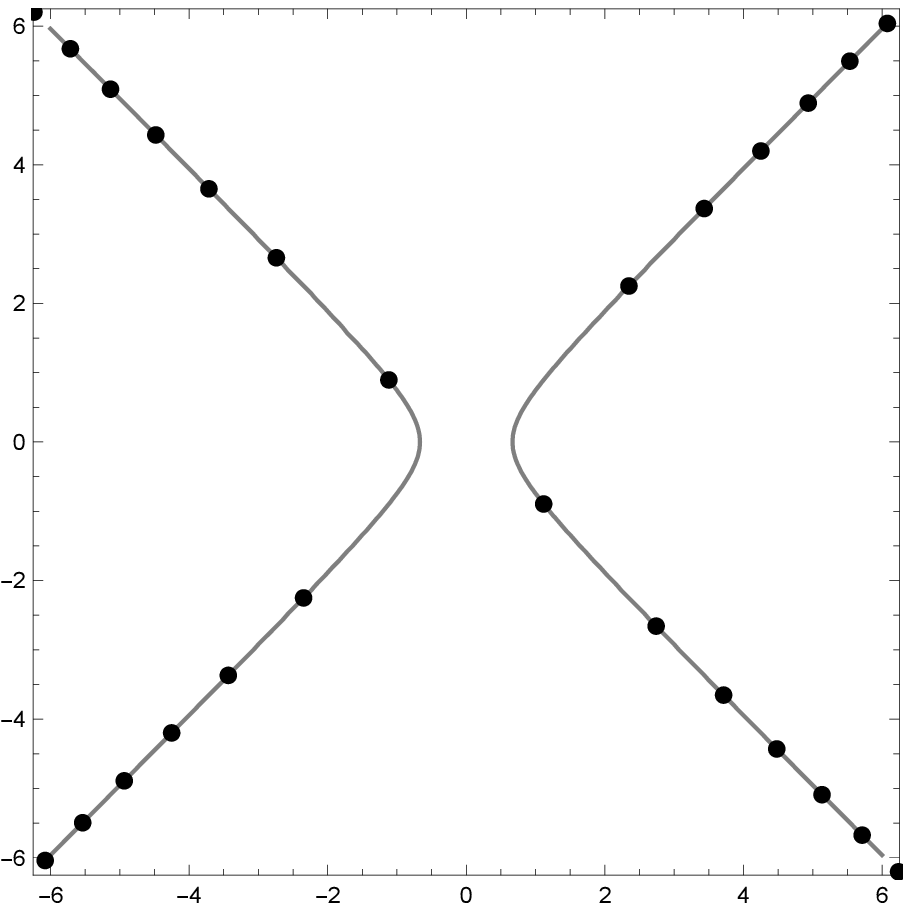}
	\caption{The pictures represent the roots of the equation $e^{-\zeta^2/2}=z$ (dots) and  the contours 
	$\Re(\zeta^2)=-2\log |z|$ (solid curves) for $z= 0.05 e^{\ii}, 0.4 e^{\ii}, 0.8 e^{\ii}$, from the left to the  right.}
\label{fig:Scontour}
\end{figure}

The function $\gdetlm(\bfz)$ is given by a Fredholm determinant. 
Before we describe the operator and the space, we first introduce a few functions.

For $|z|<1$, define the function 
\begin{equation}\label{eq:def_h_R}
	\hftn(\zeta,z)  =		
	-\frac{1}{\sqrt{2\pi}} \int_{-\infty}^{\zeta} \polylog_{1/2} \big(z e^{(\zeta^2-y^2)/2}\big) \dd y
\qquad  \text{for $\Re(\zeta) < 0$}		
\end{equation} 
and 
\begin{equation} \label{eq:def_h_L}
	\hftn(\zeta,z)  =
	-\frac{1}{\sqrt{2\pi}} \int_{-\infty}^{-\zeta} \polylog_{1/2} \big(z e^{(\zeta^2-y^2)/2}\big) \dd y
	\qquad  \text{for $\Re(\zeta) > 0$.}		
\end{equation} 
The integration contour lies in the half-plane $\Re(y)<0$, and is given by the union of the interval $(-\infty, \Re(\pm\zeta)]$ on the real axis and the line segment from $\Re(\pm\zeta)$ to $\pm \zeta$.  
Since $|z e^{(\zeta^2-y^2)/2}| <1$ on the integration contour, $\polylog_{1/2}( z e^{(\zeta^2-y^2)/2} )$ is well defined. 
Thus, we find that the integrals are well defined using $\polylog_{1/2}(\omega) \sim \omega$ as $\omega \to 0$.

Observe the symmetry,  
\begin{equation}
\label{eq:def_h_RminwithL}
	\hftn(\zeta, z) =\hftn(-\zeta,z) \quad \text{for $\Re(\zeta)<  0$.}
\end{equation} 
We also have 
\begin{equation}
\label{eq:aux_2017_09_16_01}
	\hftn(\zeta,z)= \int_{-\ii \infty}^{\ii \infty} 
	\frac{\log (1-ze^{\omega^2/2})}{\omega-\zeta}\frac{\dd\omega}{2\pi\ii} \quad \text{for $\Re(\zeta)<  0$.}
\end{equation}
This identity can be obtained by the power series expansion and using the fact that
$\frac1{\sqrt{2\pi}} \int_{-\infty}^u e^{-\omega^2/2} \dd \omega = \int_{-\ii\infty}^{\ii\infty} \frac{e^{(-u^2+\omega^2)/2}}{\omega-u} \frac{\dd\omega}{2\pi \ii}$ for $u$ with $\arg(u)\in (3\pi /4, 5\pi/4)$; see (4.8) of \cite{Baik-Liu16}.
From~\eqref{eq:def_h_RminwithL} and~\eqref{eq:aux_2017_09_16_01}, we find that
\begin{equation}
\label{eq:hftn_estimate}
	\hftn(\zeta,z)= O(\zeta^{-1})\quad \text{as $\zeta\to\infty$ in the region $\left|\arg(\zeta)\pm\frac{\pi}{2}\right|>\epsilon$}
\end{equation}
for any fixed $z$ satisfying $|z|<1$.

Let $\bfx$, $\bftau$, and $\bfgamma$ be the parameters in Definition~\ref{def:FSdefn}.
We set 
\begin{equation} \label{eq:def_sfs}
 	 \sfs_i(\zeta) :=\begin{dcases}
	 e^{-\frac13 (\tau_i-\tau_{i-1}) \zeta^3 + \frac12 (\gamma_i-\gamma_{i-1}) \zeta^2 + (x_i-x_{i-1}) \zeta}
	 \quad &\text{for $\Re(\zeta)<0$} \\
	 e^{\frac13 (\tau_i-\tau_{i-1}) \zeta^3 - \frac12 (\gamma_i-\gamma_{i-1}) \zeta^2 - (x_i-x_{i-1}) \zeta}
	 \quad & \text{for $\Re(\zeta)>0$}
 	 \end{dcases}
\end{equation}
for $i=1, \cdots, m$, where we set $\tau_0=\gamma_0=x_0=0$. 

\bigskip

Now we describe the space and the operators. 
For a non-zero complex number $z$, consider the roots $\zeta$ of the equation 
$e^{-\zeta^2/2}=z$. 
The roots are on the contour $\Re(\zeta^2)=-2\log|z|$. 
It is easy to check that if $0<|z|<1$, the contour $\Re(\zeta^2)=-2\log|z|$ consists of two disjoint components, one in $\Re(\zeta)>0$ and the other in $\Re(\zeta)<0$. 
See Figure~\ref{fig:Scontour}. 
The asymptotes of the contours are the straight lines of slope $\pm 1$.
For $0<|z|<1$, we define the discrete sets 
\beq \label{eq:LRlmtdefs}
\bes
	\inodesL_{z} &:= \{ \zeta\in \complexC : e^{-\zeta^2/2}=z\} \cap\{\Re(\zeta)<0\}, \\
	\inodesR_{z} &:= \{ \zeta\in \complexC : e^{-\zeta^2/2}=z\} \cap\{\Re(\zeta)>0\}. 
\end{split}
\eeq

For distinct complex numbers $z_1, \cdots, z_m$ satisfying $0<|z_i|<1$, define the sets 
\begin{equation} \label{eq:SSSfor2}
	\SSSL:=\inodesL_{z_1}\cup \inodesR_{z_2}\cup\inodesL_{z_3}\cup\cdots\cup
	\begin{cases}
\inodesR_{z_m} &  \text{if $m$ is even,} \\ 
\inodesL_{z_m} & \text{if $m$ is odd,} 
\end{cases}
\end{equation}
and
\begin{equation} \label{eq:SSSfor1}
	\SSSR:=\inodesR_{z_1}\cup \inodesL_{z_2}\cup\inodesR_{z_3}\cup\cdots\cup
	\begin{cases}
\inodesL_{z_m} & \text{if $m$ is even,} \\
\inodesR_{z_m}& \text{if $m$ is odd.} 
\end{cases}
\end{equation}
See Figure~\ref{fig:rootslrall}. 
Now we define two operators 
\begin{equation}
	\lmKL:\ell^2(\SSSR)\to\ell^2(\SSSL),\qquad \lmKR:\ell^2(\SSSL)\to\ell^2(\SSSR)
\end{equation}
by kernels as follows. 
If 
\beq
	\text{$\zeta\in(\inodesL_{z_i}\cup \inodesR_{z_i}) \cap \SSSL$ and $\zeta'\in(\inodesL_{z_j}\cup \inodesR_{z_j})\cap\SSSR$}
\eeq
for some $i,j\in\{1, \cdots, m\}$, then we set 
\begin{equation} \label{eq:limitkeL}
\begin{split}
	\lmKL(\zeta,\zeta')&=(\delta_i(j)+\delta_i(j+(-1)^i)) 
	 \frac{ \sfs_i(\zeta) e^{2\hftn(\zeta, z_i)-\hftn(\zeta, z_{i-(-1)^{i}}) -\hftn(\zeta', z_{j-(-1)^{j}})}}{\zeta(\zeta-\zeta')}\QWL(j) .
\end{split}
\end{equation}
Similarly, if 
\beq
	\text{$\zeta\in(\inodesL_{z_i}\cup \inodesR_{z_i})\cap \SSSR$ and $\zeta'\in(\inodesL_{z_j}\cup \inodesR_{z_j})\cap \SSSL$}
\eeq
for some $i,j\in\{1, \cdots, m\}$, then we set 
\begin{equation} \label{eq:limitkeR}
\begin{split}
	\lmKR(\zeta,\zeta')&=(\delta_i(j)+\delta_i(j-(-1)^i)) 
	\frac{\sfs_i(\zeta) e^{2\hftn(\zeta, z_i)-\hftn(\zeta, z_{i+(-1)^{i}}) -\hftn(\zeta', z_{j+(-1)^{j}})}}{\zeta(\zeta-\zeta')} \QWR(j).
\end{split}
\end{equation}
Here \rev{the delta function $\delta_i(k)=1$ if $k=i$ or $0$ otherwise. }  \rev{We also} set $z_0=z_{m+1}=0$ so that
\beq
	e^{-\hftn(\zeta, z_{0})}=e^{-\hftn(\zeta, z_{m+1})}= 1. 
\eeq
 \rev{And the functions $\QWL(j)$, $\QWR(j)$ are defined by} 
\beq \label{eq:QWbelow}
	\text{$\QWL(j)= 1- \frac{z_{j-(-1)^j}}{z_j}$ and $\QWR(j) = 1-\frac{z_{j+(-1)^j}}{z_j}$.}
\eeq

\begin{figure}
	\centering
	\includegraphics[scale=0.45]{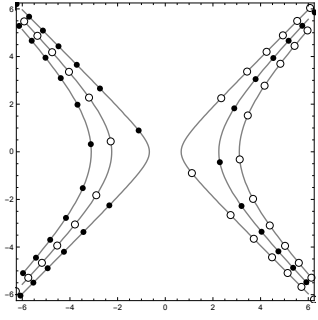}
	\caption{Example of $\SSSL$ (block dots) and $\SSSR$ (white dots) when $m=3$. The level sets are shown for visual convenience.}
	\label{fig:rootslrall}
\end{figure}


\begin{defn}
Define
\begin{equation}
\label{eq:def_gdetlm}
	\gdetlm(\bfz):=\det\left(\id-\lmKL\lmKR\right)
\end{equation}
for $\bfz=(z_1, \cdots, z_m)$ where $0<|z_i|<1$ and $z_i$ are distinct. 
\end{defn}

{In this definition, we temporarily assumed that $z_i$ are distinct in order to ensure that the term $\zeta-\zeta'$ in the denominators in~\eqref{eq:limitkeL} and~\eqref{eq:limitkeR} does not vanish. However, as we stated in (P2) in Subsubsection~\ref{sec:defFS}, $\gdetlm(\bfz)$ is still well-defined when $z_i$ are equal. See Lemma~\ref{lem:Dgetlmmdcov}.}

The definition of $\inodesL_{z}$ and $\inodesR_{z}$ implies that $|\arg(\zeta)|\to  3\pi /4$ as $|\zeta|\to \infty$ along $\zeta\in \inodesL_z$ and $|\arg(\zeta)|\to  \pi /4$ as $|\zeta|\to \infty$ along $\zeta\in\inodesR_z$. 
Hence, due to the cubic term $\zeta^3$ in~\eqref{eq:def_sfs}, $\sfs_i(\zeta)\to 0$ super exponentially as $|\zeta|\to \infty$ on the set $\inodesL_{z}\cup \inodesR_{z}$ if $\tau_1<\cdots<\tau_m$. 
Hence, using the property~\eqref{eq:hftn_estimate} of $\hftn$, we see that the kernels decay super-exponentially fast as $|\zeta|,|\zeta'|\to\infty$ on the spaces. 
Therefore, the Fredholm determinant is well defined if $\tau_1<\cdots<\tau_m$.

We now check the property (P4).
If $\tau_i=\tau_{i+1}$, the exponent of $\sfs_i$ has no cubic term $\zeta^3$. 
The quadratic term contributes to $O(1)$ since $|e^{-\zeta^2/2}|=|z_i|$ for $\zeta\in \inodesL_{z_i}\cup \inodesR_{z_i}$, and hence $|e^{c\zeta^2}|=O(1)$. 
On the other hand, the linear term in the exponent of $\sfs_i$ has a negative real part if $x_i< x_{i+1}$. 
Hence, if $\tau_i=\tau_{i+1}$ and $x_i< x_{i+1}$, then 
$|\sfs_i(\zeta)|\to 0$ exponentially as $\zeta\to \infty$ along $\zeta\in \SSSL\cup\SSSR$ and hence the kernel decays exponentially fast as $|\zeta|,|\zeta'|\to\infty$ on the spaces. 
Therefore, the Fredholm determinant is still well defined if $\tau_i=\tau_{i+1}$ and $x_i< x_{i+1}$. This proves (P4).

\subsection{Matrix kernel formula of $\lmKL$ and $\lmKR$} \label{sec:matrixkl}

Due to the delta functions, $\lmKL(\zeta, \zeta')\neq 0$ only when 
\beq
	\text{$\zeta\in \inodesL_{z_{2\ell-1}}\cup \inodesR_{z_{2\ell}}$ and $\zeta'\in \inodesR_{z_{2\ell-1}}\cup \inodesL_{z_{2\ell}}$} 	
\eeq 
for some integer $\ell$, and similarly $\lmKR(\zeta, \zeta')\neq 0$ only when 
\beq
	\text{$\zeta\in \inodesL_{z_{2\ell}}\cup \inodesR_{z_{2\ell+1}}$ and $\zeta'\in \inodesR_{z_{2\ell}}\cup \inodesL_{z_{2\ell+1}}$} 	
\eeq 
for some integer $\ell$.
Thus, if we represent the kernels as $m\times m$ matrix kernels, then they have $2\times 2$ block structures. 

For example, consider the case when $m=5$. 
Let us use $\xi_i$ and $\eta_i$ to represent variables in $\inodesL_{z_i}$ and $\inodesR_{z_i}$, respectively:
\beq
	\xi_i\in \inodesL_{z_i}, \qquad \eta_i\in \inodesR_{z_i}. 
\eeq 
The matrix kernels are given by 
\beq
	\lmKL= \begin{bmatrix} 
	\lmk(\xi_1, \eta_1) & \lmk(\xi_1, \xi_2) \\
	\lmk(\eta_2, \eta_1) & \lmk(\eta_2, \xi_2) \\
	&& \lmk(\xi_3, \eta_3) & \lmk(\xi_3, \xi_4) \\ 
	&& \lmk(\eta_4, \eta_3) & \lmk(\eta_4, \xi_4) \\
	&&&& \lmk(\xi_5, \eta_5)
	\end{bmatrix}
\eeq
and
\beq
	\lmKR= \begin{bmatrix} 
	\lmk(\eta_1, \xi_1) \\
	& \lmk(\xi_2, \eta_2) & \lmk(\xi_2, \xi_3) \\
	& \lmk(\eta_3, \eta_2) & \lmk(\eta_3, \xi_3) \\
	&&& \lmk(\xi_4, \eta_4) & \lmk(\xi_4, \xi_5) \\ 
	&&& \lmk(\eta_5, \eta_4) & \lmk(\eta_5, \xi_5) 
	\end{bmatrix}
\eeq
where the empty entries are zeros and the function $\lmk$ is given in the below.
When $m$ is odd, the structure is similar. 
On the other hand, when $m$ is even, 
$\lmKL$ consists only of $2\times 2$ blocks and
$\lmKR$ contains an additional non-zero $1\times 1$ block at the bottom right corner. 

We now define $\lmk$. For $1\le i\le m-1$, writing 
\beq
	\xi=\xi_i, \quad \eta=\eta_i,  
	\quad \xi'=\xi_{i+1}, \quad \eta'=\eta_{i+1}, 
\eeq
we define
\beq \label{eq:kblockm}
\bes
	\begin{bmatrix} 
	 \lmk(\xi, \eta) & \lmk(\xi, \xi') \\
	\lmk(\eta', \eta) & \lmk(\eta', \xi') 
	\end{bmatrix}
	= &\begin{bmatrix} 
	\sfs_i(\xi)  \\ 
	& 
	\sfs_{i+1}(\eta')
	\end{bmatrix}
	\begin{bmatrix} 
	\frac{e^{2\hftn(\xi,z_i)}}{\xi e^{\hftn(\xi, z_{i+1})}}      \\ 
	& \frac{e^{2\hftn(\eta',z_{i+1})}}{\eta' e^{\hftn(\eta', z_i)} } 
	\end{bmatrix}
	\begin{bmatrix} \frac1{\xi-\eta} & \frac1{\xi-\xi'} \\ \frac1{\eta'-\eta} & \frac1{\eta'-\xi'} 
	\end{bmatrix} \\
	&\times 
	\begin{bmatrix} \frac1{e^{\hftn(\eta, z_{i+1})}}   \\ 
	&  \frac1{e^{\hftn(\xi', z_i)}} 
	\end{bmatrix}
	\begin{bmatrix} 1-\frac{z_{i+1}}{z_i} \\ 
	&  1-\frac{z_{i}}{z_{i+1}}  
	\end{bmatrix}.
\end{split}
\eeq
\Cb{This means that 
\beq
	\lmk(\xi_i, \eta_i)= \sfs_i(\xi_i) \frac{e^{2\hftn(\xi_i,z_i)}}{\xi_i e^{\hftn(\xi_i, z_{i+1})}} \frac1{\xi_i-\eta_i} \frac1{e^{\hftn(\eta_i, z_{i+1})}} \left( 1-\frac{z_{i+1}}{z_i} \right),
\eeq
\beq
	\lmk(\xi_i, \xi_{i+1})= \sfs_i(\xi_i) \frac{e^{2\hftn(\xi_i,z_i)}}{\xi_i e^{\hftn(\xi_i, z_{i+1})}} \frac1{\xi_i-\xi_{i+1}} \frac1{e^{\hftn(\xi_{i+1}, z_i)}}  \left( 1-\frac{z_{i}}{z_{i+1}} \right),
\eeq
and so on.}
The term $\lmk(\xi_m, \eta_m)$ is defined by the $(1,1)$ entry of~\eqref{eq:kblockm} with $i=m$ where we set $z_{m+1}=0$.
The term $\lmk(\eta_1, \xi_1)$ is defined by the $(2,2)$ entry of~\eqref{eq:kblockm} with $i=0$ where we set $z_{0}=0$. 


\subsection{Series formulas for $\gdetlm(\bfz)$} \label{sec:altforjtlim}

We present two series formulas for the function $\gdetlm(\bfz)$.
The first one~\eqref{eq:Frdse1} is the series expansion of Fredholm determinant using the block structure of the matrix kernel.
The second formula~\eqref{eq:Frdse2} is obtained after evaluating the finite determinants in~\eqref{eq:Frdse1} explicitly.

   
To simplify formulas, we introduce the following notations. 
\begin{defn}[Notational conventions] \label{def:notc}
For complex vectors $\WWlm=(\wsw_1, \cdots, \wsw_n)$ and $\WWlm'=(\wsw'_1,\cdots,\wsw'_{n'})$, we set  
\begin{equation} \label{eq:notationddw}
	\Delta(\WWlm)=\prod_{i<j}(\wsw_j-\wsw_i) = \det \left[ \wsw_i^{j-1}\right], \qquad 
	\Delta(\WWlm;\WWlm')=\prod_{\substack{1\le i\le n\\ 1\le i'\le n'}}(\wsw_i-\wsw'_{i'}).
\end{equation}
For a function $h$ of single variable, we write 
\begin{equation}
\label{eq:aux_2016_12_31_01}
	h(\WWlm) = \prod_{i=1}^n h(\wsw_i) .
\end{equation}
We also use the notations
\begin{equation}
	\Delta(S;S')=\prod_{\substack{s\in S\\ s'\in S'}}(s-s'), 
	\qquad f(S)= \prod_{s\in S} f(s)
\end{equation}
for finite sets $S$ and $S'$. 
\end{defn}

\bigskip

The next lemma follows from a general result whose proof is given in Subsection~\ref{sec:eqFrsr} below.

\begin{lm}[Series formulas for $\gdetlm(\bbz)$] \label{lem:Freserfola}
We have 
\beq \label{eq:Fdmddser}
	\gdetlm(\bfz)= \sum_{\bn\in (\intZ_{\ge 0})^m} \frac{1}{(\bn !)^2}  \gdetlm_\bn(\bfz)
\eeq
with $\bn != \prod_{\ell=1}^m n_\ell!$ for $\bn=(n_1, \cdots, n_m)$, where
$\gdetlm_{\bn}(\bfz)$ can be expressed as the following two ways.
\begin{enumerate}[(i)]
\item We have, for $\bn=(n_1, \cdots, n_m)$, 
\begin{equation} \label{eq:Frdse1}
	\gdetlm_{\bn}(\bfz)
	=(-1)^{|\bn|}\sum_{\substack{\UUlm^{(\ell)}\in (\inodesL_{z_\ell})^{n_\ell} \\ \VVlm^{(\ell)}\in (\inodesR_{z_\ell})^{n_\ell} \\ \ell=1,\cdots,m}}
	\det\left[\lmKL(\sw_i, \sw'_j)\right]_{i,j=1}^{|\bn|}
	\det\left[\lmKR(\sw'_i, \sw_j)\right]_{i,j=1}^{|\bn|}
\end{equation}
where 
$\bUUlm=(\UUlm^{(1)}, \cdots, \UUlm^{(m)})$, $\bVVlm= (\VVlm^{(1)}, \cdots, \VVlm^{(m)})$ with $\UUlm^{(\ell)}=(\su_1^{(\ell)}, \cdots, \su_{n_\ell}^{(\ell)})$, $\VVlm^{(\ell)}=(\sv_1^{(\ell)}, \cdots, \sv_{n_\ell}^{(\ell)})$, and 
\begin{equation}
\label{eq:Frdse2}
	\sw_i=\begin{dcases}
	\su_k^{(\ell)}& \text{if $i=n_1 +\cdots+n_{\ell-1}+k$ for some $k\le n_\ell$ with odd integer $\ell$,}\\
\sv_k^{(\ell)}& \text{if $i=n_1 +\cdots+n_{\ell-1}+k$ for some $k\le n_\ell$ with even integer $\ell$,}
\end{dcases}
\end{equation}
and
\begin{equation}
	\sw'_i=\begin{dcases}
\sv_k^{(\ell)}& \text{if $i=n_1+\cdots+n_{\ell-1}+k$ for some $k\le n_\ell$ with odd integer $\ell$,}\\
\su_k^{(\ell)}& \text{if $i=n_1+\cdots+n_{\ell-1}+k$ for some $k\le n_\ell$ with even integer $\ell$.}
\end{dcases}
\end{equation}
Here, we set $n_0=0$.

\item
We also have 
\begin{equation}
\label{eq:aux_2017_03_29_14}
	\gdetlm_\bn(\bfz)  
	=\sum_{\substack{\UUlm^{(\ell)}\in\left(\inodesL_{z_\ell}\right)^{n_\ell} \\ \VVlm^{(\ell)}\in\left(\inodesR_{z_\ell}\right)^{n_\ell} \\ \ell=1,\cdots,m}}
	\gdetlmdd_{\bn, \bfz} (\bUUlm, \bVVlm)
\end{equation}
with 
\begin{equation}
\label{eq:aux_2017_07_01_05}
\begin{split}
	&\gdetlmdd_{\bn, \bfz} (\bUUlm, \bVVlm)
	:=
	\left[ \prod_{\ell=1}^m \frac{\Delta(\UUlm^{(\ell)})^2\Delta(\VVlm^{(\ell)})^2}{\Delta(\UUlm^{(\ell)};\VVlm^{(\ell)})^2}
	\fslm_\ell(\UUlm^{(\ell)}) 
	\fslm_\ell(\VVlm^{(\ell)})\right]\\
	&\times\left[ \prod_{\ell=2}^m  \frac{\Delta(\UUlm^{(\ell)};\VVlm^{(\ell-1)})\Delta(\VVlm^{(\ell)};\UUlm^{(\ell-1)}) e^{-\hftn(\VVlm^{(\ell)}, z_{\ell-1}) - \hftn(\VVlm^{(\ell-1)}, z_{\ell})} }
	{\Delta(\UUlm^{(\ell)};\UUlm^{(\ell-1)})\Delta(\VVlm^{(\ell)};\VVlm^{(\ell-1)}) 
	 e^{\hftn(\UUlm^{(\ell)}, z_{\ell-1}) + \hftn(\UUlm^{(\ell-1)}, z_{\ell})} }
	\left(1- \frac{z_{\ell-1}}{z_{\ell}} \right)^{n_\ell} \left(1-\frac{z_\ell}{z_{\ell-1}} \right)^{n_{\ell-1}} \right]
\end{split}
\end{equation}
where
\begin{equation}
\label{eq:aux_2017_04_05_02}
\begin{split}
	\fslm_\ell(\zeta)&:=\frac{1}{\zeta} 
	\sfs_\ell(\zeta) e^{2\hftn(\zeta, z_\ell)}.
\end{split}
\end{equation}
Recall~\eqref{eq:def_h_R},~\eqref{eq:def_h_L}, and~\eqref{eq:def_sfs} 
for the definition of $\hftn$ and $\sfs_j$.
\end{enumerate}
\end{lm}

The property (P3)  in Subsubsection~\ref{sec:defFS} follows easily from~\eqref{eq:aux_2017_07_01_05}.
Note that $\gamma_i$ only appears in the factor $\fslm_\ell(\UUlm^{(\ell)})\fslm_\ell(\VVlm^{(\ell)})$ for $\ell=i$ or $i+1$. If we replace $\gamma_i$ by $\gamma_i+1$, then $\sfs_i(\zeta)$ and $\sfs_{i+1}(\zeta)$ are changed by $z_i^{-1}\sfs_i(\zeta)$ and $z_{i+1}\sfs_{i+1}(\zeta)$ if $\Re(\zeta)<0$, or $z_i\sfs_i(\zeta)$ and $z_{i+1}^{-1}\sfs_{i+1}(\zeta)$ if $\Re(\zeta)>0$. But $\UUlm^{(\ell)}$ has the same number of components as $\VVlm^{(\ell)}$ for each $\ell$. Therefore $\fslm_\ell(\UUlm^{(\ell)})\fslm_\ell(\VVlm^{(\ell)})$ does not change.
We can also check (P3) from the original Fredholm determinant formula. 

The analyticity property (P2) is proved in Lemma~\ref{lem:Dgetlmmdcov} later using the series formula.

\section{Joint distribution function for general initial condition}
\label{sec:tran}

We obtain the limit theorem of the previous section from a finite-time formula of the joint distribution. 
In this section, we describe a formula of the finite-time joint distribution for an arbitrary initial condition. 
We simplify the formula further in the next section for the case of the periodic step initial condition. 

We state the results in terms of particle locations instead of the height function used in the previous section. 
It is easy to convert one to another; see~\eqref{eq:xandhet2}. 
The particle locations are denoted by $\xx_i(t)$ where
\beq
	\cdots< \xx_0(t)<\xx_1(t)<\xx_2(t)< \cdots.
\eeq
Due to the periodicity of the system, we have $\xx_i(t)=\xx_{i+nN}(t)-nL$ for all integers $n$. 

The periodic TASEP can be described if we keep track of $N$ consecutive particles, say $\xx_1(t)<\cdots<\xx_N(t)$. 
If we  focus only on these particles, they follow the usual TASEP rules plus the extra condition that $\xx_N(t)<\xx_1(t)+L$ for all  $t$. 
Define the configuration space 
\begin{equation}
	\conf_N(L) = \{(x_1,x_2,\cdots,x_N)\in \intZ^N: x_1 <x_2 <\cdots <x_N <x_1+L  \}.
\end{equation}
We call the process of the $N$ particles \emph{TASEP in $\conf_N(L)$}.
We use the same notations $\xx_i(t)$, $i=1, \cdots, N$, to denote the particle locations in the TASEP in $\conf_N(L)$.
We state the result for the TASEP in $\conf_N(L)$ first and then for the periodic TASEP as a corollary.

For $z\in \complexC$, consider the polynomial of degree $L$ given by 
\begin{equation}
	\label{eq:def_roots}
	q_z(w) = w^N(w+1)^{L-N}-z^L.
\end{equation}
Denote the set of the roots by 
\begin{equation}\label{eq:rootsz}
	\roots_z = \{w\in\complexC : q_z(w)=0\}. 
\end{equation} 
The roots are on the level set $|w^N(w+1)^{L-N}|=|z|^L$. 
It is straightforward to check the following properties of the level set. 
Set
\beq \label{eq:r0defn}
	\rr_0:= \rho^\rho(1-\rho)^{1-\rho}
\eeq
where, as before, $\rho=N/L$.
The level set becomes larger as $|z|$ increases, see Figure~\ref{fig:rootsfinite}. 
If $0<|z|< \rr_0$, the level set consists of two closed contours, one in $\Re(w)<-\rho$ enclosing the point $w=-1$ and the other in $\Re(w)>-\rho$. enclosing the point $w=0$.    
When $|z|=\rr_0$, the level set has a self-intersection at $w=-\rho$. 
If $|z|>\rr_0$, then the level set is a connected closed contour. 
Now consider the set of roots $\roots_z$. Note that if $z\neq 0$, then $-1, 0\notin \roots_z$.  
It is also easy to check that if a non-zero $z$ satisfies $z^L\neq \rr_0^L$, then the roots of $q_z(w)$ are all simple. 
On the other hand, if $z^L=\rr_0^L$, then there is a double root at $w=-\rho$ and the rest $L-2$ roots are simple. 
For the results in this section, we take $z$ to be any non-zero complex number. 
But in the next section, we restrict $0<|z|<\rr_0$.


\begin{figure}
	\centering
	\includegraphics[scale=0.5]{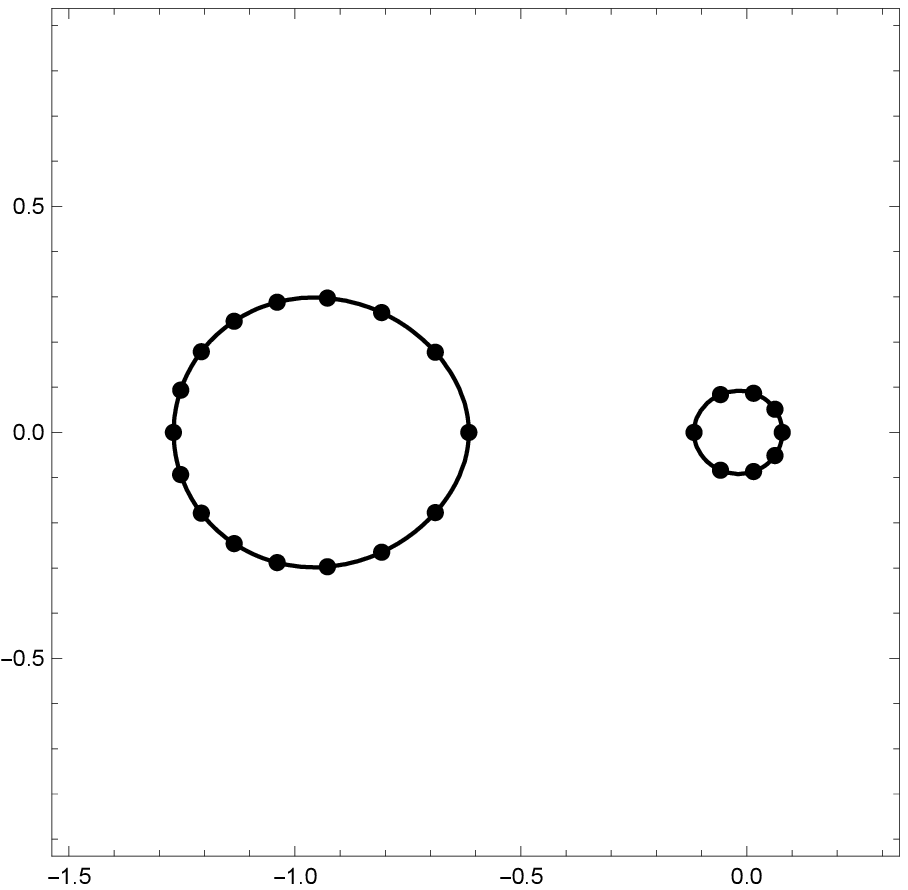}
	\includegraphics[scale=0.5]{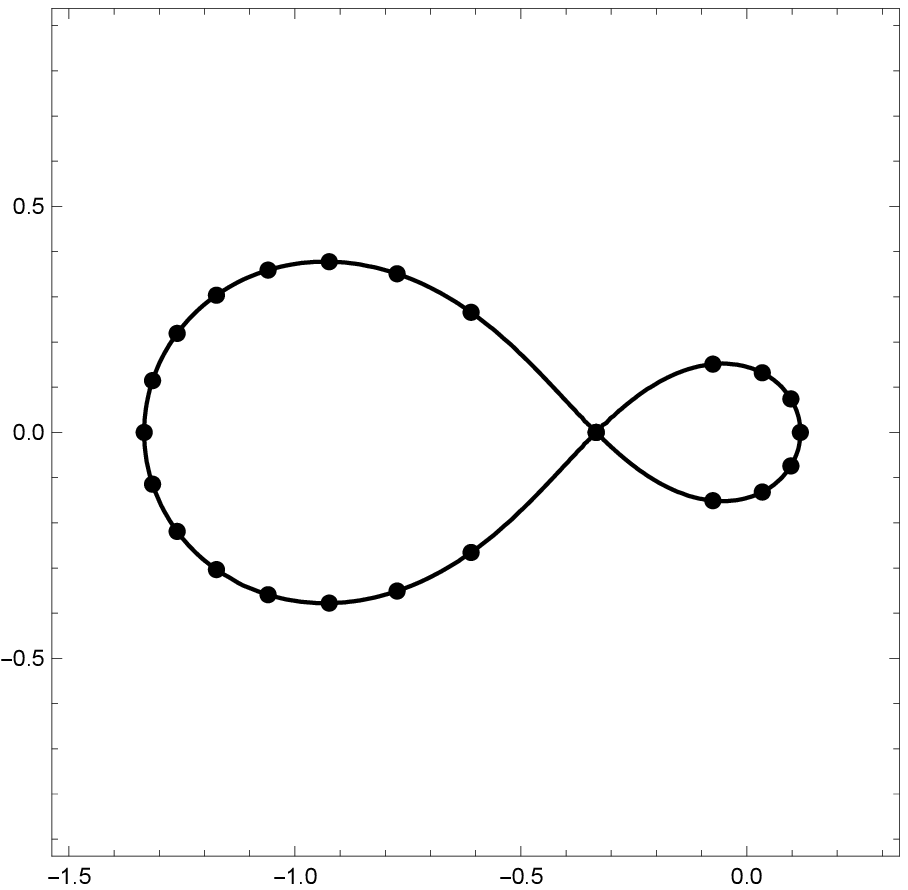}
	\includegraphics[scale=0.5]{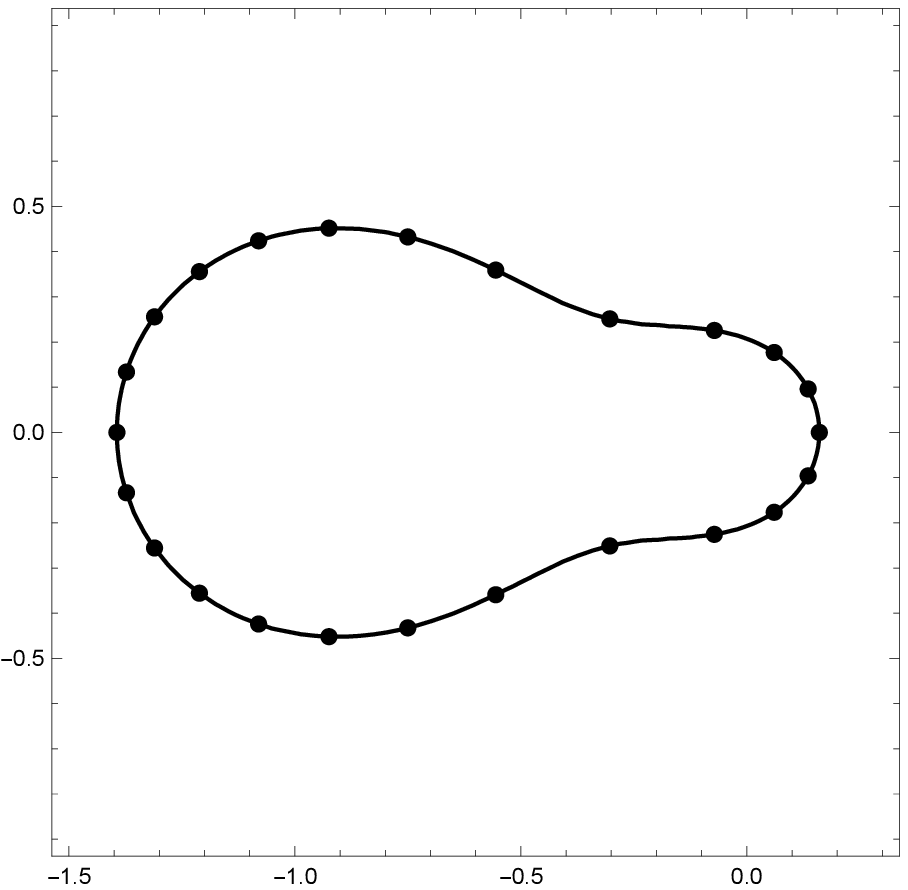}
	\caption{The pictures represent the roots of the equation $w^N(w+1)^{L-N}=z^L$ and the contours $|w^N(w+1)^{L-N}|=|z|^L$ with $N=8, L=24$ for three different values of $z$. 
The value of $|z|$ increases from the left picture to the right picture. 
The middle picture is when $z= \rho^{\rho}(1-\rho)^{1-\rho}$ where $\rho=N/L=1/3$.}
\label{fig:rootsfinite}
\end{figure}

\begin{thm}[Joint distribution of TASEP in $\conf_N(L)$ for general initial condition] \label{prop:multipoint_distribution_origin}
Consider the TASEP in $\conf_N(L)$.
Let $Y=(y_1, \cdots, y_N) \in\conf_N(L)$ and assume that $(\xx_1(0), \cdots, \xx_N(0))=Y$.  
Fix a positive integer $m$.
Let $(k_1, t_1), \cdots, (k_m, t_m)$ be $m$ distinct points in $\{1, \cdots, N\} \times [0, \infty)$. 
Assume that $0\le t_1\le\cdots\le t_m$.
Let $a_i\in \intZ$ for $1\le i\le m$. 
Then 
\begin{equation}\label{eq:multipoint_distribution_origin}
 	\begin{split}
	\prob_Y\left(\xx_{k_1}(t_1)\ge a_1,\cdots, \xx_{k_m}(t_m)\ge a_m\right) 
	=  
	\oint \cdots\oint 
	\constc(\bbz, \bk)  \detv_Y(\bbz, \bk, \ba, \bt) \ddbar{z_m}\cdots \ddbar{z_1}
 	\end{split}
\end{equation}
where the contours are nested circles in the complex plane satisfying $0<|z_m|<\cdots<|z_1|$. 
Here $\bbz=(z_1, \cdots, z_m)$, $\bk=(k_1, \cdots, k_m)$, $\ba=(a_1, \cdots, a_m)$, and $\bt=(t_1, \cdots, t_m)$. 
The functions in the integrand are  
\begin{equation} \label{eq:constcti1e}
\begin{split}
 	\constc (\bbz, \bk)
 	&= (-1)^{(k_m-1)(N+1)}z_1^{(k_1-1)L} 
	\prod_{\ell=2}^m \Bigg[ z_\ell^{(k_\ell-k_{\ell-1})L}\bigg(\bigg(\frac{z_\ell}{z_{\ell-1}}\bigg)^L-1 \bigg)^{N-1}\Bigg] 
\end{split}
\end{equation}
and
\begin{equation}\label{eq:multipoint_distribution_origin12}
 \begin{split}
		\detv_Y(\bbz, \bk, \ba, \bt)
		 =  \det\left[ \sum_{\substack{w_1\in \roots_{z_1} \\ \cdots\\ w_m\in\roots_{z_m}}}  
		\frac{w_1^i \left(w_1+1\right)^{y_i-i}w_m^{-j}}{ \prod_{\ell=2}^m (w_\ell-w_{\ell-1})  } \prod_{\ell=1}^m \fgic_\ell(w_\ell) \right]_{i,j=1}^N
\end{split}
\end{equation}
with
\begin{equation}
	\fgic_\ell(w) = \frac{w(w+1)}{L(w+\rho)}
	\frac{w^{-k_\ell}(w+1)^{-a_\ell+k_\ell}e^{t_\ell w}}{w^{-k_{\ell-1}}(w+1)^{-a_{\ell-1}+k_{\ell-1}}e^{t_{\ell-1}w}}
\end{equation}
where we set $t_0=k_0=a_0=0$. 
\end{thm}

\begin{rmk} \label{rmk:parmt}
The limiting joint distribution $\FS$ in the previous section was not defined {for all parameters:
When $\tau_i=\tau_{i+1}$, we need to put 
the restriction $x_i<x_{i+1}$. See Property (P4)} in Subsubsection~\ref{sec:defFS}.   
The finite-time joint distribution does not require such restrictions. 
The sums in 
the entries of the determinant $\detv_Y(\bbz, \bk, \ba, \bt)$ are over finite sets, and hence there is no issue with the convergence. 
Therefore, the right-hand side of~\eqref{eq:multipoint_distribution_origin} is well-defined for all real numbers  $t_i$ and integers $a_i$ and $k_i$. 
\end{rmk}

\begin{cor}[Joint distribution of periodic TASEP for general initial condition]  \label{cor:pTASEPgenIC}
Consider the periodic TASEP with a general initial condition determined by 
$Y=(y_1, \cdots, y_N) \in\conf_N(L)$ and its periodic translations;  
$\xx_{j+n N}(0)=y_j+n L$ for all $n\in\intZ$ and $j=1,\cdots,N$. 
Then~\eqref{eq:multipoint_distribution_origin} holds for all $k_i\in \Z$ without the restriction that $k_i\in \{1, \cdots, N\}$.
\end{cor}

\begin{proof}
The particles in the periodic TASEP satisfies $\xx_{j+nN}(t)=\xx_{j}(t)+nL$ for every integer $n$. 
Hence if $k_\ell$ is not between $1$ and $N$, we may translate it. 
This amounts to changing $k_\ell$ to $k_\ell+nN$ and $a_\ell$ to $a_\ell+nL$ for some integer $n$. 
Hence it is enough to show that the right-hand side of~\eqref{eq:multipoint_distribution_origin} is invariant  under these changes. 
Under these changes, the term $\constc (\bbz, \bk)$ is multiplied by the factor $z_\ell^{nNL}z_{\ell+1}^{-nNL}$ if $1\le \ell\le m-1$ and by $z_m^{nNL}$ if $\ell=m$. 
On the other hand, $\fgic_\ell(w_\ell)$ produces the multiplicative factor $w_\ell^{-nN}(w_\ell+1)^{-nL+nN}$
which is $z_\ell^{-nL}$ by~\eqref{eq:rootsz}. 
Taking this factor outside the determinant~\eqref{eq:multipoint_distribution_origin12}, we  cancel out the factor $z_\ell^{nNL}$ from $\constc (\bbz, \bk)$. Similarly  $\fgic_{\ell+1}(w_{\ell+1})$ produces a factor which cancel out $z_{\ell+1}^{-nNL}$ if $1\le \ell\le m-1$. 
\end{proof}

Before we prove the theorem, let us comment on the analytic property of the integrand in the formula~\eqref{eq:multipoint_distribution_origin}. 
The function $\constc (\bbz, \bk)$ is clearly analytic in each $z_\ell\neq 0$. 
Consider the function $\detv_Y(\bbz, \bk, \ba, \bt)$.
Note that 
\beq
	\frac{\dd}{\dd w} q_z(w)  
	= \frac{L(w+\rho)}{w(w+1)} w^N (w+1)^{L-N}.
\eeq
Hence, if $F(w)$ is an analysis function of $w$ in $\complexC\setminus\{-1,0\}$, and 
$f(w)= F(w)w^N(w+1)^{L-N}$, then
\beq \label{eq:Fqintresdi}
	\sum_{w\in \roots_z} F(w)  \frac{w(w+1)}{L(w+\rho)} 
	= \frac1{2\pi \ii} \oint_{|w|=r} \frac{f(w)}{q_z(w)} \dd w- \frac1{2\pi \ii} \oint_{|w+1|=\epsilon_1} \frac{f(w)}{q_z(w)} \dd w
	-  \frac1{2\pi \ii} \oint_{|w|=\epsilon_2} \frac{f(w)}{q_z(w)} \dd w
\eeq
for any $\epsilon_1, \epsilon_2>0$ and $r> \max\{\epsilon_1+1, \epsilon_2\}$ such that all roots of $q_z(w)$ lie in the region $\{w: \epsilon_2<|w|<r, |w+1|>\epsilon_1\}$. 
Note that $q_z(w)$ is an entire function of $z$ for each $w$. 
Since we may take $r$ arbitrarily large and $\epsilon_1, \epsilon_2$  arbitrary small and positive, the right hand-side of~\eqref{eq:Fqintresdi} defines an analytic function of $z\neq 0$.
Now the entries of the determinant in~\eqref{eq:multipoint_distribution_origin12} are of the form
\begin{equation} \label{eq:Fqintresdimut}
	\sum_{\substack{w_1\in \roots_{z1} \\ \cdots\\ w_m\in\roots_{z_m}}}  
	F(w_1, \cdots, w_m) \prod_{\ell=1}^m \frac{w_\ell(w_\ell+1)}{L(w_\ell+\rho)} 
\end{equation}
for a function $F(w_1, \cdots, w_m)$ which is analytic in each variable in $\complexC\setminus\{-1, 0\}$ as long as $w_\ell\neq w_{\ell-1}$ for all $\ell=2, \cdots, m$. The last condition is due to the factor $\prod_{\ell=2}^m (w_\ell-w_{\ell-1})$ in the denominator. 
Note that if $w_\ell=w_{\ell-1}$, then $z_\ell= z_{\ell-1}$. 
Hence by using~\eqref{eq:Fqintresdi} $m$ times, each entry of~\eqref{eq:multipoint_distribution_origin12}, and hence $\detv_Y(\bbz, \bk, \ba, \bt)$, is an analytic function of each $z_\ell\neq 0$ in the region where all $z_\ell$ are distinct.

\medskip

When $m=1$, the product in~\eqref{eq:constcti1e} is set to be $1$ and the formula~\eqref{eq:multipoint_distribution_origin} in this case was obtained in Proposition 6.1 in \cite{Baik-Liu16}. 
For $m\ge 2$, as we mentioned in Introduction, we prove~\eqref{eq:multipoint_distribution_origin} by taking a multiple sum of the transition probability. 
The main new technical result is a summation formula and we summarize it in Proposition~\ref{thm:DRL_simplification} below.

The transition probability was obtained in Proposition 5.1 of \cite{Baik-Liu16}. 
Denoting by $\prob_X(X';t)$ the transition probability from $X=(x_1, \cdots, x_N) \in\conf_N(L)$ to $X'=(x'_1, \cdots, x'_N)\in \conf_N(L)$ in time $t$, 
\begin{equation} \label{eq:transition_probability}
	\prob_X(X';t) = \oint \det \left[ \frac1L \sum_{w\in \roots_z}  
	\frac{w^{j-i+1}(w+1)^{-x'_i+x_j+i-j}e^{tw}}{w+\rho} \right]_{i,j=1}^N \ddbar{z}
\end{equation}
where the integral is over any simple closed contour in $|z|>0$ which contains $0$ inside. 
The integrand is an analytic function of $z$ for $z\neq 0$ by using~\eqref{eq:Fqintresdi}.

\begin{proof}[Proof of Theorem~\ref{prop:multipoint_distribution_origin}]
It is enough to consider $m\ge 2$. 
It is also sufficient to consider the case when the times are distinct, $t_1<\cdots<t_m$, because both sides of~\eqref{eq:multipoint_distribution_origin} are continuous functions of $t_1, \cdots, t_m$. Note that~\eqref{eq:multipoint_distribution_origin12} involve only finite sums.

Denoting by $X^{(\ell)}=(x_1^{(\ell)}, \cdots, x_N^{(\ell)})$ the configuration of the particles at time $t_\ell$, 
the joint distribution function on the left hand-side of~\eqref{eq:multipoint_distribution_origin} is equal to 
\begin{equation}
\label{eq:aux_2016_09_24_01}
	\sum_{\substack{X^{(\ell)}\in \conf_N(L) \cap\{x_{k_\ell}^{(\ell)}\ge a_\ell\} \\ \ell=1,\cdots, m}}
	\prob_Y(X^{(1)};t_1)\prob_{X^{(1)}}(X^{(2)},t_2-t_1)\cdots \prob_{X^{(m-1)}}(X^{(m)};t_m-t_{m-1}).
\end{equation}

Applying the Cauchy-Binnet formula to~\eqref{eq:transition_probability}, we have 
\begin{equation}
\label{eq:reformed_transition_probability}
	\prob_X(X';t)=\oint \sum_{\substack{W\in (\roots_{z})^N}} \DL_X(W) \DR_{X'}(W) \consto(W;t) \ddbar{z}
\end{equation}
where for $W=(w_1,\cdots,w_N)\in \complexC$, 
\begin{equation}
	\DL_X(W)=\det\left[w_i^{j}(w_i+1)^{x_j-j}\right]_{i,j=1}^N,
	\qquad
	\DR_{X'}(W)=\det\left[w_i^{-j}(w_i+1)^{-x'_j+j}\right]_{i,j=1}^N,
\end{equation}
and
\begin{equation}\label{eq:Fforma}
	\consto(W;t)=\frac{1}{N!L^N}\prod_{i=1}^N\frac{w_ie^{tw_i}}{w_i+\rho}.
\end{equation}
\Cb{Here the factor $N!$ in the denominator comes from the Cauchy-Binnet formula; it will eventually disappear since we will apply the Cauch-Binet identity backward again at the end of the proof.}

We  insert~\eqref{eq:reformed_transition_probability} into~\eqref{eq:aux_2016_09_24_01} and interchange the order of the sums and the integrals.
Assuming that the series converges absolutely so that the interchange is possible, the joint distribution is equal to 
\begin{equation}
\label{eq:multi_point_origin}
\begin{split}
	&\oint \ddbar{z_1}\cdots\oint \ddbar{z_m} \sum_{\substack{W^{(1)}\in (\roots_{z_1})^N \\ \cdots\\ W^{(m)}\in(\roots_{z_m})^N}}
	\DP(W^{(1)}, \cdots, W^{(m})  \prod_{\ell=1}^m\consto (W^{(\ell)};t_\ell-t_{\ell-1}) 
\end{split}
\end{equation}
where $W^{(\ell)}=(w_1^{(\ell)}, \cdots, w_N^{(\ell)})$ and 
\begin{equation}
\label{eq:multi_point_origin2}
\begin{split}
	\DP(W^{(1)}, \cdots, W^{(m}) 
	= \DL_Y (W^{(1)}) \left[ \prod_{\ell=1}^{m-1} \DRL_{k_\ell,a_\ell} (W^{(\ell)};W^{(\ell+1)}) \right]
	 \bigg[ \sum_{\substack{X \in \conf_N(L)\\ x_{k_m}\ge a_m}} \DR_{X}(W^{(m)}) \bigg] .
\end{split}
\end{equation}
Here we set  
\begin{equation} \label{eq:sumRLsw}
	\DRL_{k,a}(W;W') :=\sum_{\substack{X\in \conf_N(L)\cap\{x_{k}\ge a\} }} \DR_{X}(W) \DL_{X}(W')
\end{equation}
for a pair of complex vectors $W=(w_1,\cdots,w_N)$ and $W'=(w'_{1},\cdots,w'_N)$. 
Let us now show that it is possible to exchange the sums and integrals if we take the $z_i$-contours properly. 
We first consider the convergence of~\eqref{eq:sumRLsw} and the sum in~\eqref{eq:multi_point_origin2}.
Note that shifting the summation variable $X$ to $X-(b, \cdots, b)$,  
\begin{equation}
	\sum_{\substack{X\in \conf_N(L)\cap\{x_{k}= b\} }} \DR_{X}(W) \DL_{X}(W') =\left[\sum_{\substack{Y\in\conf_N(X)}\cap\{y_k=0\}} \DR_{Y}(W) \DL_{Y}(W') \right] \left(\prod_{j=1}^N\frac{w'_j+1}{w_j+1}\right)^{b}.
\end{equation}
The right hand side of~\eqref{eq:sumRLsw} is the sum of the above formula over $b\ge a$. 
Hence~\eqref{eq:sumRLsw} converges absolutely and the convergence is uniform for $w_i, w_i'$  if  $\prod_{j=1}^N \big| \frac{w'_j+1}{w_j+1} \big|$ is in a compact subset of $[0,1)$. 
Similarly, the sum of $\DR_{X}(W^{(m)})$ in~\eqref{eq:multi_point_origin2} converges if $\prod_{j=1}^N |w_j^{(m)}+1|>1$. 
Therefore,~\eqref{eq:multi_point_origin2} converge absolutely if the intermediate variables $W^{(\ell)}=(w_1^{(\ell)}, \cdots, w_N^{(\ell)})$ satisfy 
\beq \label{eq:wordtep}
	\prod_{j=1}^N|w_j^{(1)}+1| > \prod_{j=1}^N|w_j^{(2)}+1| >\cdots> \prod_{j=1}^N|w_j^{(m)}+1| > 1.
\eeq
We now show that it is possible to choose the contours of $z_i$'s so that~\eqref{eq:wordtep} is achieved. 
Since $W^{(\ell)}\in (R_{z_\ell})^N$, $w_j^{(\ell)}$ satisfies the equation $|w^N(w+1)^{L-N}|=|z_\ell^L|$. 
Hence $|w_j^{(\ell)}|=|z_j|+O(1)$ as $|z_j|\to \infty$. 
Therefore, if we take the contours $|z_\ell|=r_\ell$ where $r_1>\cdots>r_m>0$ and $r_\ell-r_{\ell+1}$ are large enough (where $r_{m+1}:=0$), then~\eqref{eq:wordtep} is satisfied. 
Thus,~\eqref{eq:sumRLsw} and the sum in~\eqref{eq:multi_point_origin2} converge absolutely.
It is easy to see that the convergences are uniform. 
Hence we can exchange the sums and integrals, and therefore, the joint distribution is indeed given by~\eqref{eq:multi_point_origin}
if we take the contours of $z_i$ to be large nested circles. 

We simplify~\eqref{eq:multi_point_origin}. 
The terms $\DRL_{k_\ell,a_\ell} (W^{(\ell)};W^{(\ell+1)})$ are evaluated in Proposition~\ref{thm:DRL_simplification} below.
Note that since the $z_i$-contours are the large nested circles, we have~\eqref{eq:wordtep}, and hence the assumptions in Proposition~\ref{thm:DRL_simplification} are satisfied.
On the other hand, the sum of $\DR_{X}(W^{(m)})$ in~\eqref{eq:multi_point_origin} was computed in \cite{Baik-Liu16}. 
Lemma 6.1 in \cite{Baik-Liu16} implies that for $W=(w_1,\cdots,w_N)\in \roots_z^N$, 
\begin{equation}
\label{eq:aux_2016_10_21_02_diff}
\begin{split}
	&\sum_{\substack{X \in \conf_N(L)\\ x_{k}= a}} \DR_{X}(W) 
	= (-1)^{(k-1)(N+1)} z^{(k-1)L}  \left[ 1- \prod_{j=1}^N (w_{j}+1)^{-1} \right]  \left[ \prod_{j=1}^N w_j ^{-k} (w_{j}+1)^{-a+k+1} \right] \det\left[ w_j^{-i}\right]_{i,j=1}^N.
\end{split}
\end{equation}
Hence, from the geometric series, for $W=(w_1,\cdots,w_N)\in \roots_z^N$, 
\begin{equation}
\label{eq:aux_2016_10_21_02}
\begin{split}
	&\sum_{\substack{X \in \conf_N(L)\\ x_{k}\ge a}}\DR_{X}(W) 
	= (-1)^{(k-1)(N+1)} z^{(k-1)L} \left[ \prod_{j=1}^N w_j ^{-k} (w_{j}+1)^{-a+k+1} \right] \det\left[ w_j^{-i}\right]_{i,j=1}^N
\end{split}
\end{equation}
if $\prod_{j=1}^N|w_j+1|>1$.  The last condition is satisfies for $W=W^{(m)}$. 
We thus find that~\eqref{eq:multi_point_origin2} is equal to an explicit factor times a product of $m-1$ Cauchy determinants times a Vandermonde determinant.
By using the Cauchy-Binet identity $m$ times, we obtain~\eqref{eq:multipoint_distribution_origin} assuming that the $z_i$-contours are large nested circles. 

Finally, using the analyticity of the integrand on the right-hand side of~\eqref{eq:multipoint_distribution_origin}, which was discussed before the start of this proof, we can deform the contours of $z_i$ to any nested circles, not necessarily large circles. This completes the proof. 
\end{proof}

The main technical part of this section is the following summation formula. 
We prove it in Section \ref{sec:proof_DRL}. 

\begin{prop} \label{thm:DRL_simplification}
Let $z$ and $z'$ be two non-zero complex numbers satisfying $z^L\ne (z')^L$.
Let $W=(w_1,\cdots,w_N)\in (\roots_z)^N$ and $W'=(w'_1,\cdots,w'_N)\in (\roots_{z'})^N$. 
Suppose that $\prod_{j=1}^N|w'_j+1|<\prod_{j=1}^N|w_j+1|$. 
Consider $\DRL_{k,a}(W;W')$ defined in~\eqref{eq:sumRLsw}. 
Then for any $1\le k\le N$ and integer $a$, 
\begin{equation}
		\label{eq:DRL_simplification}
		\DRL_{k,a}(W;W')=  \left(\frac{z}{z'}\right)^{(k-1)L}\left(1-\left(\frac{z'}{z}\right)^L\right)^{N-1}
		\left[ \prod_{j=1}^N \frac{w_j^{-k}(w_j+1)^{-a+k+1}}{(w'_j)^{-k}(w'_j+1)^{-a+k}} \right] \det\left[\frac{1}{w_i-w'_{i'}}\right]_{i,i'=1}^N.
\end{equation}
\end{prop}

\section{Periodic step initial condition}\label{sec:stepicn}

We now assume the following periodic step condition: 
\beq \label{eq:psic}
	\text{$\xx_{i+nN}(0)=i-N+nL$ for $1\le i\le N$ and $n\in \intZ$.}
\eeq 
In the previous section, we obtain a formula for general initial conditions. 
In this section, we find a simpler formula for the periodic step initial condition which is suitable for the asymptotic analysis. 
We express $\detv_Y(\bbz, \bk, \ba, \bt)$ as a Fredholm determinant times a simple factor. 
The result is described in terms of two functions $\const(\bbz)$ and $\gdet(\bbz)$. 
We first define them and then state the result.

Throughout this section, we fix a positive integer $m$, and fix parameters $k_1,\cdots,k_m, a_1,\cdots,a_m, t_1,\cdots,t_m$ as in the previous section.

\subsection{Definitions} \label{sec:sticde}

Recall the function $q_z(w) = w^N(w+1)^{L-N}-z^L$ for complex $z$ in~\eqref{eq:def_roots} and the set of its roots 
\beq
	\roots_z = \{w\in\complexC : q_z(w)=0\}
\eeq 
in~\eqref{eq:rootsz}. 
Set 
\beq
	\rr_0:= \rho^\rho(1-\rho)^{1-\rho}, \qquad \rho=N/L 
\eeq 
as in~\eqref{eq:r0defn}. 
We discussed in the previous section that if $0<|z|<\rr_0$, then the contour $|q_z(w)|=0$ consists of two closed contours, one in $\Re(w)<-\rho$ enclosing the point $w=-1$ and the other in $\Re(w)>-\rho$ enclosing the point $w=0$.
Now, for $0<|z|<\rr_0$, 
set
\begin{equation} \label{eq:RLfinds}
	\rootsL_{z}=\{w\in\roots_{z}:\Re(w)<-\rho\},\qquad \rootsR_{z}=\{w\in\roots_{z}:\Re(w)>-\rho\}.
\end{equation}
It is not difficult to check that
\beq
	|\rootsL_z|=L-N, \qquad |\rootsR_z|=N. 
\eeq
See the left picture in Figure~\ref{fig:rootsfinite} in Section~\ref{sec:tran}.
(Note that if $z=0$, then the roots are $w=-1$ with multiplicity $L-N$ and $w=0$ with multiplicity $N$.)
From the definitions, we have  
\beq
	\roots_z= \rootsL_z\cup \rootsR_z.
\eeq

In Theorem~\ref{prop:multipoint_distribution_origin}, we took the contours of $z_i$ as nested circles of arbitrary sizes. 
In this section, we assume that the circles satisfy 
\beq \label{eq:finforzord}
	0<|z_m|<\cdots<|z_1|<\rr_0 .
\eeq
Hence $\rootsL_{z_i}$ and $\rootsR_{z_i}$ are all well-defined. 

We define two functions $\const(\bbz)$ and $\gdet(\bbz)$ of $\bbz=(z_1, \cdots, z_m)$, both of which depend on the parameters $k_i, t_i, a_i$. The first one is the following. Recall the notational convention introduced in Definition~\ref{def:notc}.
For example, $\Delta(\rootsR_{z};\rootsL_{z})=\prod_{v\in \rootsR_z} \prod_{u\in \rootsL_z} (v-u)$.

\begin{defn}
Define 
\begin{equation} \label{eq:forBs}
		\begin{split}
		\const(\bbz)		&=\left[\prod_{\ell=1}^m\frac{\CA_\ell(z_\ell)}{\CA_{\ell-1}(z_{\ell})}\right]\left[\prod_{\ell=1}^m\frac{\prod_{u\in\rootsL_{z_\ell}}(-u)^{N}\prod_{v\in\rootsR_{z_\ell}}(v+1)^{L-N}}{\Delta(\rootsR_{z_\ell};\rootsL_{z_\ell})}\right]\\
		&\qquad\times \left[\prod_{\ell=2}^m\frac{z_{\ell-1}^L}{z_{\ell-1}^L-z_\ell^L}\right]\left[\prod_{\ell=2}^m\frac{\Delta(\rootsR_{z_\ell};\rootsL_{z_{\ell-1}})}{\prod_{u\in\rootsL_{z_{\ell-1}}}(-u)^N\prod_{v\in\rootsR_{z_\ell}}(v+1)^{L-N}}\right]
		\end{split}
\end{equation}
where
\begin{equation} \label{eq:CAbutE}
	\CA_i(z):=\prod_{u\in\rootsL_{z}}(-u)^{k_i-N-1}\prod_{v\in\rootsR_{z}}(v+1)^{-a_i+k_i-N}e^{t_iv}
\end{equation}
for $i=1,\cdots,m$, and $\CA_0(z):=1$.
\end{defn}

It is easy to see that all terms in $\const(\bbz)$ other than  $\prod_{\ell=2}^m\frac{z_{\ell-1}^L}{z_{\ell-1}^L-z_\ell^L}$ are analytic for $z_1,\cdots,z_m$ within the disk $\{z;|z|<\rr_0\}$. Hence $\const(\bbz)$ is analytic in the disk except the simple poles when $z_{\ell-1}^L=z_\ell^{L}$, $\ell=2, \cdots,m$.

\bigskip

We now define $\gdet(\bbz)$. It is given by a Fredholm determinant. 
Set 
\begin{equation} \label{eq:def_fftn}
\begin{split}
	\fftn_{i}(w) &:=  w^{-k_i+N+1} (w+1)^{-a_i+k_i-N} e^{t_iw}\quad  \text{for $i=1,\cdots,m$,}\\
	\fftn_{0}(w) &:=0. 
\end{split}
\end{equation}
Define 
\begin{equation}\label{eq:def_fs}
 	 \fs_i(w)=\begin{dcases}
 	 \frac{\fftn_i(w)}{\fftn_{i-1}(w)} \quad &\text{for $\Re(w)<-\rho$,} \\
 	 \frac{\fftn_{i-1}(w)}{\fftn_i(w)} \quad & \text{for $\Re(w)>-\rho$.}
 	 \end{dcases}
\end{equation}	
Also, set 
\beq \label{eq:jac}
	\jac(w)= \frac{w(w+1)}{L(w+\rho)}. 
\eeq
	
Define, for $0<|z|<\rr_0$,  
\beq \label{eq:def_original_hL_hR}
	\ql_z(w) = \frac1{(w+1)^{L-N}}\prod_{u\in\rootsL_{z}}(w-u),  
	\qquad 
	\qr_z(w) = \frac1{w^N} \prod_{u\in\rootsR_{z}}(w-u).
\eeq
{Note that $\ql_z(w) \qr_z(w) 
	= \frac{q_z(w) }{(w+1)^{L-N}w^N}  
	= \frac{w^N(w+1)^{L-N}-z^L}{w^N(w+1)^{L-N}}$. }
Set 
\beq \label{eq:defofHszw}
	\sH_z(w) := \begin{dcases}
	\qol_z(w) \qquad &\text{for $\Re(w)>-\rho$,} \\
	\qor_z(w) \qquad &\text{for $\Re(w)<-\rho$,}
	\end{dcases}
\eeq	
When $z=0$, we define $\qol_z(w)=\qor_z(w)=1$ and hence $\sH_z(w)=1$.

\begin{figure}
	\centering
	\includegraphics[scale=0.8]{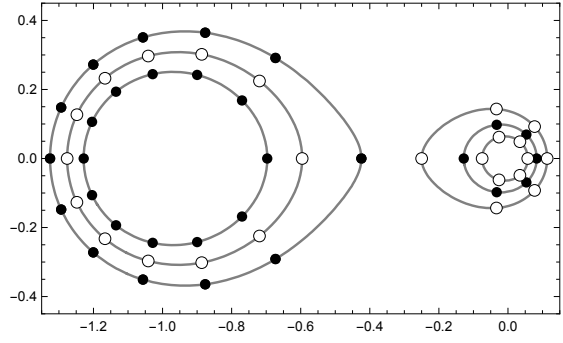}
	\caption{Example of $\mrootL$ and $\mrootR$ when $m=3$. The block dots are $\mrootL$ and the white dots are $\mrootR$. The level sets are shown for visual convenience.}
	\label{fig:rootslr}
\end{figure}

Define two sets  
\begin{equation}
	\mrootL:=\rootsL_{z_1}\cup \rootsR_{z_2} \cup\rootsL_{z_3}\cup\cdots\cup\begin{cases}
\rootsL_{z_m}& \text{if $m$ is odd,} \\
\rootsR_{z_m}& \text{if $m$ is even,} \\
\end{cases}
\end{equation}
and 
\begin{equation}
	\mrootR:=\rootsR_{z_1}\cup \rootsL_{z_2} \cup\rootsR_{z_3}\cup\cdots\cup\begin{cases}
	\rootsR_{z_m}& \text{if $m$ is odd,} \\
	\rootsL_{z_m}& \text{if $m$ is even.} 
\end{cases}
\end{equation}
See Figure~\ref{fig:rootslr}. 
We define two operators 
\beq
	\KsL:\ell^2(\mrootR)\to\ell^2(\mrootL), \qquad \KsR:\ell^2(\mrootL)\to\ell^2(\mrootR)
\eeq
by kernels.
{If $w\in\roots_{z_i}\cap\mrootL$ and $w'\in\roots_{z_j}\cap\mrootR$
for some $i,j\in\{1, \cdots, m\}$, }
we set
\begin{equation} \label{eq:aux_2017_3_17_01}
 \begin{split}
 	\KsL(w,w') &=  (\delta_i(j)+\delta_i (j+(-1)^i) ) 
	\frac{\jac(w)\fs_i(w) (\sH_{z_{i}}(w))^2 }{\sH_{z_{i-(-1)^{i}}}(w) \sH_{z_{j-(-1)^{j}}}(w')(w-w')} \sQL(j) .
\end{split} \end{equation}
{Similarly, if $w\in\roots_{z_i}\cap\mrootR$ and $w'\in\roots_{z_j}\cap\mrootL$ for some $i,j\in\{1, \cdots, m\}$, 
we set }
\begin{equation} \label{eq:aux_2017_3_17_01_0002}
 \begin{split}
 	\KsR(w,w') &= (\delta_i(j)+\delta_i (j-(-1)^i ) )
 	\frac{\jac(w) \fs_i(w) (\sH_{z_i}(w))^2}{\sH_{z_{i+(-1)^{i}}}(w) \sH_{z_{j+(-1)^{j}}}(w')(w-w')} \sQR(j).
\end{split} \end{equation}
Here we set $z_0=z_{m+1}=0$. 
We also set
\begin{equation}
\label{eq:sQ1}
\sQL(j):=1-\left(\frac{z_{j-(-1)^j}}{z_j}\right)^L,\qquad \sQR(j):=1-\left(\frac{z_{j+(-1)^j}}{z_j}\right)^L.
\end{equation}

\begin{defn}
Define
\begin{equation}  \label{eq:stpMfp}
	\gdet(\bbz) =\det (I-\KsL\KsR).
\end{equation}
\end{defn}

\medskip

\begin{rmk} \label{rmk:matrixkernl}
The matrix kernels for $\KsL$ and $\KsR$ have block structures similar to the infinite time case discussed in Subsection~\ref{sec:matrixkl}. 
The only change is that $\lmk$ is replaced by $\kk$ which is given as follows. 
For
\beq
	u\in \rootsL_{z_i}, \qquad v\in \rootsR_{z_i}, \qquad u'\in \rootsL_{z_{i+1}}, \qquad v'\in \rootsR_{z_{i+1}},
\eeq
we have 
\beq
\bes
	\begin{bmatrix} 
	\kk(u,v) & \kk (u,u') \\ \kk(v',v) & \kk(v',u')
	\end{bmatrix}
	&
	= \begin{bmatrix} 
	\frac{\fftn_i(u)}{\fftn_{i-1}(u)}    \\ 
	& \frac{\fftn_{i}(v')}{\fftn_{i+1}(v')} 
	\end{bmatrix}
	\begin{bmatrix} 
	 \frac{u(u+1)\qor_{z_i}(u)^2}{L(u+\rho)\qor_{z_{i+1}}(u)}   \\ 
	 &  \frac{v'(v'+1) \qol_{z_{i+1}}(v')^2}{L(v'+\rho) \qol_{z_i}(v') }  \\
	\end{bmatrix} \\
	&\qquad \times 
	\begin{bmatrix} 
	\frac1{u-v} & \frac1{u-u'}    \\ \frac1{v'-v} & \frac1{v'-u'}
	\end{bmatrix}
	\begin{bmatrix} 
	\frac1{\qol_{z_{i+1}}(v)} \\
	 &  \frac1{\qor_{z_i}(u')}  \\ 
	\end{bmatrix}
	\begin{bmatrix} 
	   1- \frac{z_{i+1}^L}{z_i^L} \\
	 &  1- \frac{z_{i}^L}{z_{i+1}^L}   \\ 
	\end{bmatrix} .
\end{split}
\eeq
\end{rmk}

\bigskip

As in Subsection~\ref{sec:altforjtlim}, the above Fredholm determinant also has two series formulas. 
The proof of the following lemma is given in Subsection~\ref{sec:eqFrsr}. 

\begin{lm}[Series formulas of $\gdet(\bbz)$] \label{lem:DFredexpf} \label{thm:multi_point_formula2}
We have 
\begin{equation} 
\label{eq:aux_2017_03_23_02}
	\gdet(\bbz)
	=\sum_{\bn\in (\intZ_{\ge 0})^m} \frac{1}{(\bn !)^2} \gdet_{\bn}(\bbz)
\end{equation}
with $\bn != \prod_{\ell=1}^m n_\ell!$ for $\bn=(n_1, \cdots, n_m)$ and $\gdet_{\bn}(\bbz)$ can be 
expressed in the following two ways. 
Here we set $\gdet_{\bn}(\bbz)=0$ if one of $n_\ell$ is larger than $N$. 
\begin{enumerate}[(a)]
\item 
We have
\begin{equation} \label{eq:aux_2017_03_22_01}
	\gdet_{\bn}(z_1,\cdots,z_m)
	=(-1)^{|\bn|}\sum_{\substack{\UU^{(\ell)}\in (\rootsL_{z_\ell})^{n_\ell} \\ \VV^{(\ell)}\in (\rootsR_{z_\ell})^{n_\ell} \\ l=1,\cdots,m}}\det\left[\KsL(w_i,w'_j)\right]_{i,j=1}^{|\bn|}
	\det\left[\KsR(w'_i,w_j)\right]_{i,j=1}^{|\bn|}
\end{equation}
where $\bUU=(\UU^{(1)}, \cdots, \UU^{(m)})$, $\bVV= (\VV^{(1)}, \cdots, \VV^{(m)})$ with $\UU^{(\ell)}=(u_1^{(\ell)}, \cdots, u_{n_\ell}^{(\ell)})$, $\VV^{(\ell)}=(v_1^{(\ell)}, \cdots, v_{n_\ell}^{(\ell)})$, and 
\begin{equation}
w_i=\begin{dcases}
	u_k^{(\ell)}& \text{if $i=n_1 +\cdots+n_{\ell-1}+k$ for some $k\le n_\ell$ with odd integer $\ell$,}\\
v_k^{(\ell)}& \text{if $i=n_1 +\cdots+n_{\ell-1}+k$ for some $k\le n_\ell$ with even integer $\ell$,}
\end{dcases}
\end{equation}
and
\begin{equation}
w'_i=\begin{dcases}
v_k^{(\ell)}& \text{if $i=n_1+\cdots+n_{\ell-1}+k$ for some $k\le n_\ell$ with odd integer $\ell$,}\\
u_k^{(\ell)}& \text{if $i=n_1+\cdots+n_{\ell-1}+k$ for some $k\le n_\ell$ with even integer $\ell$.}
\end{dcases}
\end{equation}

\item 
We have  
\begin{equation} \label{eq:stpMfpnnn}
	\begin{split}
	\gdet_{\bn}(\bbz)
	& =
	\sum_{\substack{\UU^{(\ell)}\in (\rootsL_{z_\ell})^{n_\ell}\\ \VV^{(\ell)}\in (\rootsR_{z_\ell})^{n_\ell} \\ \ell=1,\cdots,m}} 
	 \left[\prod_{\ell=1}^m 
	\frac{(\Delta(\UU^{(\ell)}))^2(\Delta(\VV^{(\ell)}))^2}{(\Delta(\UU^{(\ell)};\VV^{(\ell)}))^2}
	\hatf_\ell(\UU^{(\ell)}) \hatf_\ell(\VV^{(\ell)})
	\right]  \\
	&\times   
	\left[ \prod_{\ell=2}^{m}
	\frac{\Delta(\UU^{(\ell)};\VV^{(\ell-1)})\Delta(\VV^{(\ell)};\UU^{(\ell-1)}) (1-\frac{z^L_\ell}{z^L_{\ell-1}})^{n_{\ell-1}}(1-\frac{z^L_{\ell-1}}{z^L_{\ell}})^{n_{\ell}}}{\Delta(\UU^{(\ell)};\UU^{(\ell-1)})\Delta(\VV^{(\ell)};\VV^{(\ell-1)})\qor_{z_{\ell-1}}(\UU^{(\ell)}) \qor_{z_\ell}(\UU^{(\ell-1)})\qol_{z_{\ell-1}}(\VV^{(\ell)})\qol_{z_\ell}(\VV^{(\ell-1)})} \right]
	\end{split}
\end{equation}
where
\beq \label{eq:hatfel}
	\hatf_\ell(w) := \jac(w) \fs_\ell(w) (\sH_{z_\ell}(w))^2 
	= \begin{cases}
	\jac(w) \frac{\fftn_\ell(w)}{\fftn_{\ell-1}(w)} (\qor_{z_\ell}(w))^2 \quad & \text{for $w\in \rootsL_{z_\ell}$,} \\
	\jac(w) \frac{\fftn_{\ell-1}(w)}{\fftn_{\ell}(w)} (\qol_{z_\ell}(w))^2 \quad &\text{for $w\in \rootsR_{z_\ell}$.}
	\end{cases}
\eeq
\end{enumerate}
\end{lm}


\begin{rmk}
From~\eqref{eq:stpMfpnnn}, we can check that $\gdet_{\bn}(\bbz)$ is analytic for each $z_\ell$ in $0<|z_\ell|<\rr_0$, $1\le \ell\le m$ just like $\gdetlm(\bfz)$ of Section~\ref{sec:limitdisformula}.
{The proof for $\gdetlm(\bfz)$ is in Lemma~\ref{lem:Dgetlmmdcov}. 
The proof for $\gdet_{\bn}(\bbz)$ is similar, and we skip it. }
\end{rmk}

\subsection{Result and proof}

\begin{thm}[Joint distribution of TASEP in $\conf_N(L)$ for step initial condition] \label{thm:multi_point_formula}
Consider the TASEP in $\conf_N(L)$ with the step initial condition $\xx_i(0)=i-N$, $1\le i\le N$. 
Set $\rho=N/L$. 
Fix a positive integer $m$.
Let $(k_1, t_1), \cdots, (k_m, t_m)$ be $m$ distinct points in $\{1, \cdots, N\} \times [0, \infty)$. 
Assume that $0\le t_1\le\cdots\le t_m$.
Let $a_i\in \intZ$ for $1\le i\le m$.
Then
\begin{equation}
	\label{eq:multi_point_formula}
	\begin{split}
	&\prob\left(\xx_{k_1}(t_1)\ge a_1,\cdots, \xx_{k_m}(t_m)\ge a_m\right) 
	= \oint\cdots\oint \const(\bbz) \gdet(\bbz) \ddbar{z_m}\cdots\ddbar{z_1} 
	\end{split}
\end{equation}
where $\bbz = (z_1,\cdots,z_m)$ and the contours are nested circles satisfying 
\beq
	0<|z_m|<\cdots <|z_1|<\rr_0
\eeq
with $\rr_0=\rho^\rho(1-\rho)^{1-\rho}$.
The functions $\const(\bbz)$ and $\gdet(\bbz)$ are defined in~\eqref{eq:forBs} and~\eqref{eq:stpMfp}, respectively. 
\end{thm}

Recall that $\const(\bbz)$ is analytic in $|z_\ell|<\rr_0$ except for the poles when $z_{\ell-1}^L=z_\ell^L$, and $\gdet(\bbz)$ is analytic in $0<|z_\ell|<\rr_0$. We point out that Remark~\ref{rmk:parmt} still applies to the above theorem; the Fredholm determinant expansion involves only finite sums.

\begin{cor}[Joint distribution of periodic TASEP for periodic step initial condition] \label{cor:pTASEPstep}
Consider the periodic TASEP with periodic step initial condition, $\xx_{i+nN}(0)=i-N+nL$ for $1\le i\le N$ and $n\in \intZ$.
Then~\eqref{eq:multi_point_formula} holds for all integer indices $k_1, \cdots, k_m$ without the restriction that they are between $1$ and $N$. 
\end{cor}

\begin{proof}
As in the proof of Corollary~\ref{cor:pTASEPgenIC}, it is enough to show that the formulas are invariant under the changes $k_i\mapsto k_i\pm N$ and $a_i\mapsto a_i\pm L$ for  \rev{some} $i$.
This can be checked easily for $\const(\bbz)$ using the identity $\prod_{u\in\rootsL_{z}}(-u)^N=\prod_{v\in\rootsR_{z}}(v+1)^{L-N}$, which is easy to prove; see~\eqref{eq:vvarerela} below. 
For $\gdet(\bbz)$, we use the fact that $u^N(u+1)^{L-N}=z_\ell^L$ for $u\in \rootsL_{z_\ell}$ and $v_i^N(v_i+1)^{L-N}=z_i^L$ for $v\in \rootsR_{z_\ell}$ (plus the special structure of $\KsL$ and $\KsR$.) 
\end{proof}

\begin{proof}[Proof of Theorem~\ref{thm:multi_point_formula}]
When $m=1$, the result was obtained in Theorem 7.4 of \cite{Baik-Liu16}.
We assume $m\ge 2$. 
In Theorem~\ref{prop:multipoint_distribution_origin}, $\constc (\bbz, \bk)$ does not depend on the initial condition. 
Let us denote $\detv_Y(\bbz, \bk, \ba, \bt)$ by $\detv_{\st}$ when $Y=(1-N, \cdots, 1,0)$, the step initial condition.
We need to show that $\detv_{\st}=\gdet(\bbz)\frac{\const(\bbz)}{\constc (\bbz, \bk)}$. 

Inserting the initial condition $y_i=i-N$, re-writing $\fgic_\ell$ in terms of $\fftn_{\ell}$ and $\Jb$  in~\eqref{eq:def_fftn} and~\eqref{eq:jac}, and reversing the rows,
\beq
\begin{split}
		\detv_{\st} 
		 =  (-1)^{N(N-1)/2} \det\Bigg[ \sum
		\frac{w_1^{-i} w_m^{-j}}{ \prod_{\ell=2}^m (w_\ell-w_{\ell-1})  } \prod_{\ell=1}^m 
		\Jb(w_\ell)  \frac{\fftn_\ell(w_\ell)}{\fftn_{\ell-1}(w_\ell)} \Bigg]_{i,j=1}^N .
\end{split}
\eeq
The sum is over all $w_1\in \roots_{z_1}, \cdots, w_m\in\roots_{z_m}$.
Using the Cauchy-Binet identity $m$ times, 
\beq \label{eq:sumdetvap}
	\detv_{\st}
	= \frac{(-1)^{N(N-1)/2}}{(N!)^m}  \sum_{\substack{W^{(\ell)}\in (\roots_{z_\ell})^N\\ \ell=1,\cdots,m}}  \detvv(\bW)
	 \prod_{\ell=1}^m \jac(W^{(\ell)}) \frac{\fftn_\ell(W^{(\ell)})}{\fftn_{\ell-1}(W^{(\ell)})} 
\eeq
where $W^{(\ell)}=(w_1^{(\ell)}, \cdots, w_N^{(\ell)})$ with $w_i^{(\ell)}\in \roots_{z_\ell}$ for each $i$, $\bW=(W^{(1)}, \cdots, W^{(m)})$, and 
\begin{equation}
\begin{split}
 	\detvv(\bW) &=\det\left[(w_i^{(1)})^{-j}\right] 
	\det\left[\frac{1}{w_i^{(2)}-w_j^{(1)}}\right] \cdots 
	\det\left[\frac{1}{w_i^{(m)}-w_j^{(m-1)}}\right]
 	\det\left[(w_i^{(m)})^{-j}\right].
\end{split}
\end{equation}
Here all matrices are indexed by $1\le i,j\le N$. 
Note that in~\eqref{eq:sumdetvap}, we use the notational convention such as $\fftn_\ell(W^{(\ell)}) = \prod_{i=1}^N \fftn_\ell(w_i^{(\ell)})$ mentioned in Definition~\ref{def:notc}.  
Evaluating the Vandermonde determinants and the Cauchy determinants, 
$\detv_{\st}$ is equal to (recall the notations~\eqref{eq:notationddw})
\begin{equation}
\label{eq:aux_2016_12_31_05}
\begin{split}
	& \frac{(-1)^{mN(N-1)/2}}{(N!)^m} \sum_{\substack{W^{(\ell)}\in (\roots_{z_\ell})^N\\ \ell=1,\cdots,m}} 
	\frac{\prod_{\ell=1}^m \Delta (W^{(\ell)})^2}{\prod_{\ell=2}^{m}\Delta(W^{(\ell)};W^{(\ell-1)})}
	\left[ \prod_{i=1}^N (w_i^{(1)}w_i^{(m)})^{-N} \right]
	 \prod_{\ell=1}^m \jac(W^{(\ell)})\frac{\fftn_\ell(W^{(\ell)})}{\fftn_{\ell-1}(W^{(\ell)})} .
\end{split}
\end{equation}
Note that for each $\ell$, we may assume that the coordinates of the vector $W^{(\ell)}$ are all distinct since otherwise the summand is zero due to $\Delta(W^{(\ell)})$. 
Also note that the summand is a symmetric function of the coordinates of $W^{(\ell)}$ for each $\ell$.
Hence instead of taking the sum over the vectors $W^{(\ell)}\in (\roots_{z_\ell})^N$, we can take a sum over 
the subsets $\tW^{(\ell)}\subset \roots_{z_\ell}$ of size $N$: $\detv_{\st}$ is equal to 
\begin{equation} \label{eq:vtost}
\begin{split}
	(-1)^{mN(N-1)/2} \sum_{\substack{\tW^{(\ell)}\subset \roots_{z_\ell} \\ |\tW^{(\ell)}|=N \\ \ell=1,\cdots,m}} 
	\frac{\prod_{\ell=1}^m \Delta (\tW^{(\ell)})^2}{\prod_{\ell=2}^{m}\Delta(\tW^{(\ell)};\tW^{(\ell-1)})}
	\left[ \prod_{i=1}^N (w_i^{(1)}w_i^{(m)})^{-N} \right]
	 \prod_{\ell=1}^m \jac(\tW^{(\ell)})\frac{\fftn_\ell(\tW^{(\ell)})}{\fftn_{\ell-1}(\tW^{(\ell)})} 
\end{split}
\end{equation}
where $w_i^{(1)}$ are the elements of $\tW^{(1)}$ and $w_i^{(m)}$ are the elements of $\tW^{(m)}$.

We now change the sum as follows. 
Since $\roots_{z_\ell}$ is the disjoint union of $\rootsL_{z_\ell}$ and $\rootsR_{z_\ell}$, 
some elements of the set $\tW^{(\ell)}$ are in $\rootsL_{z_\ell}$ and the rest in $\rootsR_{z_\ell}$.
(Recall that $|\rootsL_{z_\ell}|=L-N$ and $|\rootsR_{z_\ell}|=N$.)
Let $\tU^{(\ell)}=\tW^{(\ell)}\cap \rootsL_{z_\ell}$ and $\tV^{(\ell)}=\rootsR_{z_\ell}\setminus \tW^{(\ell)}$.
Observe that since $|\tW^{(\ell)}|=|\rootsR_{z_\ell}|(=N)$, we have $|\tU^{(\ell)}|=|\tV^{(\ell)}|$. 
Call this last number $n_\ell$. 
We thus find that the sum in~\eqref{eq:vtost} can be replaced by the sums
\beq
	\sum_{\substack{n_\ell=0, \cdots, N \\ \ell=1, \cdots, m}} 
	\quad\sum_{\substack{\tU^{(\ell)}\subset \rootsL_{z_\ell},  \tV^{(\ell)}\subset \rootsR_{z_\ell} \\ 
	|\tU^{(\ell)}|=|\tV^{(\ell)}|=n_\ell \\ \ell=1, \cdots, m}}.
\eeq
We now express the summand in terms of $\tU^{(\ell)}$ and $\tV^{(\ell)}$ instead of $\tW^{(\ell)}$. 
First, for any function $h$, 
\beq \label{eq:aux_2016_12_31_03}
	h(\tW^{(\ell)}) 
	= h(\tU^{(\ell)})  \frac{h(\rootsR_{z_\ell})}{h(\tV^{(\ell)})}.
\eeq
Now consider $\Delta (\tW^{(\ell)})^2$.
We suppress the dependence on $\ell$ in the next a few sentences to make the notations light. 
Setting tentatively $\tS=\rootsR_{z}\setminus \tV$ so that $\rootsR_z=\tV\cup\tS$. 
Note that $\tW=\tU\cup \tS$, a disjoint union. We thus have 
\beq
	\Delta (\tW)^2= \Delta (\tU)^2 \Delta (\tS)^2 \Delta(\tU; \tS)^2, 
	\qquad
	\Delta(\rootsR_z)^2= \Delta(\tV)^2\Delta(\tS)^2\Delta(\tV; \tS)^2. 
\eeq 
Let
\beq \label{eq:qRandr}
	q_{z, \RR}(w) := \prod_{v\in \rootsR_z} (w-v) = w^N \qor_z(w). 
\eeq 
Then, 
\beq
	q_{z, \RR}(\tU) = \Delta(\tU; \tV) \Delta(\tU;\tS), 
	\qquad 
	q'_{z, \RR}(\tV) = (-1)^{N(N-1)/2}\Delta(\tV)^2 \Delta(\tV;\tS).
\eeq
It is also direct to see that 
\beq
	\Delta(\rootsR_{z})^2= (-1)^{N(N-1)/2} q'_{z,\RR}(\rootsR_{z}). 
\eeq
From these, after canceling out all terms involving $\tS$ and inserting the dependence on $\ell$, we find that 
\begin{equation}
\label{eq:aux_2016_12_31_04}
\begin{split}
	&\Delta(\tW^{(\ell)})^2
	= (-1)^{N(N-1)/2}
	\frac{\Delta(\tU^{(\ell)})^2 \Delta(\tV^{(\ell)})^2 (q_{z_\ell,\RR}(\tU^{(\ell)}))^2 }{\Delta(\tU^{(\ell)} ;\tV^{(\ell)})^2 (q'_{z_\ell,\RR}(\tV^{(\ell)}))^2}
	q'_{z_\ell,\RR}(\rootsR_{z_\ell}) .
\end{split}
\end{equation}
(This computation was also given in (7.48) and (7.50) of \cite{Baik-Liu16}.)
Similarly,  
\begin{equation}
\label{eq:aux_2016_12_31_02}
	\frac{\Delta(\tW^{(\ell)}; \tW^{(\ell-1)})}{\Delta(\rootsR_{z_{\ell}};\rootsR_{z_{\ell-1}})}=
	\frac{\Delta(\tU^{(\ell)} ;\tU^{(\ell-1)})\Delta(\tV^{(\ell)};\tV^{(\ell-1)})q_{z_{\ell-1},\RR}(\tU^{(\ell)})q_{z_{\ell},\RR}(\tU^{(\ell-1)})}{\Delta(\tU^{(\ell)};\tV^{(\ell-1)})\Delta(\tV^{(\ell)};\tU^{(\ell-1)})q_{z_{\ell-1},\RR}(\tV^{(\ell)})q_{z_{\ell},\RR}(\tV^{(\ell-1)})}.
\end{equation}

We express the summands in~\eqref{eq:vtost} using~\eqref{eq:aux_2016_12_31_03},~\eqref{eq:aux_2016_12_31_04}, and~\eqref{eq:aux_2016_12_31_02}.
We then change the subsets $\tU^{(\ell)}\subset \rootsL_{z_\ell}$ and $\tV^{(\ell)}\subset \rootsR_{z_\ell}$ to vectors $\UU^{(\ell)}\in (\rootsL_{z_\ell})^{n_\ell}$ and $\VV^{(\ell)}\in (\rootsR_{z_\ell})^{n_\ell}$.
This has the effect of introducing the factors $\frac1{(n_\ell!)^2}$.
We thus obtain  
\begin{equation} \label{eq:DEMt}
\begin{split}
	\detv_{\st}
	& = \cnt(\bbz)   \left[ \sum_{\bn\in (\intZ_{\ge 0})^m} \frac{1}{(\bn !)^2} \hat{\gdet}_{\bn}(\bbz) \right]
\end{split}
\end{equation}
with  
\begin{equation} \label{eq:forBs1}
\begin{split}
	\cnt(\bbz)
	= \left[ \frac{\prod_{\ell=1}^mq'_{z_\ell,\RR} (\rootsR_{z_\ell} )}{\prod_{\ell=2}^{m}\Delta (\rootsR_{z_\ell};\rootsR_{z_{\ell-1}} )} \right] 
\left[\prod_{v\in\rootsR_{z_1}}v^{-N}\right] 	\left[\prod_{v\in\rootsR_{z_m}}v^{-N}\right]
	\left[ \prod_{\ell=1}^{m} \jac(\rootsR_{z_\ell})\frac{\fftn_\ell (\rootsR_{z_\ell})}{\fftn_{\ell-1}(\rootsR_{z_\ell})}   \right]
\end{split}
\end{equation}
and
\beq \label{eq:stpMfpnnn00}
	\begin{split}
	\hat{\gdet}_{\bn}(\bbz)
	&=
	\sum_{\substack{\UU^{(\ell)}\in (\rootsL_{z_\ell})^{n_\ell} \\ \VV^{(\ell)}\in (\rootsR_{z_\ell})^{n_\ell} \\ \ell=1,\cdots,m}} 
	 \left[\prod_{\ell=1}^m 
	\frac{(\Delta(\UU^{(\ell)}))^2 (\Delta(\VV^{(\ell)}))^2 (q_{z_\ell,\RR} (\UU^{(\ell)}))^2 \jac(\UU^{(\ell)}) \fftn_\ell(\UU^{(\ell)})\fftn_{\ell-1}(\VV^{(\ell)})} {(\Delta(\UU^{(\ell)};\VV^{(\ell)}))^2 (q'_{z_\ell,\RR} (\VV^{(\ell)}))^2  \jac(\VV^{(\ell)})\fftn_\ell(\VV^{(\ell)})\fftn_{\ell-1}(\UU^{(\ell)})} 
	\right]  \\
	&\times   
	\left[\prod_{i=1}^{n_1} \bigg( \frac{u_i^{(1)}}{v_i^{(1)}} \bigg)^{-N}\right] 
	\left[\prod_{i=1}^{n_m} \bigg( \frac{u_i^{(m)}}{v_i^{(m)}} \bigg)^{-N}\right] 
	\left[ \prod_{\ell=2}^{m}
	\frac{\Delta (\UU^{(\ell)};\VV^{(\ell-1)}) \Delta (\VV^{(\ell)};\UU^{(\ell-1)}) q_{z_{\ell-1},\RR}(\VV^{(\ell)}) q_{z_\ell,\RR} (\VV^{(\ell-1)})}{\Delta(\UU^{(\ell)};\UU^{(\ell-1)}) \Delta(\VV^{(\ell)};\VV^{(\ell-1)})q_{z_{\ell-1},\RR} (\UU^{(\ell)}) q_{z_\ell,\RR} (\UU^{(\ell-1)})} \right].
	\end{split}
\eeq
We re-write $(q'_{z_\ell,\RR} (\VV^{(\ell)}))^2$ using the identity, $q_{z, \RR}'(v)= \frac{v^N}{\jac(v)\qol_z(v)}$, which we prove later in~\eqref{eq:qprimezrrv}.
We also use $q_{z, \RR}(w) =w^N \qor_z(w)$ (see~\eqref{eq:qRandr}). 
Furthermore, we re-express $\qor_{z_\ell}(\VV^{(\ell')})$ in terms of $\qol_z(\VV^{(\ell')})$ using the identity
\beq
	\qor_z(v)= \frac{1}{\qol_z(v)} \big( 1-\frac{z^L}{(z')^L} \big) 
	\qquad \text{for $v\in \rootsR_{z'}$}
\eeq
which follows from the fact that $\qor_z(w)\qol_z(w)= 1- \frac{z^L}{w^N(w+1)^{L-N}}$. 
Then, using the notation~\eqref{eq:hatfel}, we find that $\hat{\gdet}_{\bn}(\bbz)= \gdet_{\bn}(\bbz)$, given by~\eqref{eq:stpMfpnnn}.

\bigskip

Thus, the theorem is proved if we show that $\constc (\bbz, \bk)\cnt(\bbz)= \const(\bbz)$.
Before we prove it, we make the following observations.
\begin{itemize}
\item
For any $v'\in\rootsR_{z'}$, we have $0=q_{z'}(v')=(v')^N(v'+1)^{L-N}-(z')^L$. Hence, for another complex number $z$, 
\begin{equation}
\label{eq:aux_2017_03_26_01}
	z'^L-z^L= (v')^N(v'+1)^{L-N}-z^L= q_z(v')= \prod_{u\in\rootsL_{z}}(v'-u)\prod_{v\in\rootsR_{z}}(v'-v).
\end{equation}

\item As a special case of the above identity, taking $z'=0$ and $v'=0$, we obtain
\begin{equation}
\label{eq:aux_2017_03_26_02}
	z^{L}=(-1)^{N-1}\prod_{u\in\rootsL_{z}}(-u)\prod_{v\in\rootsR_{z}}v.
\end{equation}

\item 
Setting $q_{z, \LL}(w)= \prod_{u\in \rootsL_z}(w-u)$, we have $q_{z,\RR}(w)q_{z, \LL}(w)=q_z(w)=w^N(w+1)^{L-N}-z^L$, and hence,
\beq \label{eq:qprimezrrv}
	q'_{z,\RR}(v) = \frac{L(v+\rho)}{v(v+1)} \frac{v^N(v+1)^{L-N}}{q_{z, \LL}(v)}
	= \frac{v^N}{\jac(v)\qol_{z}(v)}
	\quad \text{for $v\in \rootsR_{z}$.}
\eeq

\item Since $z^{NL}= \prod_{u\in\rootsL_{z}}(-u)^{N}\prod_{v\in\rootsR_{z}}v^{N}$ from~\eqref{eq:aux_2017_03_26_02} and 
$z^{NL}=\prod_{v\in\rootsR_{z}}v^{N}(v+1)^{L-N}$ by using the definition of $\rootsR_{z}$, we find that 
\begin{equation} \label{eq:vvarerela}
	\prod_{u\in\rootsL_{z}}(-u)^N=\prod_{v\in\rootsR_{z}}(v+1)^{L-N}.
\end{equation}
\end{itemize}

We now prove that $\constc (\bbz, \bk)\cnt(\bbz)=\const(\bbz)$.
Consider $\constc (\bbz, \bk)$ defined in~\eqref{eq:constcti1e}. 
Using~\eqref{eq:aux_2017_03_26_01},
\beqq \label{eq:aux_2017_03_26_06}
\begin{split}
&(-1)^{(k_m-1)(N+1)}z_1^{(k_1-1)L}\prod_{\ell=2}^mz_\ell^{(k_\ell-k_{\ell-1})L}\\
&=\left[\prod_{u\in\rootsL_{z_1}}(-u)^{k_1-1}\prod_{v\in\rootsR_{z_1}}v^{(k_1-1)}\right]\cdot\prod_{\ell=2}^m\left[\prod_{u\in\rootsL_{z_\ell}}(-u)^{k_\ell-k_{\ell-1}}\prod_{v\in\rootsR_{z_\ell}}v^{(k_\ell-k_{\ell-1})}\right] .
\end{split}
\eeqq
Using~\eqref{eq:aux_2017_03_26_02} and~\eqref{eq:aux_2017_03_26_01}, 
\beqq \label{eq:aux_2017_03_26_07}
\begin{split}
	\prod_{\ell=2}^m\left(\left(\frac{z_\ell}{z_{\ell-1}}\right)^L-1\right)^N
	&=\prod_{\ell=2}^m\frac{\Delta(\rootsR_{z_\ell};\rootsL_{z_{\ell-1}}) \Delta(\rootsR_{z_\ell};\rootsR_{z_{\ell-1}})}{\prod_{u\in\rootsL_{z_{\ell-1}}}(-u)^N\prod_{v\in\rootsR_{z_{\ell-1}}}v^N}\\
	&= \left[\prod_{\ell=2}^m \Delta(\rootsR_{z_\ell};\rootsR_{z_{\ell-1}})\right]
\left[\prod_{\ell=2}^m \frac{\Delta(\rootsR_{z_\ell};\rootsL_{z_{\ell-1}})}{\prod_{u\in\rootsL_{z_{\ell-1}}}(-u)^N\prod_{v\in\rootsR_{z_{\ell}}}(v+1)^{L-N}}\right]\\
&\qquad \times
	\left[\prod_{\ell=2}^m\prod_{v\in\rootsR_{z_{\ell}}}(v+1)^{L-N}\right]\left[\prod_{\ell=1}^{m-1}\prod_{v\in\rootsR_{z_{\ell}}}v^{-N}\right].
\end{split}
\eeqq
Now consider $\cnt(\bbz)$.
Using~\eqref{eq:qprimezrrv}, the fact that $\prod_{v\in \rootsR_z} q_{z,\LL}(v)= \Delta(\rootsR_{z};\rootsL_{z})$, and~\eqref{eq:vvarerela}, we see that
\beqq \label{eq:aux_2017_03_26_03}
\begin{split}
	q'_{z,\RR} (\rootsR_{z}) \jac(R_{z})
	&
	= \frac{\prod_{v\in\rootsR_{z}}(v+1)^{L-N}}{\Delta(\rootsR_{z};\rootsL_{z})}\left[\prod_{v\in\rootsR_{z}}v^N\right] \\
	&= \frac{\left[\prod_{v\in\rootsR_{z}}(v+1)^{L-N}\right]\left[\prod_{u\in\rootsL_{z}}(-u)^N\right]}{\Delta(\rootsR_{z};\rootsL_{z})}\left[\prod_{v\in\rootsR_{z}}v^N\right]\left[\prod_{v\in\rootsR_{z}}(v+1)^{-L+N}\right].
\end{split}
\eeqq
This implies that 
\beqq \label{eq:aux_2017_03_26_05}
\begin{split}
	&\prod_{\ell=1}^m q'_{z_\ell,\RR}(\rootsR_{z_\ell}) \jac(\rootsR_{z_\ell})\\
	&= \prod_{\ell=1}^m\frac{\left[\prod_{v\in\rootsR_{z_\ell}}(v+1)^{L-N}\right]\left[\prod_{u\in\rootsL_{z_\ell}}(-u)^N\right]}{\Delta(\rootsR_{z_\ell};\rootsL_{z_\ell})} \left[\prod_{\ell=1}^m\prod_{v\in\rootsR_{z_\ell}}v^N\right]\left[\prod_{\ell=2}^m\prod_{v\in\rootsR_{z_\ell}}(v+1)^{-L+N}\right]\left[\prod_{u\in\rootsL_{z_1}}(-u)^{-N}\right].
\end{split}
\eeqq 
From these calculations, we find that $\constc (\bbz, \bk)\cnt(\bbz)=\const(\bbz)$.

\end{proof}

\subsection{Equivalence of the Fredholm determinant and series formulas} \label{sec:eqFrsr}

We presented three formulas of $\gdet(\bbz)$ in Subsection~\ref{sec:sticde}: see Lemma~\ref{lem:DFredexpf}. 
One of them is a Fredholm determinant~\eqref{eq:stpMfp} and the other two are series formulas~\eqref{eq:aux_2017_03_23_02} with~\eqref{eq:aux_2017_03_22_01} and~\eqref{eq:stpMfpnnn}. 
In this subsection, we prove the equivalence of these formulas.
The proof is general, and the same argument also gives a proof of Lemma~\ref{lem:Freserfola}, which states the equivalence of three formulas of $\gdetlm(\bbz)$, a limit of $\gdet(\bbz)$.

First, we prove Lemma~\ref{lem:DFredexpf}(a) and Lemma~\ref{lem:Freserfola}(a).
These are special cases of the next general result which follows from the series definition of Fredholm determinant together with a block structure of the operator. 

\begin{lm}	\label{lm:aux2}
Let $\Sigma_1, \cdots, \Sigma_m$ be disjoint sets in $\C$ 
and let $\cH=L^2(\Sigma_1\cup\cdots \cup \Sigma_m, \mu)$ for some measure $\mu$.
Let $\Sigma_1', \cdots, \Sigma_m'$ be another collection of disjoint sets in $\C$ 
and let $\cH'=L^2(\Sigma_1'\cup\cdots \cup \Sigma_m', \mu')$ for some measure $\mu'$. 
Let $A$ be an operator from $\cH'$ to $\cH$  
and $B$ an operator from $\cH$ to $\cH'$, both of which are defined by kernels. 
Suppose the following block structures:
\begin{itemize}
\item $A(w,w')=0$ unless there is an index $i$ such that  $w\in \Sigma_{2i-1}\cup \Sigma_{2i}$ and $w'\in \Sigma'_{2i-1}\cup \Sigma'_{2i}$,
\item $B(w',w)=0$ unless there is an index $i$ such that $w'\in \Sigma'_{2i}\cup \Sigma'_{2i+1}$ and $w\in \Sigma_{2i}\cup \Sigma_{2i+1}$. 
\end{itemize}
Assume that the Fredholm determinant $\det(1-AB)$ is well-defined and is equal to the usual Fredholm determinant series expansion. 
Then 
\begin{equation} \label{eq:detabblk}
\bes
	\det(1-AB )
	= & \sum_{\bn\in(\intZ_{\ge 0})^m} \frac{(-1)^{|\bn|}}{(\bn!)^2} 
	\int_{\Sigma_1^{n_1}\times\cdots \times \Sigma_m^{n_m}}  
	\int_{(\Sigma_1')^{n_1}\times\cdots \times (\Sigma_m')^{n_m}}  \\
	&\quad \det\left[ A(w_i,w_j')\right]_{i,j=1}^{|\bn|} \det\left[ B(w_i',w_j)\right]_{i,j=1}^{|\bn|} \prod_{i=1}^{|\bn|} \dd\mu'(w_i') \prod_{i=1}^{|\bn|} \dd\mu(w_i) 
\end{split}
\end{equation}
where $\bn=(n_1, \cdots, n_m)$. 
\end{lm}

\begin{proof}
From the standard Fredholm determinant series expansion,  
\beqq
	\det(1-AB )
	=\sum_{\bn\in(\intZ_{\ge 0})^m} \frac{(-1)^{|\bn|}}{\bn!} 
	\int_{\Sigma_1^{n_1}\times\cdots \times \Sigma_m^{n_m}}  
	\det\left[ (AB)(w_i,w_j)\right]_{i,j=1}^{|\bn|} \prod_{i=1}^{|\bn|} \dd\mu(w_i)
\eeqq
where among $|\bn|$ variables $w_i$, the first $n_1$ are in $\Sigma_1$, the next $n_2$ are in $\Sigma_2$, and so on. 
By the Cauchy-Binet formula, for given $\bn$, 
\beqq
\bes
	&\det\left[ (AB)(w_i,w_j)\right]_{i,j=1}^{|\bn|}
	= \frac1{|\bn|!}
	\int_{(\Sigma_1'\cup \cdots \cup \Sigma_m')^{|\bn|}}  
	\det\left[ A(w_i,w_j')\right]_{i,j=1}^{|\bn|} \det\left[ B(w_i',w_j)\right]_{i,j=1}^{|\bn|} \prod_{i=1}^{|\bn|} \dd\mu'(w_i') \\
	&\qquad = \sum_{\substack{\bn'\in(\intZ_{\ge 0})^m \\ |\bn'|=|\bn|}}
	\frac1{(\bn')!}  \int_{(\Sigma_1')^{n_1'}\times\cdots \times (\Sigma_m')^{n_m'}}  
	\det\left[ A(w_i,w_j')\right]_{i,j=1}^{|\bn|} \det\left[ B(w_i',w_j)\right]_{i,j=1}^{|\bn|} \prod_{i=1}^{|\bn|} \dd\mu'(w_i') .
\end{split}
\eeqq
From the structure of the kernels, we find that the matrix $A(w_i,w_j')$ has a natural  block structure of block sizes $n_1+n_2$, $n_3+n_4$, and so on for the rows and of sizes 
$n_1'+n_2'$, $n_3'+n_4'$, so on for the columns. 
The matrix has non-zero entries only in the diagonal blocks, which have sizes $(n_1+n_2)\times (n_1'+n_2')$, $(n_3+n_4)\times (n_3'+n_4')$, ...
Similarly,  the matrix $B(w_i,w_j')$ has non-zero entries only on the diagonal blocks, which have sizes 
$n_1\times n_1'$, $(n_2+n_3)\times (n_2'+n_3')$, $(n_4+n_5)\times (n_4'+n_5')$, ...
From these structures, we find that the determinant of $A(w_i,w_j')$ times the determinant of    $B(w_i,w_j')$ is zero unless $\bn'=\bn$. 
Hence
\beqq
\bes
	&\det\left[ (AB)(w_i,w_j)\right]_{i,j=1}^{|\bn|}
	= 
	\frac1{\bn!}  \int_{(\Sigma_1')^{n_1}\times\cdots \times (\Sigma_m')^{n_m}}  
	\det\left[ A(w_i,w_j')\right]_{i,j=1}^{|\bn|} \det\left[ B(w_i',w_j)\right]_{i,j=1}^{|\bn|} \prod_{i=1}^{|\bn|} \dd\mu'(w_i').
\end{split}
\eeqq
This implies~\eqref{eq:detabblk}.
\end{proof}

Now, we prove Lemma~\ref{lem:DFredexpf}(b) by showing that~\eqref{eq:aux_2017_03_22_01} and~\eqref{eq:stpMfpnnn} are the same. 
We prove Lemma~\ref{lem:Freserfola}(b) similarly. 
For this purpose, we use the following general result. 

\begin{lm} \label{lm:aux_01}
Let $\UU=(u_1,\cdots,u_m)$, $\VV=(v_1,\cdots,v_m)$, $\UU'=(u'_1,\cdots,u'_n)$ and $\VV'=(v'_1,\cdots,v'_n)$ be four complex vectors. 
Then for any single-variable functions $\rF$, $\rG$, $\rI$, and $\rJ$,  	
\begin{equation}
\bes
	& \frac{\Delta(\UU)\Delta(\VV)\Delta(\UU')\Delta(\VV')\Delta(\UU';\VV)\Delta(\VV';\UU)}{\Delta(\UU;\VV)\Delta(\UU';\VV')\Delta(\UU';\UU)\Delta(\VV';\VV)} 
	\rF(\UU) \rG(\VV) \rI(\UU') \rJ(\VV') \\
	&\qquad =(-1)^{{n+m\choose 2}+n+mn}
	\det\left[ 
	\begin{array}{cc}
	\left[ \frac{\rF(u_i)\rG(v_j)}{u_i-v_j} \right]_{m\times m} 
	&\left[ \frac{\rF(u_i)\rI(u'_j)}{u_i-u_j'} \right]_{m\times n}  \\ 
&\\
\left[ \frac{\rJ(v_i')\rG(v_j)}{v_i'-v_j} \right]_{n\times m}  & 
\left[ \frac{\rJ(v_i')\rI(u_j')}{v_i'-u_j'} \right]_{n\times n} \end{array}
	\right] .
\end{split}
\end{equation}
\end{lm}

\begin{proof}
	It follows directly from the Cauchy determinant formula.
\end{proof}

Consider~\eqref{eq:stpMfpnnn}. We can write it as 
\begin{equation}
	\gdet_{\bn}(\bbz)=\sum_{\substack{\UU^{(\ell)}\in (\rootsL_{z_\ell})^{n_\ell}  \\ \VV^{(\ell)}\in (\rootsR_{z_\ell})^{n_\ell} \\ l=1,\cdots,m}} \Dt_0\cdots \Dt_{m}
\end{equation}
where
\begin{equation}
\bes
	 \Dt_0 =\frac{\Delta(\UU^{(1)})\Delta(\VV^{(1)})}{\Delta(\UU^{(1)};\VV^{(1)})}		
	\hatf_1(\VV^{(1)}), 
\qquad 
	 \Dt_{m} =\frac{\Delta(\UU^{(m)})\Delta(\VV^{(m)})}{\Delta(\UU^{(m)};\VV^{(m)})}
	\hatf_m(\UU^{(m)}) , 
\end{split}
\end{equation}
and
\begin{equation} \label{eq:Dli2m}
	\begin{split}
	\Dt_\ell&=\frac{\Delta(\UU^{(\ell)})\Delta(\VV^{(\ell)})\Delta(\UU^{(\ell+1)})\Delta(\VV^{(\ell+1)})\Delta(\UU^{(\ell+1)};\VV^{(\ell)})\Delta(\VV^{(\ell+1)};\UU^{(\ell)})}{\Delta(\UU^{(\ell)};\VV^{(\ell)})\Delta(\UU^{(\ell+1)};\VV^{(\ell+1)})\Delta(\UU^{(\ell+1)};\UU^{(\ell)})\Delta(\VV^{(\ell+1)};\VV^{(\ell)})}
	\\
	&\qquad \times 
	\frac{ \hatf_{\ell}(\UU^{(\ell)})   (1-\frac{z^L_{\ell+1}}{z^L_{\ell}})^{n_{\ell}}(1-\frac{z^L_{\ell}}{z^L_{\ell+1}})^{n_{\ell+1}} \hatf_\ell(\VV^{(\ell+1)})}{\qor_{z_{\ell+1}}(\UU^{(\ell)}) \qol_{z_{\ell+1}}(\VV^{(\ell)}) \qor_{z_{\ell}}(\UU^{(\ell+1)}) \qol_{z_{\ell}}(\VV^{(\ell+1)})} 
	\end{split}
\end{equation}
for $1\le \ell\le m-1$. 
For $1\le \ell\le m-1$, we apply Lemma~\ref{lm:aux_01} with $m=n_{\ell}$, $n=n_{\ell+1}$, 
$\rF(u)=\frac{\hatf_{\ell}(u)}{\qor_{z_{\ell+1}}(u)}$, 
$\rG(v)= \frac1{\qol_{z_{\ell+1}}(v)} ( 1-\frac{z_{\ell+1}^L}{z_{\ell}^L})$, 
$\rI(u')= \frac1{\qor_{z_{\ell}}(u')} ( 1-\frac{z_{\ell}^L}{z_{\ell+1}^L})$, and
$\rJ(v')= \frac{\hatf_{\ell+1}(v')}{\qol_{z_{\ell}}(v')}$.
Then, recalling~\eqref{eq:hatfel}, 
\begin{equation*}
	\Dt_{\ell}=(-1)^{{n_{\ell}+ n_{\ell+1}\choose 2}+n_{\ell+1}+n_{\ell}n_{\ell+1} }\det\left[
\begin{array}{cc}
\left[\KsL(u^{(\ell)}_i,v^{(\ell)}_j)\right]_{n_{\ell} \times n_{\ell} }&
\left[\KsL(u_i^{(\ell)},u^{(\ell+1)}_j)\right]_{n_{\ell}\times n_{\ell+1}}\\
&\\
\left[\KsL(v^{(\ell+1)}_i,v^{(\ell)}_j)\right]_{n_{\ell+1}\times n_{\ell}} &
\left[\KsL(v^{(\ell+1)}_i,u^{(\ell+1)}_j)\right]_{n_{\ell+1}\times n_{\ell+1}} 
\end{array}
\right]
\end{equation*}
if $\ell$ is odd, and 
\begin{equation*}
	\Dt_{\ell}=(-1)^{{n_{\ell}+ n_{\ell+1}\choose 2}+n_{\ell+1}+n_{\ell}n_{\ell+1} }\det\left[
\begin{array}{cc}
\left[\KsR(u^{(\ell)}_i,v^{(\ell)}_j)\right]_{n_{\ell} \times n_{\ell} }&
\left[\KsR(u_i^{(\ell)},u^{(\ell+1)}_j)\right]_{n_{\ell}\times n_{\ell+1}}\\
&\\
\left[\KsR(v^{(\ell+1)}_i,v^{(\ell)}_j)\right]_{n_{\ell+1}\times n_{\ell}} &
\left[\KsR(v^{(\ell+1)}_i,u^{(\ell+1)}_j)\right]_{n_{\ell+1}\times n_{\ell+1}} 
\end{array}
\right]
\end{equation*}
if $\ell$ is even. 
On the other hand, using the Cauchy determinant formula,
\begin{equation}
	\Dt_0=(-1)^{{n_1 \choose 2}+n_1} \det\left[\KsR(v_i^{(1)},u_j^{(1)})\right]_{i,j=1}^{n_1} 
\end{equation}
and
\begin{equation}
	\Dt_{m}=\begin{dcases}
(-1)^{{n_m \choose 2}}\det\left[\KsL(u^{(m)}_i,v^{(m)}_j)\right]_{i,j=1}^{n_m} 
	& \text{if $m$ is even,}\\
(-1)^{{n_m \choose 2}}\det\left[\KsR(u^{(m)}_i,v^{(m)}_j)\right]_{i,j=1}^{n_m} 
	& \text{if $m$ is odd.}
\end{dcases}
\end{equation}
The formula~\eqref{eq:aux_2017_03_22_01} follows by combining the product of $\Dt_\ell$ for odd $\ell$ as a single determinant of a block diagonal matrix, and combining the product $\Dt_\ell$ for even indices as another single determinant, 
we obtain~\eqref{eq:aux_2017_03_22_01}.
This proves Lemma~\ref{lem:DFredexpf}(a). 
The proof of Lemma~\ref{lem:Freserfola}(b) is similar.

\section{Proof of Proposition~\ref{thm:DRL_simplification}}
\label{sec:proof_DRL}

As we mentioned before, Proposition~\ref{thm:DRL_simplification} is the key technical result of this paper. 
We prove it in this section. 

Let $z$ and $z'$ be two non-zero complex numbers satisfying $z^L\ne (z')^L$. 
Let   $W=(w_1,\cdots,w_N)\in (\roots_z)^N$ and $W'=(w'_1,\cdots,w'_N)\in(\roots_{z'})^N$ be two complex vectors satisfying $\prod_{j=1}^N|w'_j+1|<\prod_{j=1}^N|w_j+1|$. 
Let 
\begin{equation}
	\DL_Y(W)=\det\left[w_i^{j}(w_i+1)^{y_j-j}\right]_{i,j=1}^N,
	\qquad
	\DR_X(W)=\det\left[w_i^{-j}(w_i+1)^{-x_j+j}\right]_{i,j=1}^N.
\end{equation}
The goal is to evaluate 
\begin{equation}\label{eq:DRLsumtemo}
	\DRL_{k,a}(W;W')
	=\sum_{\substack{X\in \conf_N(L)\cap\{x_{k}\ge a\} }} \DR_{X}(W) \DL_{X}(W')
\end{equation}
and show that it is equal to the right hand-side of~\eqref{eq:DRL_simplification}.

We first reduce general $k$ case to $k=1$ case. 

\begin{lm} \label{lem:kto1}
Under the same assumptions as Proposition~\ref{thm:DRL_simplification}, 
\begin{equation}
\label{eq:aux_2016_10_20_08}
	\DRL_{k,a}(W;W')=\left(\frac{z}{z'}\right)^{(k-1)L} 
	\left[ \prod_{j=1}^N\left(\frac{(w_j+1)w'_j}{w_j(w'_j+1))}\right)^{k-1} \right] \DRL_{1,a}(W;W').
\end{equation}
\end{lm}

\begin{proof}
The sum in~\eqref{eq:DRLsumtemo} is over the discrete variables $x_1<\cdots< x_N<x_{1}+L$ with $x_k\ge a$. The first condition is equivalent to 
$x_k<\cdots< x_N<x_1+L<\cdots< x_{k-1}+L<x_k+L$.
Hence, if we set $x_j'=x_{j+ k-1}$ for $j=1, \cdots, N-k+1$ and $x_j'=x_{j+k-1-N}+L$ for $j=N-k+2, \cdots, N$, 
then the sum is over $x_1'<\cdots< x_N'<x_1'+L$ with $x_1'\ge a$. 
Consider 
\begin{equation*}
\begin{split}
	\DR_X(W)&=\det\left[w_i^{-j}(w_i+1)^{-x_j+j}\right]_{i,j=1}^N.
\end{split}
\end{equation*}
If we move the first $k-1$ columns to the end of the matrix and use the variables $x_j'$,
then $\DR_X(W)$
is equal to $(-1)^{(k-1)(N-1)}$ times the determinant of the matrix whose $(i,j)$-entry is 
$w_i^{-j-k+1}(w_i+1)^{-x_j'+j+k-1}$ for the first $N-k+1$ columns and 
$w_i^{-j-k+1+N}(w_i+1)^{-x_j'+L+j+k-1-N}$ for the remaining $k-1$ columns. 
Since $w_i^N(w_i+1)^{L-N}=z^L$ for $w_i\in R_z$, the entries of the last $k-1$ are
$z^L$ times $w_i^{-j-k+1}(w_i+1)^{-x_j'+j+k-1}$.  
Thus, the row $i$ of the matrix has the common multiplicative factor $w_i^{-k+1}(w_i+1)^{k-1}$. 
Factoring out $z^L$ from the last $k-1$ columns and also the common row factors, we find that, 
setting $X'=(x_1', \cdots, x_N')$,  
\begin{equation}
\begin{split}
	\DR_X(W) =(-1)^{(k-1)(N-1)}z^{(k-1)L} \left[ \prod_{j=1}^N\left(\frac{w_j+1}{w_j}\right)^{k-1} \right]
	\DR_{X'} (W) . 
\end{split}
\end{equation}
Similarly, 
\begin{equation}
\begin{split}
	\DL_X(W')
	=(-1)^{(k-1)(N-1)}(z')^{-(k-1)L} \left[ \prod_{j=1}^N\left(\frac{w'_j+1}{w'_j}\right)^{-k+1} \right] 	
	\DL_{X'}(W').
\end{split}
\end{equation}
Hence the sum~\eqref{eq:DRLsumtemo} is a certain explicit constant  times the sum 
\beqq
	\sum_{\substack{X'\in \conf_N(L)\cap\{x_{1}'\ge a\} }} \DR_{X'}(W) \DL_{X'}(W')
\eeqq
which is $\DRL_{1,a}(W;W')$.
Checking the multiplicative constant factor explicitly, we obtain the lemma.
\end{proof}

It is thus enough to prove Proposition~\ref{thm:DRL_simplification} for $k=1$. 
Set
\beq \label{eq:Hsum}
	H_a(W;W'):=\DRL_{1,a}(W;W')-\DRL_{1,a+1}(W;W')=\sum_{\substack{X\in\conf_N(L)\\ x_1=a}}\DR_X(W)\DL_X(W').
\eeq
We prove the following result in this section. 

\begin{prop}\label{prop:01}
Let $z,z'\in\complexC\setminus\{0\}$ such that $z^L\ne (z')^L$. 
Then for every $W=(w_1, \cdots, w_N)\in (\roots_z)^N$, $W'=(w'_1, \cdots, w'_N)\in (\roots_{z'})^N$, and integer $a$, 
\begin{equation}
\label{eq:aux_2016_12_25_05}
	H_{a}(W;W')=  - \left( \left(\frac{z'}{z}\right)^L -1 \right)^{N-1}\left( \prod_{j=1}^N\frac{w'_j(w'_j+1)^{a-1}}{w_j(w_j+1)^{a-2}} - \prod_{j=1}^N\frac{w'_j(w'_j+1)^{a}}{w_j(w_j+1)^{a-1}} \right)\det\left[\frac{1}{w'_{i'}-w_i}\right]_{i,i'=1}^N.
\end{equation}
\end{prop}

Note that the sum~\eqref{eq:Hsum} is over the finite set $a<x_2<\cdots< x_N<a+L$.  
Since it is a finite sum, there is no issue of convergence, and hence we do not need to assume that  $\prod_{j=1}^N|w'_j+1|<\prod_{j=1}^N|w_j+1|$. 
Now, if we assume this extra condition, then the sum of $H_b(W;W')$ over $b=a,a+1, a+2, \cdots$ converges and equal to $\DRL_{1,a}(W;W')$, and the resulting formula is 
\begin{equation}
\label{eq:DRL_1}
	\DRL_{1,a}(W;W')= - \left( \left(\frac{z'}{z}\right)^L-1 \right)^{N-1} \left[ \prod_{j=1}^N\frac{w'_j(w'_j+1)^{a-1}}{w_j(w_j+1)^{a-2}} \right] \det\left[\frac{1}{w'_{i'}-w_i}\right]_{i,i'=1}^N.
\end{equation} 
This is precisely Proposition~\ref{thm:DRL_simplification} when $k=1$. 
Hence, by Lemma~\ref{lem:kto1}, Proposition~\ref{thm:DRL_simplification} is obtained if we prove Proposition~\ref{prop:01}. The rest of this section is devoted to the proof of Proposition~\ref{prop:01}.

The proof is split into four steps. 
The condition that $W\in (\roots_z)^N$, $W'\in (\roots_{z'})^N$ is used only in step 4. 
Step 1, 2, and 3 apply to any complex vectors.

\subsection{Step 1}

We expand $\DR_X(W)$ and $\DL_X(W')$ as sums and  interchange the order of summation.

Let $S_N$ be the symmetric group of order $N$. Let $\sign(\sigma)$ denote the sign of permutation $\sigma\in S_N$. 
Expanding the determinant of $\DR_X(W)$ and $\DL_X(W')$, we have 
\begin{equation}
\label{eq:aux_2016_10_16_01}
	H_a(W;W') = \sum_{\sigma,\sigma'\in S_N}\sign(\sigma\sigma') \left[ \prod_{j=1}^N\left(\frac{w'_{\sigma'(j)}}{w_{\sigma(j)}}\right)^{j} \right] 
	S_{\sigma, \sigma'}(W;W')
\end{equation}
where
\beq 
	S_{\sigma, \sigma'}(W;W') = \sum_{\substack{X\in\conf_{N}(L)\\ x_1= a}}  
	\prod_{j=1}^N\left(\frac{w'_{\sigma'(j)}+1}{w_{\sigma(j)}+1}\right)^{x_j-j} . 
\eeq
We re-write the last sum 
using the following formula. 

\begin{lm}
For complex numbers $f_j$, set 
\beq
	F_{m,n}= \prod_{j=m}^{n} f_j \qquad \text{for $1\le m\le n$.}
\eeq
Then 
\beq \label{eq:aux_2016_12_20_02}
\bes
	&\sum_{\substack{X\in\conf_{N}(L)\\ x_1= a}}    \prod_{j=1}^N (f_j )^{x_j-j}    \\
	&=  \sum_{k=0}^{N-1}\sum_{1<s_1<\cdots<s_k\le N}\dfrac{(F_{1, s_1-1})^{a-1}}{\prod_{j=2}^{s_1-1}(1-F_{j, s_1-1})} 
	 \prod_{i=1}^{k}\frac{(F_{s_i, s_{i+1}-1})^{a+L-N-1}}{(1-(F_{s_i, s_{i+1}-1})^{-1})\prod_{j=s_i+1}^{s_{i+1}-1}(1-F_{j, s_{i+1}-1})}
\end{split}
\eeq
where we set $s_{k+1}=N+1$.  
When $k=0$, the summand is $\dfrac{(F_{1, N})^{a-1}}{\prod_{j=2}^{N}(1-F_{j, N})}$. 
\end{lm}

\begin{proof}
The sum  is over $a<x_2<\cdots<x_N<a+L$. 
We evaluate the repeated sums in the following order: $x_N, x_{N-1}, \cdots, x_2$. 
For $x_j$ the summation range is $x_{j-1}+1\le x_j\le a+L-N+j-1$. Hence the sum is over
\beqq 
	\sum_{x_2=a+1}^{a+L-N+1} \cdots \sum_{x_{N-1}=x_{N-2}+1}^{a+L-2} \sum_{x_N=x_{N-1}+1}^{a+L-1}.
\eeqq
For the sum over $x_j$, we use the  simple identity 
\begin{equation}
\label{eq:identity_basic_sum}
	\sum_{x_j=x_{j-1}+1}^{a+L-N+j-1} A^{x_j-j}=\frac{A^{x_{j-1}-(j-1)}}{1-A}+\frac{A^{a+L-N-1}}{1-A^{-1}}.
\end{equation}
Note that the first term of the right hand-side involves $x_{j-1}$ and the second term does not. 
Applying~\eqref{eq:identity_basic_sum} repeatedly, the left hand-side of~\eqref{eq:aux_2016_12_20_02} 
is equal to a sum of $2^{N-1}$ terms. 
Each term is a product of $N-1$ terms some of which are from the first term of~\eqref{eq:identity_basic_sum} and the rest are from the second term. 
We combine the summand into groups depending on how many times the second term $\frac{A^{a+L-N-1}}{1-A^{-1}}$ is used. 
This number is represented by $k$: $0\le k\le N-1$. 
For given $1\le k\le N-1$, we denote by $1\le s_1<\cdots< s_k<N$ the indices $j$ such that we had chosen the second term when we take the sum over $x_j$~\eqref{eq:identity_basic_sum}.
The result then follows after a simple algebra. 
\end{proof}

We apply the lemma to $S_{\sigma, \sigma'}(W;W')$ and exchange the summation orders.
Then, 
\begin{equation}
\label{eq:aux_2016_10_16_02}
\begin{split}
	H_a(W;W')
	=& \sum_{k= 0}^{N-1}\sum_{1< s_1<\cdots<s_k\le N}\sum_{\sigma,\sigma'\in S_N}
	\sign(\sigma\sigma') \prod_{i=0}^k U_i
\end{split}
\end{equation}
where
\begin{equation}
\begin{split}
	U_0= \dfrac{P_1(s_1-1)\left(Q_{1}(s_1-1)\right)^{a-1}}{\prod_{j=2}^{s_1-1}\left(1-Q_j(s_1-1)\right)} 
\end{split}
\end{equation}
and, for $1\le i\le k$, 
\begin{equation}
\begin{split}
	U_i= \frac{P_{s_i}(s_{i+1}-1)\left(Q_{s_i}(s_{i+1}-1)\right)^{a+L-N-1}}{\left(1-\left(Q_{s_i}(s_{i+1}-1)\right)^{-1}\right)\prod_{j=s_i+1}^{s_{i+1}-1}\left(1-Q_{j}(s_{i+1}-1)\right)} 
\end{split}
\end{equation}
with 
\begin{equation}
	Q_m(n):=\prod_{j=m}^n\frac{w'_{\sigma'(j)}+1}{w_{\sigma(j)}+1}, 
	\qquad P_m(n):=\prod_{j=m}^n\left(\frac{w'_{\sigma'(j)}}{w_{\sigma(j)}}\right)^{j}
\end{equation}
for $1\le m\le n\le N$.
Here we set $s_{k+1}:=N+1$.

We now re-write the last two sums in~\eqref{eq:aux_2016_10_16_02} for fixed $k$. 
Given $s_1, \cdots, s_k$ and $\sigma, \sigma'$, let 
$I_i=\sigma(\{s_i,\cdots,s_{i+1}-1\})$ and $I_i'=\sigma'(\{s_i,\cdots,s_{i+1}-1\})$ for $0\le i\le k$, 
where we set $s_0:=1$. 
With these notations, 
\begin{equation}
\begin{split}
	Q_{s_i}(s_{i+1}-1) = \frac{\prod_{j'\in I_i'} (w'_{j'}+1)}{\prod_{j\in I_i}(w_j+1)}.
\end{split}
\end{equation}
Consider the restriction of the permutation $\sigma$ on $\{s_i, \cdots, s_{i+1}-1\}$. 
This gives rise to a bijection $\sigma_i$ from $\{1, \cdots, |I_i|\}$ to $I_i$ by setting $\sigma_i(\ell)=\sigma(s_i+\ell-1)$. 
Similarly we obtain a bijection $\sigma'_i$ from $\sigma$. 
Then we have 
\begin{equation}
\begin{split}
	Q_{s_i+j}(s_{i+1}-1) =  \prod_{\ell=j}^{|I_i|} \left(\frac{w'_{\sigma'_i(\ell)}+1}{w_{\sigma_i(\ell)}+1}\right)
\end{split}
\end{equation}
for $j=1, \cdots, |I_i|$, and
\begin{equation}
\begin{split}
	P_{s_i}(s_{i+1}-1) = \left[ \prod_{j=1}^{|I_i|} \left(\frac{w'_{\sigma_i'(j)}}{w_{\sigma_i(j)}}\right)^{j}  \right]
	\left( \frac{\prod_{j'\in I_i'} w'_{j'}}{\prod_{j\in I_i} w_j}\right)^{|I_0|+\cdots+|I_{i-1}|}.
\end{split}
\end{equation}
We now reorganize the summations over $s_1, \cdots, s_k$ and $\sigma, \sigma'$ in~\eqref{eq:aux_2016_10_16_02} as follows. 
We first decide on two partitions $I_0, \cdots, I_k$ and $I'_0, \cdots, I'_k$ of $\{1, \cdots, N\}$ satisfying $|I_i|=|I_i'|\neq 0$ for all $i$.
Then for each $I_i$, we consider a bijection $\sigma_i$ from $\{1, \cdots, |I_i|\}$ to $I_i$ and a bijection $\sigma'_i$ from $\{1, \cdots, |I_i|\}$ to $I_i$. 
The collection of $\sigma_0, \cdots, \sigma_k$ is equivalent to a permutation $\sigma$. 
Note that the sign  becomes
\beq
	\sign(\sigma)= (-1)^{\#(I_0,\cdots,I_k)} \prod_{i=0}^k \sign(\sigma_i)
\eeq
where 
\begin{equation}
	\#(I_0,\cdots,I_k):=
	\left| \{ (m,n)\in I_i\times I_j : m<n, \, i>j\}\right|.
\end{equation}
It is easy to see that summing over the partitions and then over the bijections $\sigma_i$ and $\sigma_i'$ is equivalent to summing  
over $s_1, \cdots, s_k$ and $\sigma, \sigma'$ in~\eqref{eq:aux_2016_10_16_02}. 
Hence we obtain 
\begin{equation}
\label{eq:aux_2016_10_17_02}
\begin{split}
	&H_a(W;W')= 	\sum_{k=0}^{N-1}  \sum_{\cI, \cI'}  (-1)^{\#(\cI)+\#(\cI')} 
 	\left[ \prod_{i=1}^k \left( \frac{\prod_{j'\in I_i'} w'_{j'}}{\prod_{j\in I_i} w_j}\right)^{|I_0|+\cdots+|I_{i-1}|} \right]
	\PQS(I_0,I'_0)\prod_{i=1}^k\PQ(I_i,I'_i)
\end{split}
\end{equation}
where the second sum is over two partitions $\cI=(I_0, \cdots, I_k)$ and $\cI'=(I'_0, \cdots, I'_k)$ of $\{1, \cdots, N\}$ such that $|I_i|=|I'_i|\neq 0$ for all $0\le i\le k$.
The function $\PQ(I;I')$ on two subsets of $\{1, \cdots, N\}$ of equal cardinality is defined by 
\begin{equation}
\label{eq:aux_2016_12_23_01}
\begin{split}
	\PQ(I,I')&= \left[ \sum_{\substack{	\sigma:\{1,\cdots, |I|\}\to I\\
		\sigma':\{1,\cdots,|I'|\}\to I' }}
	\sign(\sigma\sigma') \frac{P(\sigma, \sigma')}
{\prod_{j=2}^{|I|}\left(1-Q_j(\sigma, \sigma')\right)} \right] 
	\frac{Q(I,I')^{a+L-N-1}}
{1-Q(I,I')^{-1}} .
\end{split}
\end{equation}
Here the sum is over bijections $\sigma$ and $\sigma'$, and 
\begin{equation}
\begin{split}
	P(\sigma, \sigma')  = \prod_{\ell=1}^{|I|} \left(\frac{w'_{\sigma'(\ell)}}{w_{\sigma(\ell)}}\right)^{\ell}, 
	\quad
	Q_{j}(\sigma, \sigma') =  \prod_{\ell=j}^{|I|} \left(\frac{w'_{\sigma'(\ell)}+1}{w_{\sigma(\ell)}+1}\right), 
	\quad 
	Q(I,I')= \frac{\prod_{\ell'\in I'} (w'_{\ell'}+1)}{\prod_{\ell\in I}(w_\ell+1)} .
\end{split}
\end{equation}
Note that $Q_1(\sigma, \sigma')$ does not depend on $\sigma, \sigma'$ and is equal to $Q(I, I')$. 
On the other hand, 
\begin{equation}
\label{eq:aux_2016_12_23_01_2}
\begin{split}
	\PQS(I,I')&= \left[ \sum_{\substack{	\sigma:\{1,\cdots, |I|\}\to I\\
		\sigma':\{1,\cdots,|I'|\}\to I' }}
	\sign(\sigma\sigma') \frac{P(\sigma, \sigma')}
{\prod_{j=2}^{|I|}\left(1-Q_j(\sigma, \sigma')\right)} \right] 
{Q(I,I')^{a-1}}.
\end{split}
\end{equation}

\subsection{Step 2} 

We simplify $\PQS(I,I')$ and $\PQ(I,I')$.
They have a common sum. 
We claim that this sum is equal to 
\begin{equation}
\label{eq:aux_2016_10_18_06}
	\left(\frac{\prod_{i'\in I'} w'_{i'}}{\prod_{i\in I}w_i}\right) \left(\prod_{i\in I}(w_i+1)-\prod_{i'\in I'} (w'_{i'}+1)\right)
\det\left[\frac{1}{w_i-w'_{i'}}\right]_{i\in I,i'\in I'}.
\end{equation}
This implies that 
\begin{equation}
\label{eq:aux_2016_12_25_01}
\begin{split}
	\PQS(I,I')&=\left(\frac{\prod_{i'\in I'} (w'_{i'}+1)}{\prod_{i\in I}(w_i+1)}\right)^{a-1}\left(\frac{\prod_{i'\in I'} w'_{i'}}{\prod_{i\in I}w_i}\right)\left(\prod_{i\in I}(w_i+1)-\prod_{i'\in I'} (w'_{i'}+1)\right)
\det\left[\frac{1}{w_i-w'_{i'}}\right]_{i\in I,i'\in I'}
\end{split}
\end{equation}
and
\begin{equation}
\label{eq:aux_2016_12_25_01_2}
\begin{split}
	\PQ(I,I')&=-\frac{\left(\prod_{i'\in I'} (w'_{i'}+1)\right)^{a+L-N}}{\left(\prod_{i\in I}(w_i+1)\right)^{a+L-N-1}} \left(\frac{\prod_{i'\in I'} w'_{i'}}{\prod_{i\in I}w_i}\right)\det\left[\frac{1}{w_i-w'_{i'}}\right]_{i\in I,i'\in I'}.
\end{split}
\end{equation}

We now prove~\eqref{eq:aux_2016_10_18_06}. This follows from the next lemma. 
This lemma assumes that $I=I'=\{1, \cdots, n\}$ but the general case is obtained if we re-label the indices.

\begin{lm} \label{lm:identity_01}
For complex numbers $w_i$ and $w'_i$, $i=1, \cdots, n$, 
\begin{equation}
\label{eq:identity_01}
	\sum_{\sigma,\sigma'\in S_n}
	 \frac{\sign(\sigma\sigma') \displaystyle \prod_{i=1}^n \Big(\dfrac{w'_{\sigma'(i)}}{w_{\sigma(i)}}\Big)^{i-1}}{\displaystyle \prod_{j=2}^n\Big(1-\displaystyle \prod_{i=j}^n  \dfrac{w'_{\sigma'(i)}+1}{w_{\sigma(i)}+1}  \Big)}
	 =\left(\prod_{j=1}^n(w_j+1)-\prod_{j=1}^n (w'_{j}+1)\right)
\det\left[\frac{1}{w_i-w'_{i'}}\right]_{i,i'=1}^n.
\end{equation}
\end{lm}

\begin{proof} 
We use an induction in $n$. It is direct to check that the identity holds for $n=1$.

Let $n\ge 2$ and 
assume that the identity holds for index $n-1$. We now prove the identity for index $n$. 
We consider the sum according to different values of $\sigma(1)$ and $\sigma'(1)$. 
Setting $\sigma(1)=l$ and $\sigma'(1)=l'$, and then renaming the shifted version of the rest of $\sigma$ and $\sigma'$ by again $\sigma$ and $\sigma'$ respectively, 
the left-hand side of~\eqref{eq:identity_01} is equal to 
\begin{equation*} 
\begin{split}
	\sum_{\ell,\ell'=1}^n (-1)^{\ell+\ell'}
	\frac{\dfrac{w_\ell}{w'_{\ell'}}\prod_{k=1}^n\bigg(\dfrac{w'_k}{w_k}\bigg)}{1-\dfrac{w_\ell+1}{w'_{\ell'}+1}\prod_{k=1}^n \bigg( \dfrac{w'_k+1}{w_k+1} \bigg) }
	\left[ \sum_{\substack{\sigma:\{1,\cdots,n-1\}\to \{1,\cdots,n\}\setminus\{\ell\}\\ \sigma':\{1,\cdots,n-1\}\to \{1,\cdots,n\}\setminus\{\ell'\}}}
		\frac{\sign(\sigma)\sign(\sigma')\prod_{i=2}^{n-1}\bigg(\dfrac{w'_{\sigma'(i)}}{w_{\sigma(i)}}\bigg)^{i-1}}{\prod_{j=2}^{n-1}\bigg(1-\prod_{i=j}^{n-1}\dfrac{w'_{\sigma'(i)}+1}{w_{\sigma(i)}+1}\bigg)} \right] .
\end{split}
\end{equation*}
From the induction hypothesis, the bracket term is equal to 
\begin{equation*}
\begin{split}
	\Big(\prod_{\substack{j=1\\ j\neq \ell}}^n(w_j+1)-\prod_{\substack{j=1\\ j\neq \ell'}}^n (w'_{j}+1)\Big)
	\det\left[\frac{1}{w_i-w'_{i'}}\right]_{\substack{1\le i,i'\le n\\ i\ne \ell,i'\ne \ell'}}.
\end{split}
\end{equation*}
Hence the sum is equal to 
\begin{equation*}
\begin{split}
	& \sum_{\ell,\ell'=1}^n (-1)^{\ell+\ell'}
	\Big( \prod_{\substack{1\le k'\le n\\ k'\ne \ell'}}w'_{k'} \Big) 
	\Big( \prod_{\substack{1\le k\le n\\ k\ne \ell}}\dfrac{w_k+1}{w_k} \Big) 
	\det\left[\frac{1}{w_i-w'_{i'}}\right]_{\substack{1\le i,i'\le n\\ i\ne \ell,i'\ne \ell'}} .
\end{split}
\end{equation*}
This is same as 
\begin{equation*}
\begin{split}
	\left( \prod_{k=1}^n\frac{w'_k(w_k+1)}{w_k} \right)
	\left( \sum_{i,i'=1}^n   (-1)^{\ell+\ell'}  \frac{w_\ell}{(w_\ell+1)w'_{\ell'}}\det\left[\frac{1}{w_i-w'_{i'}}\right]_{\substack{1\le i,i'\le n\\ i\ne \ell,i'\ne \ell'}} \right).
\end{split}
\end{equation*}
Lemma~\ref{lem:sumCauchyq} below implies that the sum in the second parentheses is equal to 
\begin{equation*}
\begin{split}
	\left(\prod_{k=1}^n\frac{w_k}{w'_k(w_k+1)}\right) \left(\prod_{i=1}^n(w_i+1)-\prod_{i=1}^n(w'_i+1)\right) \det\left[\frac{1}{w_i-w'_{i'}}\right]_{i,i'=1}^n.
\end{split}
\end{equation*}
This completes the proof. 
\end{proof}

\begin{lm} \label{lem:sumCauchyq}
For distinct complex numbers $x_1, \cdots, x_n$ and $y_1, \cdots, y_n$, let $C$ be the $n\times n$ Cauchy matrix with entries $\frac{1}{x_i-y_j}$. 
Let $C_{\ell,k}$ be the matrix obtained from $C$ by removing row $\ell$ and column $k$. 
Define the functions  
\beq
	A(z)= \prod_{i=1}^n (z-x_i), \qquad B(z)=\prod_{i=1}^n (z-y_i). 
\eeq
Then
\beq \label{eq:Cauchyratosu}
	\sum_{\ell,k=1}^n (-1)^{l+k} \frac{x_\ell}{(x_\ell+1) y_k} \det \left[C_{\ell,k}\right]
	 = \frac{A(0)}{B(0)} \left( 1-\frac{B(-1)}{A(-1)} \right)  \det \left[C \right] .
\eeq
\end{lm}

\begin{proof}
Recall  the Cauchy determinant formula, 
\beqq	
	\det\left[ C \right] = \frac{\prod_{i<j}(x_i-x_j)(y_j-y_i)}{\prod_{i,j} (x_i-y_j)}.
\eeqq
Note that $C_{\ell,k}$ is also a Cauchy matrix. 
Hence we find that 
\beq \label{eq:ratioCst}
	\frac{\det\left[C_{\ell,k}\right]}{\det \left[C \right]} = (-1)^{\ell+k+1}\frac{B(x_\ell)A(y_{k})}{(x_\ell-y_{k})A'(x_\ell)B'(y_{k})}.
\eeq

Let $f$ and $g$ be meromorphic functions with finitely many poles and consider the double integral
\beq \label{eq:dintl}
	\frac1{(2\pi \ii)^2} \oint_{|\xi|=r_2} \oint_{|z|=r_1} f(z)g(\xi) \frac{B(z)A(\xi)}{(z-\xi)A(z)B(\xi)}  \dd z\dd\xi .
\eeq
Here we take $r_1<r_2$  large so that the poles of $\frac{f(z)}{A(z)}$ are inside $|z|<r_1$ and the poles of $\frac{g(\xi)}{B(\xi)}$ are inside $|\xi|<r_2$. 
Consider $f(z)= \frac{z}{z+1}$ and $g(\xi)=\frac1{\xi}$. 
By changing the order of integration and noting that the integrand is $O(\xi^{-2})$ as $\xi\to\infty$ for fixed $z$, the double integral is zero by taking the $\xi$ contour to infinity. 
On the other hand, we also evaluate the double integral by residue Calculus. 
For given $\xi$,  the integral in $z$ is equal to 
\beqq
	\frac{A(\xi)B(-1)}{\xi(\xi+1)B(\xi)A(-1)} 
	+ \sum_{\ell=1}^n  \frac{x_\ell A(\xi)B(x_\ell)}{(x_\ell+1)\xi(x_\ell-\xi)B(\xi)A'(x_\ell)}.
\eeqq
The first term is $O(\xi^{-2})$ and hence the integral with respect to $\xi$ is zero. 
The integral of the second term is, by residue Calculus, 
\beqq
	\left[ \sum_{\ell=1}^n  \frac{B(x_\ell)}{(x_\ell+1)A'(x_\ell)}  \right] \frac{A(0)}{B(0)}
	 + \sum_{\ell,k=1}^n\frac{x_\ell B(x_\ell)A(y_k)}{(x_\ell+1) y_k(x_\ell-y_k)A'(x_\ell)B'(y_k)} .
\eeqq
Since the double integral is zero, the above expression is zero. 
Now the sum inside the bracket can be simplified to $1-\frac{B(-1)}{A(-1)}$ by considering the integral of $\frac{B(z)}{(z+1)A(z)}$.
Therefore, we obtain
\beqq
	\sum_{\ell,k=1}^n\frac{x_\ell B(x_\ell)A(y_k)}{(x_\ell+1) y_k(x_\ell-y_k)A'(x_\ell)B'(y_k)} 
	 = - \left( 1-\frac{B(-1)}{A(-1)} \right) \frac{A(0)}{B(0)}.
\eeqq
Using~\eqref{eq:ratioCst}, we obtain the lemma. 
\end{proof}

\begin{rmk}
Note that using the rank one property and  the Hadamard's formula for the inverse of a matrix, 
\beq \label{eq:Cran1pr}
\bes
	\det\left[ \frac1{x_i-y_j}+ f(x_i)g(y_j) \right]_{i,j=1}^n
	&= \det[C] \left(1+ \sum_{\ell, k=1}^n g(y_k) (C^{-1})_{k,\ell} f(x_\ell) \right) \\
	&= \det[C] + \sum_{\ell, k=1}^n (-1)^{\ell+k}  \det[ C_{\ell, k}] f(x_\ell) g(y_k) . 
\end{split}
\eeq
Hence,  the above lemma implies that 
\beq
	\det\left[ \frac1{x_i-y_j}+ \frac{x_i}{(x_i+1)y_j} \right]_{i,j=1}^n 
	= \left( 1 + \frac{A(0)}{B(0)}  -\frac{A(0)B(-1)}{B(0)A(-1)} \right)    \det\left[ \frac1{x_i-y_j} \right]_{i,j=1}^n . 
\eeq
\end{rmk}

We will need the following variation of the above lemma in the next subsection. 

\begin{lm} \label{lem:sumCauchyq2}
Using the same notations as Lemma~\ref{lem:sumCauchyq}, 
\beq \label{eq:Cauchyratosu2}
	\det\left[ \frac1{x_i-y_j}+ \frac{u}{x_j} \right]_{i,j=1}^n 
	= \left(1 + u \left( 1-  \frac{B(0)}{A(0)} \right) \right) \det \left[C \right] . 
\eeq
\end{lm}

\begin{proof}
The proof is similar as the previous lemma using $f(z)=\frac{u}{z}$ and $g(\xi)=1$ in~\eqref{eq:dintl} instead. We also use~\eqref{eq:Cran1pr}. 
\end{proof}

\subsection{Step 3}

We insert the formulas~\eqref{eq:aux_2016_12_25_01} and~\eqref{eq:aux_2016_12_25_01_2} into~\eqref{eq:aux_2016_10_17_02},
and then reorganize the sum as follows. 
For the partitions $\cI=(I_0, \cdots, I_k)$ and $\cI'=(I'_0, \cdots, I'_k)$, we consider the first parts $I_0$ and $I'_0$ separately. 
Set $\ccI=(I_1, \cdots, I_k)$ and $\ccI'=(I'_1, \cdots, I'_k)$. 
Note that 
\beqq
	\#(\cI)= \#(I_0, I_0^c)+ \#(\ccI), \qquad \#(\cI')= \#(I'_0, (I'_0)^c)+ \#(\ccI')
\eeqq
where $I_0^c=\{1, \cdots, N\}\setminus I_0$ and $(I'_0)^c=\{1, \cdots, N\}\setminus I'_0$. 
Then~\eqref{eq:aux_2016_10_17_02} can be written as
\begin{equation}
\label{eq:aux_2016_12_25_02}
\begin{split}
	H_a(W;W')=
	\sum_{\substack{I_0,I'_0\subseteq\{1,\cdots,N\} \\ |I_0|=|I'_0|\ne 0}}
	 (-1)^{\#(I_0,I_0^c)+\#(I'_0,(I'_0)^c)}	 \PQS(I_0,I'_0)
	\left[ \frac{\left(\prod_{i'\in (I'_0)^c}(w'_{i'}+1)\right)^{a+L-N}}
	{\left(\prod_{i\in I_0^c}(w_i+1)\right)^{a+L-N-1}} \right] S(I_0, I'_0)
\end{split}
\end{equation}
where
\begin{equation}
\label{eq:aux_2016_12_25_02_2}
\begin{split}
	S(I_0, I'_0)=\sum_{k=1}^{N-|I_0|} \sum_{\ccI, \ccI' }  (-1)^{k+\#(\ccI)+\#(\ccI')}
		\prod_{j=1}^k  \left[  \left( \frac{\prod_{i'\in I'_j}w'_{i'}}{\prod_{i\in I_j}w_i} 
		\right)^{|I_0|+\cdots+|I_{j-1}|+1} 
		\det \left[\frac{1}{w_i-w'_{i'}} \right]_{i\in I_j,i'\in I'_j} \right] .
\end{split}
\end{equation}
Here the second sum in~\eqref{eq:aux_2016_12_25_02_2} is over all partitions $\ccI=(I_1, \cdots, I_k)$ of $\{1, \cdots, N\}\setminus I_0$ and partitions $\ccI'=(I'_1, \cdots, I'_k)$ of $\{1, \cdots, N\}\setminus I'_0$ satisfying $|I_i|=|I'_i|\neq 0$ for all $1\le i\le k$. 
When $|I_0|=N$, we understand that $S(I_0, I'_0)=1$. 
The following lemma simplifies $S(I_0, I_0')$.

\begin{lm}\label{lm:identity_02}
Let $n\ge 1$. Let  $w_i$ and $w'_i$, $1\le i\le n$, be complex numbers. Then 
\begin{equation} \label{eq:identity_02}
	\begin{split}
	&\sum_{k=1}^{n}
	\sum_{\cJ, \cJ'} (-1)^{k+\#(\cJ)+\#(\cJ')} 
	\prod_{j=1}^k \left[ \left(\frac{\prod_{i'\in J'_j} w'_{i'}}{\prod_{i\in J_j} w_i}\right)^{|J_1|+\cdots+|J_{j-1}|+1}  
	\det\left[\frac{1}{w_i-w'_{i'}}\right]_{{i\in J_j, i'\in J'_j}} \right] \\
	&= \left[ \prod_{j=1}^n\left(\frac{ w'_{j}}{ w_j}\right)^{n} \right] \det\left[\frac{1}{-w_i+w'_{i'}}\right]_{1\le i,i'\le n} .
	\end{split}
\end{equation}
Here the second sum is over all partitions $\cJ=(J_1, \cdots, J_k)$ and $\cJ'=(J'_1, \cdots, J'_k)$ of $\{1, \cdots, n\}$ such that 
$|J_i|=|J'_i|\neq 0$ for all $1\le i\le k$. 
\end{lm}

\begin{proof}
We use an induction in $n$. When $n=1$ the identity is trivial. 
Now we prove that the identity holds for index $n$ assuming that it  holds for all indices less than $n$. 
We fix $J_1= J$ and $J'_1=J'$ with $|J|=|J'|\neq 0$ and then apply the induction hypothesis on the remaining sum with index $n-|J|$.
Then, the left-hand side of~\eqref{eq:identity_02} is equal to
\beq \label{eq:yyz}
\bes
	-\sum_{\substack{J,J'\subseteq\{1,\cdots,n\}\\ |J|=|J'|\ne 0}}
	& (-1)^{\#(J,J^c)+\#(J',(J')^c)}
	\left( \frac{\prod_{i'\in J'}w'_{i'}}{\prod_{i\in J}w_i} \right) 
	\left(\frac{\prod_{i'\in (J')^c} w'_{i'}}{\prod_{i\in J^c} w_i}\right)^{n} \\
	&\qquad \times \det\left[\frac{1}{w_i-w'_{i'}}\right]_{{i\in J,i'\in J'}}
	\det\left[\frac{1}{-w_i+w'_{i'}}\right]_{{i\in J^c, i'\in (J')^c}}.
\end{split}
\eeq
Here the minus sign comes from $(-1)^k$ where $k$, the number of parts of partition $\cJ$, is reduced by $1$ when we apply the induction hypothesis since we remove $J_1$ in the counting. 
Now, note that the right-hand side of~\eqref{eq:identity_02} is the summand of~\eqref{eq:yyz} when $|J|=|J'|=0$.
Hence, after dividing by $\prod_{j=1}^n\big(\frac{ w'_{j}}{ w_j}\big)^{n}$, 
we find that~\eqref{eq:identity_02} is obtained if we show that 
\begin{equation*} \label{eq:aux_2016_10_20_01}
\begin{split}
	\sum_{\substack{J,J'\subseteq\{1,\cdots,n\}\\ |J|=|J'|}} 
	& (-1)^{\#(J,J^c)+\#(J',(J')^c)}
	\left( \frac{\prod_{i'\in J'}w'_{i'}}{\prod_{i\in J}w_i} \right)^{1-n}  \det\left[\frac{1}{w_i-w'_{i'}}\right]_{{i\in J,i'\in J'}}
	\det\left[\frac{1}{-w_i+w'_{i'}}\right]_{{i\in J^c, i'\in (J')^c}} 
\end{split}
\end{equation*}
is equal to $0$. 
Here the sum is over all $J,J'\subseteq\{1,\cdots,n\}$ with $|J|=|J'|$, including the case when $|J|=|J'|=0$.
By Lemma~\ref{lem:sumofAB} below, this sum is equal to
\beqq
	\det\left[\frac{(w_{i})^{n-1}}{(w'_{i'})^{n-1}}\frac{1}{w_i-w'_{i'}}+\frac{1}{-w_i+w'_{i'}}\right]_{i,i'=1}^n.
\eeqq
Hence it is enough to show that 
\beqq
	\det\left[\frac{(w'_{i'})^{n-1}-w_i^{n-1}}{w'_{i'}-w_i}\right]_{i,i'=1}^n=0.
\eeqq
Using $\frac{a^{n-1}-b^{n-1}}{a-b}=a^{n-2}+a^{n-3}b+\cdots + b^{n-2}$, this last determinant is a sum of the determinants of form
\beqq
	\det\left[(w'_{i'})^{\alpha_i}w_i^{n-2-\alpha_i}\right]_{i,i'=1}^{n} = \det\left[(w'_{i'})^{\alpha_i}\right]_{i,i'=1}^{n} \prod_{i=1}^n w_i^{n-2-\alpha_i} 
\eeqq
for some $\alpha_i\in\{0,1,\cdots,n-2\}$.
The last determinant is zero since at least two rows are equal. This completes the proof. 
\end{proof}

\begin{lm}\label{lem:sumofAB}
For two $n\times n$ matrices $A$ and $B$, 
\beq
	\sum_{\substack{J,J'\subseteq\{1,\cdots,n\}\\ |J|=|J'|}} 
	(-1)^{\#(J,J^c)+\#(J',(J')^c)}
	\det\left[A(i,i')\right]_{{i\in J,i'\in J'}}
	\det\left[B(i,i')\right]_{{i\in J^c, i'\in (J')^c}}
	= \det[A+B]
\eeq
where the sum is over all subsets $J$ and $J'$ of equal size, including the case when $J=J'=\emptyset$. 
\end{lm}

\begin{proof}
It is direct to check by expanding all determinants by sums using definition. 
\end{proof}

We apply Lemma \ref{lm:identity_02} to $S(I_0, I'_0)$. Note that the power of the products of $w'_{i'}$ and $w_i$ in~\eqref{eq:aux_2016_12_25_02_2} is $|I_0|+|I_1|+\cdots+|I_k|+1$ while the corresponding products have power $|J_1|+ \cdots+|J_k|+1$ in the lemma. 
We find that 
\begin{equation}
\begin{split}
	S(I_0, I'_0)=\left( \frac{\prod_{i'\in (I'_0)^c}w'_{i'}}{\prod_{i\in I_0^c}w_i}  \right)^{N} 
	\det \left[\frac{1}{w_i-w'_{i'}} \right]_{i\in I_0^c,i'\in (I'_0)^c}  .
\end{split}
\end{equation}
Using this formula and also the formula~\eqref{eq:aux_2016_12_25_01} for $\PQS(I_0, I'_0)$, the equation~\eqref{eq:aux_2016_12_25_02} becomes
\begin{equation}
\label{eq:aux_2016_12_25_03}
\begin{split}
	&H_a(W;W') \\ 
&=
	\sum_{\substack{I_0,I'_0\subseteq\{1,\cdots,N\}\\ |I_0|=|I'_0|\ne 0}} 
	(-1)^{\#(I_0,I_0^c)+\#(I'_0,(I'_0)^c)}
	\left[ \frac{\prod_{i'\in I'_0}w'_{i'}(w'_{i'}+1)^{a}}
{\prod_{i\in I_0}w_i(w_i+1)^{a-1}} \right]
	\left[  \frac{\prod_{i'\in (I'_0)^c} (w'_{i'})^{N} (w'_{i'}+1)^{a+L-N}}
{\prod_{i\in I_0^c} w_i^{N} (w_i+1)^{a+L-N-1}} \right]
\\
&
	\qquad \qquad \times 
	\left(\frac{\prod_{i\in I_0}(w_i+1)}{\prod_{i'\in I'_0}(w'_{i'}+1)}-1\right)	
\det\left[\frac{1}{w_i-w'_{i'}}\right]_{i\in I_0,i'\in I'_0}
\det\left[\frac{1}{-w_i+w'_{i'}}\right]_{i\in I_0^c,i'\in (I'_0)^c}					 
\end{split}
\end{equation}
where we added the zero partitions $|I_0|=|I'_0|=0$ in the sum since the summand is zero in that case.

\subsection{Step 4}

We evaluate the sum~\eqref{eq:aux_2016_12_25_03}. 
So far $w_i$ and $w'_{i'}$ were any complex numbers. 
We now use the assumption that $w_i\in\roots_{z}$ and $w'_{i'}\in \roots_{z'}$. 
This means that 
\beqq
	w_i^N (w_i+1)^{L-N}= z^L, 
	\qquad 
	(w'_{i'})^N (w'_{i'}+1)^{L-N} = (z')^L
\eeqq
for all $i$ and $i'$, and hence the second square bracket in~\eqref{eq:aux_2016_12_25_03} can be simplified. 
We then separate the formula of $H_a(W;W')$ into two terms using the two terms in the big parentheses. 
The second term, which comes from the term $1$ in the big parenthesis, can be written as 
\beqq
\bes
	\left[ \prod_{j=1}^N\frac{(w'_j+1)^{a}}{(w_j+1)^{a-1}} \right] 
	& \left\{ \sum_{\substack{I_0,I'_0\subseteq\{1,\cdots,N\}\\ |I_0|=|I'_0|}} 
	 (-1)^{\#(I_0,I_0^c)+\#(I'_0,(I'_0)^c)}
	\left[ \frac{\prod_{i'\in I'_0}w'_{i'}}
{\prod_{i\in I_0}w_i} \right]
	\left[  \frac{\prod_{i'\in (I'_0)^c} (z')^L}
{\prod_{i\in I_0^c} z^L} \right] \right.
\\
&
	\qquad \qquad \times \left.
\det\left[\frac{1}{w_i-w'_{i'}}\right]_{i\in I_0,i'\in I'_0}
\det\left[\frac{1}{-w_i+w'_{i'}}\right]_{i\in I_0^c,i'\in (I'_0)^c} \right\}.
\end{split}
\eeqq
The sum is, after inserting the terms in the products into the determinants, of form in Lemma~\ref{lem:sumofAB} above, and hence is equal to 
\beqq
\bes
	\det\left[\frac{w'_{i'}}{w_i(w_i-w'_{i'})}-\left(\frac{z'}{z}\right)^L\frac{1}{w_i-w'_{i'}}\right]_{1\le i, i'\le N} .
\end{split}
\eeqq
The first term is also similar, and we obtain
\beq \label{eq:yyr1}
\bes
	H_a(W;W') 
	= 	\left[ \prod_{j=1}^N\frac{(w'_j+1)^{a}}{(w_j+1)^{a-1}} \right] (D_1-D_2)
\end{split}
\eeq
where, for $k=1,2$, 
\beq
\bes
	D_k= \det\left[\frac{ \frac{f_k(w'_{i'})}{f_k(w_i)}}{w_i-w'_{i'}}- \frac{q}{w_i-w'_{i'}}\right]_{i,i'=1}^N
\end{split}
\eeq
with
\beq
	f_1(w)= \frac{w}{w+1}, \qquad f_2(w)=w, \qquad 	q=\left(\frac{z'}{z}\right)^L . 
\eeq
The entries of the determinants $D_k$ are of form 
\beqq
	\frac{\frac{f(w'_{i'})}{f(w_i)}-1}{w_i-w'_{i'}} + \frac{1-q}{w_i-w'_{i'}}
	=(1-q) \left(  \frac{1}{w_i-w'_{i'}} + \frac{f(w'_{i'})-f(w_i)}{(1-q)f(w_i)(w_i-w'_{i'})}  \right) 
\eeqq
where $f$ represents either $f_1$ or $f_2$. 
For $f(w)=f_2(w)=w$, the last term is $-\frac{1}{(1-q)w_i}$ and 
\beq \label{eq:yyr2}
\bes
	D_2	
	&= (1-q)^N \det\left[\frac{1}{w_i-w'_{i'}} - \frac{1}{(1-q)w_i}\right] \\
	&= (1-q)^N \bigg( 1- \frac1{1-q} \bigg(1- \prod_{i=1}^N \frac{w'_{i'}}{w_i} \bigg) \bigg) 
	\det\left[\frac{1}{w_i-w'_{i'}}\right]_{i,i'=1}^N
\end{split}
\eeq 
where we used Lemma~\ref{lem:sumCauchyq2} for the second equality. 
On the other hand, 
\beqq
	D_1	
	= (1-q)^N \det\left[\frac{1}{w_i-w'_{i'}} - \frac{1}{(1-q)w_i(w'_{i'}+1)}\right] .
\eeqq
We may evaluate this determinant by finding a variation of Lemma~\ref{lem:sumCauchyq2}. Alternatively, we set $w_i=\frac1{x_i-1}$ and $w'_{i'}=\frac1{y_{i'}-1}$, and use Lemma~\ref{lem:sumCauchyq2} to obtain 
\beq \label{eq:yyr3}
\bes
	D_1	
	&= \frac{(1-q)^N}{\prod_{i=1}^N (w_iw'_{i})} \det\left[\frac{1}{y_{i'}-x_i} - \frac{1}{(1-q) y_{i'}}\right] \\
	&= (1-q)^N \bigg(1-\frac{1}{1-q}\bigg(1-\prod_{j=1}^N\frac{(w_j+1)w'_{j}}{w_j(w'_{j}+1)}\bigg)\bigg)
	\det\left[\frac{1}{w_i-w'_{i'}}\right]_{i,i'=1}^N.
\end{split}
\eeq 
From~\eqref{eq:yyr1},~\eqref{eq:yyr2}, and~\eqref{eq:yyr3}, we obtain 
\beqq
	H_a(W;W')=\bigg(1-\bigg(\frac{z'}{z}\bigg)^L\bigg)^{N-1} 
	\left[ \prod_{j=1}^N\frac{w'_j}{w_j} \right] 
	\bigg(\prod_{j=1}^N\frac{(w'_j+1)^{a-1}}{(w_j+1)^{a-2}} - \prod_{j=1}^N\frac{(w'_j+1)^{a}}{(w_j+1)^{a-1}}\bigg)
	\det\left[\frac{1}{w_i-w'_{i'}}\right]_{i,i'=1}^N .
\eeqq
This completes the proof of Proposition~\ref{prop:01}.

\section{Proof of Theorem~\ref{thm:htfntasy}}\label{sec:pfastp}

In this section, we prove Theorem~\ref{thm:htfntasy} starting from Corollary~\ref{cor:pTASEPstep}.

\subsection{Parameters} \label{sec:ptpamr}

We evaluate the limit of 
\begin{equation} 
    \label{eq:asgoaol}
	\prob\left( \height(\mr{p}_1) \ge b_1, \cdots,  \height(\mr{p}_m) \ge b_m \right) 
\end{equation}
as $L\to\infty$, 
where $\mr{p}_j= s_j \mr e_1+t_j\mr e_c = \ell_j \mr e_1+t_j\mr e_2$ with $s_j=\ell_j-(1-2\rho)t_j$,
\begin{equation}
\label{eq:scaling_t}
	t_j= \tau_j\frac{L^{3/2}}{\sqrt{\rho(1-\rho)}}   + O(L),
\end{equation}
\begin{equation}
\label{eq:scaling_ell}
	\ell_j = (1-2\rho) t_j + \gamma_j L + O(L^{1/2}) ,
\end{equation}
and 
\begin{equation}
\label{eq:scaling_b}
	b_j=2\rho(1-\rho)t_j+(1-2\rho)\ell_j-2x_j\rho^{1/2}(1-\rho)^{1/2} L^{1/2}
\end{equation}
for fixed $0<\tau_1< \cdots< \tau_m$, $\gamma_j\in[0,1]$ and $x_j\in\realR$. 
The analysis needs small changes if $\tau_i=\tau_{i+1}$ and $x_i<x_{i+1}$ for some $i$. 
We will comment such changes in the footnotes throughout the analysis.

It is tedious but easy to check in the analysis that 
the convergence is uniform for the parameters $\tau_j, \gamma_j$, and $x_j$ in compact sets. 
In order to keep the writing light, we do not keep track of the uniformity.

The height function and the particle location have the following relation:
\beq \label{eq:xandhet2}
	\text{$\height(\ell \mr e_1+t \mr e_2) \ge b$ if and only if $\xx_{N-\frac{b-\ell}2+1}(t)\ge \ell+1$}
\eeq
for all $b-2\ell\in 2\intZ$ satisfying $b\ge \height(\ell  \mr e_1)$. 
Hence, 
\begin{equation}
 \label{eq:asgoaol2}
 	\prob\left( \height(\mr{p}_1) \ge b_1, \cdots,  \height(\mr{p}_m) \ge b_m \right) 
 	= \prob \left( \xx_{k_1}(t_1)\ge a_1, \cdots, \xx_{k_m}(t_m)\ge a_m \right) 
\end{equation}
where $t_j$ are same as above,
\begin{equation}
\label{eq:scaling_a}
a_j=\ell_j+1,
\end{equation}
and
\begin{equation}
\label{eq:scaling_k}
k_j=N-\frac{b_j-\ell_j}{2}+1.
\end{equation}
From Corollary~\ref{cor:pTASEPstep}, we thus need to evaluate the limit of   
\begin{equation}
	\begin{split}
	&\prob \left( \xx_{k_1}(t_1)\ge a_1, \cdots, \xx_{k_m}(t_m)\ge a_m \right)
	= \oint\cdots\oint \const(\bzz) \gdet(\bzz)  
	\ddbar{\bz_1}\cdots\ddbar{\bz_m} 
	\end{split}
\end{equation}
where the contours are nest circles satisfying $0<|\bz_m|<\cdots< |\bz_1|<\rr_0$. 
Here we use the notation $\bz_j$, instead of $z_j$. The notation $z_j$ will be reserved for the rescaled parameter of $\bz_j$ in the asymptotic analysis.

In \cite{Baik-Liu16} and \cite{Liu16}, we analyzed the case when $m=1$. 
The case when $m\ge 2$ is similar and we follows the same strategy. 
The results and analysis of the above two papers are used heavily in this section.

We change the variables $\bz_j$ to $z_j$ defined by 
\begin{equation} 
\label{eq:ztozz}
z_j= (-1)^N \frac{\bz_j^{L}}{\rr_0^L}.
\end{equation}
Then $\frac{\dd \bz_j}{\bz_j} = \frac{\dd z_j}{L z_j}$, but the simple closed circle of $\bz_j$ becomes a circle with multiplicity $L$ of $z_j$. 
Note that $\const(\bzz)$ and $\gdet(\bzz)$ depend on $\bz_j$ through the set of the roots of the equation $w^N(w+1)^{L-N}=\bz_j^L$, which is unchanged if $\bz_j$ is changed $\bz_j e^{2\pi \ii k/L}$ for integer $k$. 
Hence the above integral is same as 
\begin{equation}
\label{eq:aux_2017_04_03_01}
	\begin{split}
	&\prob \left( \xx_{k_1}(t_1)\ge a_1, \cdots, \xx_{k_m}(t_m)\ge a_m \right)
	= \oint\cdots\oint \const(\bzz) \gdet(\bzz) \ddbar{z_1}\cdots\ddbar{z_m} 
	\end{split}
\end{equation}
where the integrals are nested simple circles such that $0<|z_m|<\cdots<|z_1|<1$ on the contours, and 
for given $z_j$, $\bz_j$ is any one of the $L$ roots of the equation $\bz_j^L=  (-1)^L \rr_0^L z_j$.

\begin{rmk} \label{rmk:zord}
We note that the analysis in this section does not depend on the fact that $|z_i|$ are ordered in a particular way.  
It is easy to check in each step of the analysis that we only require that $|z_i|$ are distinct. 
Hence, if we fix $0<r_m<\cdots<r_1<1$, and $|z_j|=r_{\sigma(j)}$ for a permutation $\sigma$ of $1,\cdots, m$,  then the asymptotics and error estimates of $\const(\bzz)$, $\gdet(\bzz)$ and the integral on the right hand side of~\eqref{eq:aux_2017_04_03_01} still hold (with different constants $\epsilon, c, C$'s in the error estimates). This fact is used in an important way in the next section; see the proof of Theorem~\ref{thm:pofEpm}. 
\end{rmk}

\subsection{Asymptotics of $\const(\bzz)$}

\begin{lm}[\cite{Liu16}] \label{lem:katzlt}
Let $N=N_L$ be a sequence of integers such that $\rho=\rho_L:= \frac{N}{L}$ stays in a compact subset of $(0,1)$ for all $L$. 
Fix $\epsilon\in (0,1/2)$.  
Fix a complex number $z$ such that $0<|z|<1$ and let
$\bz$ satisfy $\bz^L=(-1)^N\rr_0^L z$.  
Assume that for fixed $\tau>0$, $\gamma\in [0,1)$, and $x\in \realR$, 
\beq \label{eq:akfm}
	t= \frac{\tau}{\sqrt{\rho(1-\rho)}}L^{3/2}+ O(L), 
	\qquad 
	a=\ell+1, \qquad k= N-\frac{b-\ell}2+1
\eeq
where 
\beq \label{eq:blfm}
	\ell= (1-2\rho)t + \gamma L + O(L^{1/2}), 
	\qquad
	b= 2\rho(1-2\rho)t + (1-2\rho)\ell- 2x\rho^{1/2}(1-\rho)^{1/2} L^{1/2} . 
\eeq
Recall the definition of $\CA(\bz)$ in~\eqref{eq:CAbutE}:
\begin{equation}
	\CA(\bz)=\CA(\bz;a,k,t):=\prod_{u\in\rootsL_{\bz}}(-u)^{k-N-1}\prod_{v\in\rootsR_{\bz}}(v+1)^{-a+k-N}e^{tv}.
\end{equation}
Then 
\beq \label{eq:limit_h2102}
	\CA(\bz)	= e^{x A_1(z) +\tau A_2(z)} \left(1+O(N^{\epsilon-1/2})\right)
\eeq
where $A_1(z)$ and $A_2(z)$ are polylogarithm functions defined in~\eqref{eq:A1andA2}.
\end{lm}
\begin{proof}
This follows from (4.25) and (4.26) of Section 4.3 of \cite{Liu16}: we have $\CA(\bz;a,k,t)=\mathcal{C}^{(2)}_{N,2}(\bz;k-N,a)$ in terms of the notation used in \cite{Liu16}. 
(Note that the limit there is $e^{\tau^{1/3}x A_1(z) +\tau A_2(z)}$, but in this paper, we use the scale of the height function so that $\tau^{1/3}x$ is changed to $x$.) 
It is easy to check that the conditions (4.1), (4.2), and (4.3) in \cite{Liu16} are satisfied, and hence we obtain the above lemma. 
The analysis was based on the integral representation of $\log \CA(\bz)$ (see (4.28) of \cite{Liu16}) and applying the method of the steepest-descent. 
Indeed, by the residue theorem, ($w=0$ is the case we need)
\beq
	\log \bigg( \prod_{u\in\rootsL_{\bz}}(w-u) \bigg) = (L-N)\log (w+1)
	+ \frac{L\bz^L}{2\pi \ii} \oint \frac{(u+\rho)\log (w-u)}{u(u+1)q_{\bz}(u)} \dd u
\eeq
for $\Re(w)>-\rho$ where the contour is a simple closed curve in $\Re(u)<-\rho$ which contains $\rootsL_{\bz}$ inside. We can find a similar formula for the product of $v$. 
\end{proof}

\begin{lm} \label{lem:lemma82} 
Let $N$, $L$, and $\epsilon$ as in the previous Lemma. 
Fix complex numbers $z_1, z_2$ such that $0<|z_1|, |z_2|<1$. Then for $\bz_1, \bz_2$ satisfying $\bz_1^L=(-1)^N\rr_0^L z_1$,  $\bz_2^L=(-1)^N\rr_0^L z_2$, 
\begin{equation}
\label{eq:aux_2017_03_26_08}
	\frac{\prod_{v\in\rootsR_{\bz_1}}(v+1)^{L-N}\prod_{u\in\rootsL_{\bz_2}}(-u)^N} {\Delta\left(\rootsR_{\bz_1};\rootsL_{\bz_2}\right)}
	= e^{2B(z_1,z_2)} (1+O(N^{\epsilon-1/2})) 
\end{equation}
where $B(z_1, z_2)$ is defined in~\eqref{eq:aux_2017_03_26_10}.
\end{lm}

\begin{proof}
The case when $\bz_1=\bz_2$ was obtained in \cite{Baik-Liu16}: 
take the square of (8.17) in Lemma 8.2, whose proof was given in Section 9.2.
The case of general $\bz_1$ and $\bz_2$ is almost the same, which we outline here.  
From the residue theorem, 
\begin{equation}\label{eq:digra}
	\begin{split}
	&(L-N)\sum_{v\in\rootsR_{\bz_1}}\log(v+1)+N\sum_{u\in\rootsL_{\bz_2}}\log(-u)-\sum_{v\in\rootsR_{\bz_1}}\sum_{u\in\rootsL_{\bz_2}}\log(v-u)\\
	&=\bz_1^L\bz_2^L\int_{-\rho_\mathfrak{a}-\ii\infty}^{-\rho_\mathfrak{a}+\ii\infty}\int_{-\rho_\mathfrak{b}-\ii\infty}^{-\rho_\mathfrak{b}+\ii\infty}
	\log\bigg(\frac{N^{1/2}(v-u)}{\rho\sqrt{1-\rho}}\bigg) \frac{L(v+\rho)}{v(v+1)q_{\bz_1}(v)}\frac{L(u+\rho)}{u(u+1)q_{\bz_2}(u)}\ddbarr{u}\ddbarr{v}
	\end{split}
\end{equation}
where the contours for $u$ and $v$ are two vertical lines $\Re(u)=-\rho_\mathfrak{a}:=-\rho+\mathfrak{a}\rho\sqrt{1-\rho}N^{-1/2}$ and $\Re(v)=-\rho_\mathfrak{b}:=-\rho+\mathfrak{b}\rho\sqrt{1-\rho}N^{-1/2}$ with constants $\mathfrak{a}$ and $\mathfrak{b}$ satisfying $-\sqrt{-\log|z_2|}<\mathfrak{a}<0<\mathfrak{b}<\sqrt{-\log|z_1|}$. 
This formula is similar to (9.26) in \cite{Baik-Liu16} and the proof is also similar.  
We divide the double integral~\eqref{eq:digra} into two parts: $|u+\rho_\mathfrak{a}|,|v+\rho_\mathfrak{a}|\le N^{\epsilon/3}$ and the rest. 
It is a direct to check that the formula~\eqref{eq:digra} where the integral is restricted to the $|u+\rho_\mathfrak{a}|,|v+\rho_\mathfrak{a}|\le N^{\epsilon/3}$ is equal to 
	\begin{equation}
	\begin{split}
	z_1z_2\int_{\Re\xi=\mathfrak{a}}\int_{\Re\zeta=\mathfrak{b}}\frac{\xi\zeta\log(\zeta-\xi)}{\left(e^{-\xi^2/2}-z_2\right)\left(e^{-\zeta^2/2}-z_1\right)}\ddbarr{\xi}\ddbarr{\zeta}\left(1+O(N^{\epsilon-1/2})\right)
	\end{split}
	\end{equation}
where we used the change of variables $u=-\rho+\xi\rho\sqrt{1-\rho}N^{-1/2}$ and $v=-\rho+\zeta\rho\sqrt{1-\rho}N^{-1/2}$.
On the other hand, if $|u+\rho_\mathfrak{a}|\ge N^{\epsilon/3}$ or $|v+\rho_\mathfrak{b}|\ge N^{\epsilon/3}$, the integrand decays exponentially fast and hence the integral in these regions is exponentially small. See Section 9.2 of \cite{Baik-Liu16} for more discussions. 
From~\eqref{eq:aux_2017_03_26_10}, we obtain~\eqref{eq:aux_2017_03_26_08}.
\end{proof}

Recalling the definitions~\eqref{eq:forBs} and~\eqref{eq:costcdefaq}, and using Lemma~\ref{lem:katzlt} and Lemma~\ref{lem:lemma82}, we find that  
\begin{equation}
\begin{split}
	&\const(\bzz)
	= \ccc(\bfz) \big(1+O(N^{\epsilon-1/2})\big) .
\end{split}
\end{equation}

\subsection{Analysis of $\gdet(\bzz)$} \label{sec:M}

We can obtain the limit of $\gdet(\bzz)$ using either the Fredholm determinant or its series expansion~\eqref{eq:aux_2017_03_22_01}. 
Both are suitable for the asymptotic analysis. 
Here we use the series expansion. 
From~\eqref{eq:aux_2017_03_22_01}, we have 
\begin{equation} 
\begin{split}
	\gdet(\bzz) =\sum_{\bn\in (\intZ_{\ge 0})^m} \frac{1}{(\bn !)^2}
\gdet_{\bn}(\bzz)
	\end{split}
\end{equation}
with
\begin{equation} \label{eq:aux_2017_03_22_01aa}
	\gdet_{\bn}(\bzz)
	=(-1)^{|\bn|}\sum_{\substack{\UU^{(\ell)}\in (\rootsL_{\bz_\ell})^{n_\ell} \\ \VV^{(\ell)}\in (\rootsR_{\bz_\ell})^{n_\ell} \\ l=1,\cdots,m}} \det\left[\KsL(w_i,w'_j) \right]_{i,j=1}^{|\bn|}
	\det\left[\KsR(w'_i,w_j)\right]_{i,j=1}^{|\bn|}
\end{equation}
where $\bUU=(\UU^{(1)}, \cdots, \UU^{(m)})$, $\bVV= (\VV^{(1)}, \cdots, \VV^{(m)})$ with $\UU^{(\ell)}=(u_1^{(\ell)}, \cdots, u_{n_\ell}^{(\ell)})$, $\VV^{(\ell)}=(v_1^{(\ell)}, \cdots, v_{n_\ell}^{(\ell)})$, and 
\begin{equation} 
w_i=\begin{dcases}
	u_k^{(\ell)}& \text{if $i=n_1 +\cdots+n_{\ell-1}+k$ for some $k\le n_\ell$ with odd integer $\ell$,}\\
v_k^{(\ell)}& \text{if $i=n_1 +\cdots+n_{\ell-1}+k$ for some $k\le n_\ell$ with even integer $\ell$,}
\end{dcases}
\end{equation}
and
\begin{equation}
	w'_i=\begin{dcases}
v_k^{(\ell)}& \text{if $i=n_1+\cdots+n_{\ell-1}+k$ for some $k\le n_\ell$ with odd integer $\ell$,}\\
u_k^{(\ell)}& \text{if $i=n_1+\cdots+n_{\ell-1}+k$ for some $k\le n_\ell$ with even integer $\ell$.}
\end{dcases}
\end{equation}
We prove the convergence of this series to the series~\eqref{eq:Fdmddser} with~\eqref{eq:Frdse1}.

\subsubsection{Strategy}

To be able to cite easily later, let us state the following simple fact. 

\begin{lm} \label{lem:pfofDlimbasic}  
Suppose that 
\begin{enumerate}[(A)]
\item for each fixed $\bn$, $\gdet_\bn (\bzz)\to \gdetlm_\bn(\bfz)$ as $L\to \infty$, and
\item there is a constant $C>0$ such that $|\gdet_\bn (\bzz)|\le C^{|\bn|}$ for all $\bn$ and for all large enough $L$.
\end{enumerate}
Then $\gdet (\bzz)\to \gdetlm(\bfz)$ as $L\to \infty$.
\end{lm}

\begin{proof}
It follows from the dominated convergence theorem.
\end{proof}

We are going to show that the conditions (A) and (B) are satisfied. 
To be precise, we need to show that (A) and (B) hold locally uniformly in $\bfz$ so that the conclusion holds locally uniformly. 
The local uniformity is easy to check throughout the proof. 
To make the presentation light, we do not state the local uniformity explicitly and, instead, state only the pointwise convergence for each $\bfz$ in the rest of this section.

Let us discuss the strategy of verifying the conditions (A) and (B). 
Suppose for a moment that $N/L=\rho$ is fixed. 
If $\bz^L=(-1)^N \rr_0^L z$ for a fixed $z$, then 
the contour $|w^N(w+1)|^{L-N}=|\bz|^L$, on which the roots of $q_{\bz}(w)$ lie, converges to the self-crossing contour  $|w^\rho(w+1)^{1-\rho}|=\rr_0$ as $L,N \to\infty$ since $|\bz|\to \rr_0$ (see Figure~\ref{fig:rootsfinite}.)
The point of self-intersection is $w=-\rho$. 
For large $L,N$ and the parameters satisfying the conditions of Theorem~\ref{thm:htfntasy}, 
it turns out that the main contribution to the sum $\gdet_{\bn}(\bzz)$ comes from the points $\UU^{(\ell)}$ and $\VV^{(\ell)}$ near the self-crossing point $w=-\rho$. 
As $L\to \infty$, $q_{\bz}(w)$ has more and more roots.
We scale the roots near the point $w=-\rho$ in such a way that the distances between the scaled roots are $O(1)$: this is achieved if\footnote{The roots are less dense near the self-crossing point $w=-\rho$ than elsewhere: A typical distance between two neighboring roots is $O(N^{-1})$, but near $w=-\rho$ this distance is $O(N^{-1/2})$.} we take $N^{1/2}(w+\rho)\mapsto w$. 
Under this scale, for each $w$ on the set $\rootsL_{\bz}\cup \rootsR_\bz$ (which depends on $L$ and $N$; see~\eqref{eq:RLfinds}) in a neighborhood of the point $-\rho$, there is a unique point $\zeta$ on the set $\inodesL_z\cup \inodesR_z$ (which is independent of $L$ and $N$; see~\eqref{eq:LRlmtdefs}), and vice versa. 
See Lemma~\ref{lm:limit_nodes} for the precise statement. 
We show that $\KsL$ and $\KsR$ converges to $\lmKL$ and $\lmKR$ pointwise for the points near $-\rho$. 
We then estimate the kernels when one of the argument is away from $-\rho$. 
These two calculations are enough to prove the conditions (A) and (B). 
See Lemma~\ref{lem:K1Kw2estmfo} for the precise statement. 
The fact that we only assume that $N/L=\rho$ is in a compact subset of $(0,1)$ for all $L$ does not change the analysis. 

The following lemma is from \cite{Baik-Liu16}. 

\begin{lm}[Lemma 8.1 of \cite{Baik-Liu16}] \label{lm:limit_nodes}
Let $0<\epsilon<1/2$. 
Fix a complex number $z$ satisfying $0<|z|<1$ and let $\bz$ be a complex number satisfying $\bz^L=(-1)^N\rr_0^L z$. 
Recall the definitions of the sets $\rootsL_{\bz}$ and $\rootsR_\bz$ in~\eqref{eq:RLfinds}, and the sets $\inodesL_z$ and $\inodesR_z$ in~\eqref{eq:LRlmtdefs}. 
Let $\mathcal{M}_{N,\LL}$ be the map from $\rootsL_{\bz}\cap\left\{w:|w+\rho|\le \rho\sqrt{1-\rho}N^{\epsilon/4-1/2}\right\}$ to $\inodesL_{z}$ defined by
\begin{equation} \label{eq:lmlimit_nodesdc}
	\mathcal{M}_{N,\LL}(w)=\xi \quad
	\mbox{where }\xi\in\inodesL_{z} \mbox{ and } 
	\left|\xi-\frac{(w+\rho)N^{1/2}}{\rho\sqrt{1-\rho}}\right|\le N^{3\epsilon/4-1/2}\log N.	
\end{equation}
Then for all large enough $N$ the following holds. 
\begin{enumerate}[(a)]
		\item $\mathcal{M}_{N,\LL}$ is well-defined.
		\item $\mathcal{M}_{N,\LL}$ is injective.
		\item Setting $\inodesL_{z}^{(c)}:=\inodesL_{z}\cap\{\xi:|\xi|\le c\}$ for $c>0$, we have 
		\begin{equation}
		\inodesL_{z}^{(N^{\epsilon/4}-1)}\subseteq I(\mathcal{M}_{N,\LL})\subseteq \inodesL_{z}^{(N^{\epsilon/4}+1)}
\end{equation}
where $I(\mathcal{M}_{N,\LL}):=\mathcal{M}_{N,\LL}\left(\rootsL_{\bz}\cap \{w :|w+\rho|\le \rho\sqrt{1-\rho}N^{\epsilon/4-1/2}\}\right)$, the image of the map $\mathcal{M}_{N,\LL}$.
\end{enumerate}
If we define the mapping $\mathcal{M}_{N,\RR}$ in the same way but replace $\rootsL_{\bz}$ and $\inodesL_{z}$ by $\rootsR_{\bz}$ and $\inodesR_{z}$ respectively, the same results hold for $\mathcal{M}_{N,\RR}$.
\end{lm}

Before we go further, we conjugate the kernels which leaves the determinants in~\eqref{eq:aux_2017_03_22_01aa} unchanged. 
We use the conjugated kernels in the rest of this subsection. 
This makes the necessary convergences possible. 
For $w\in\roots_{z_i}\cap\mrootL$ and $w'\in\roots_{z_j}\cap\mrootR$, we change~\eqref{eq:aux_2017_3_17_01} to  
\begin{equation} 
 \begin{split}
 	\tKsL(w,w') & 
	= -\left(\delta_i(j)+\delta_i\left(j+(-1)^i\right)\right) 
	\frac{\jac(w) \sqrt{\hfs_i(w)}\sqrt{ \hfs_j(w')} (\sH_{z_i}(w))^2 }{\sH_{z_{i-(-1)^{i}}}(w) \sH_{z_{j-(-1)^{j}}}(w')(w-w')} \sQL(j)
\end{split} \end{equation}
and for $w\in\roots_{z_i}\cap\mrootR$ and $w'\in\roots_{z_j}\cap\mrootL$,
we change~\eqref{eq:aux_2017_3_17_01_0002} to  
\begin{equation} \label{eq:aux_2017_3_17_01_0003}
 \begin{split}
 	\tKsR(w,w') & 
	=- \left(\delta_i(j)+\delta_i\left(j-(-1)^i\right)\right)
 	\frac{\jac(w) \sqrt{\hfs_i(w)}\sqrt{\hfs_j(w')} (\sH_{z_i}(w))^2 }{\sH_{z_{i+(-1)^{i}}}(w) \sH_{z_{j+(-1)^{j}}}(w') (w-w')} \sQR(j)
\end{split} \end{equation}
where we set (note the change from $\fs_j(w)$ in~\eqref{eq:def_fs}) 
\begin{equation} \label{eq:newfhfs}
	  	\hfs_j(w) 
		=\begin{dcases}
		 \frac{\fftn_j(w)\fftn_{j-1}(-\rho)}{\fftn_{j-1}(w)\fftn_j(-\rho)} 
		 &\text{for $\Re(w)<-\rho$,}\\ 
		 \frac{\fftn_{j-1}(w)\fftn_j(-\rho)}{\fftn_j(w)\fftn_{j-1}(-\rho)} 
		 &\text{for $\Re(w)>-\rho$} 
	  	 \end{dcases}
\end{equation}
with (see~\eqref{eq:def_fftn})
\begin{equation} \label{eq:CapfF}
	\fftn_{i}(w):= w^{-k_i+N+1} (w+1)^{-a_i+k_i-N} e^{t_iw}. 
\end{equation}
We multiplied by $-1$ to remove a minus sign in the limit.
The square-root function is defined as follows: for a complex number $w=re^{\ii \theta}$ with $r\ge 0$ and $-\pi<\theta\le \pi$, we set $\sqrt{w}= r^{1/2} e^{\ii \theta/2}$. 
Note that $\sqrt{w}$ is not continuous when $w$ is a negative real number, and hence $\tKsL(w,w')$ and $\tKsR(w,w')$ may be discontinuous for some $w$ and $w'$. 
However, the product of $\det\big[\tKsL(w_i,w'_j) \big]_{i,j=1}^{|\bn|}$ and $\det\big[\tKsR(w'_i,w_j)\big]_{i,j=1}^{|\bn|}$ is still continuous at the branch cuts since each of $\sqrt{\hfs_i(w)}$ is multiplied twice. 
We also note that the change from $\fs_j$ to $\hfs_j$ has the effect of conjugating the matrices in the determinants $\det\big[\KsL(w_i,w'_j) \big]_{i,j=1}^{|\bn|}$ and $\det\big[\KsR(w'_i,w_j)\big]_{i,j=1}^{|\bn|}$ and multiplying both by $(-1)^{|\bn|}$.
Hence $\gdet_{\bn}(\bzz)$ in~\eqref{eq:aux_2017_03_22_01aa} is unchanged if we replace
$\KsL$ and $\KsR$ by $\tKsL$ and $\tKsR$.

We also conjugate the limiting kernels. 
For $\zeta\in(\inodesL_{z_i}\cup \inodesR_{z_i})\cap \SSSL$ and $\zeta'\in(\inodesL_{z_j}\cup \inodesR_{z_j})\cap \SSSR$ for some $i,j$, we change~\eqref{eq:limitkeL} to 
\begin{equation} 
\begin{split}
	\tlmKL(\zeta,\zeta') 
	&=(\delta_i(j)+\delta_i(j+(-1)^i)) 
	 \frac{ \sqrt{\sfs_i(\zeta)}\sqrt{\sfs_j(\zeta')} e^{2\hftn(\zeta, z_i)-\hftn(\zeta, z_{i-(-1)^{i}}) -\hftn(\zeta', z_{j-(-1)^{j}})}}{\zeta (\zeta-\zeta')}\QWL(j) ,
\end{split}
\end{equation}
and for $\zeta\in(\inodesL_{z_i}\cup \inodesR_{z_i})\cap \SSSR$ and $\zeta'\in(\inodesL_{z_j}\cup \inodesR_{z_j})\cap\SSSL$, we change~\eqref{eq:limitkeR} to 
\begin{equation} 
\begin{split}
	\tlmKR(\zeta,\zeta') 
	&=(\delta_i(j)+\delta_i(j-(-1)^i)) 
	\frac{\sqrt{\sfs_i(\zeta)}\sqrt{\sfs_j(\zeta')} e^{2\hftn(\zeta, z_i)-\hftn(\zeta, z_{i+(-1)^{i}}) -\hftn(\zeta', z_{j+(-1)^{j}})}}{\zeta (\zeta-\zeta')} \QWR(j).
\end{split}
\end{equation}
Recall the definition of $\sfs_i$ in~\eqref{eq:def_sfs}.

\bigskip

Next lemma shows how we prove the strategy mentioned above. 

\begin{lm} \label{lem:K1Kw2estmfo}
Fix $0<\epsilon<1/(1+2m)$. 
Let
\beq
	\isk=\isk_N:= \bigg\{ w\in \complexC: |w+\rho| \le \frac{\rho\sqrt{1-\rho}}{N^{1/2-\epsilon/4}}\bigg \}
\eeq
be a disk centered at $-\rho$.
Suppose the following: 
\begin{enumerate}[(i)]
\item We have   
\beq
\label{eq:aux_2017_07_01_01}
	|\tKsL(w,w')| = |\tlmKL(\zeta, \zeta')| +O(N^{\epsilon-1/2}\log N), 
	\quad 
	|\tKsR(w',w)| = |\tlmKR(\zeta', \zeta)|+O(N^{\epsilon-1/2}\log N)
\eeq
as $L\to \infty$, uniformly for $w\in \mrootL\cap \isk$ and $w'\in \mrootR\cap\isk$, 
where $\zeta\in\SSSL, \zeta'\in\SSSR$ are the unique points corresponding to $w, w'$ under either the map $\mathcal{M}_{N,\LL}$ or $\mathcal{M}_{N,\RR}$ in Lemma~\ref{lm:limit_nodes}. 
\item 
For each $\bn$, 
\beqq 
	\det\left[\tKsL(w_i, w'_j)\right]_{i,j=1}^{|\bn|} \to \det\left[\tlmKL(\sw_i, \sw'_j)\right]_{i,j=1}^{|\bn|}, 
	\quad
	\det\left[\tKsR(w_i', w_j)\right]_{i,j=1}^{|\bn|} \to \det\left[\tlmKR(\sw'_i, \sw_j)\right]_{i,j=1}^{|\bn|}
\eeqq
as $L\to \infty$, where for $w_i\in \mrootL\cap \isk$ and $w_i'\in \mrootR\cap\isk$, 
$\zeta_i\in\SSSL$ and $\zeta_i'\in\SSSR$ are the unique points corresponding to $w, w'$ under either the map $\mathcal{M}_{N,\LL}$ or $\mathcal{M}_{N,\RR}$ in Lemma~\ref{lm:limit_nodes}. 
\item 
There are positive constants $c$ and $\alpha$ such that 
\beqq
	|\tKsL(w,w')|= O(e^{-cN^{\alpha }} ), 
	\qquad
	|\tKsR(w',w)|= O(e^{-cN^{\alpha }} )
\eeqq
as $L\to\infty$, uniformly for $w\in \mrootL\cap \isk^c$ and $w'\in \mrootR$, and also for $w'\in \mrootR\cap \isk^c$ and $w\in \mrootL$. 
\end{enumerate}
Then the conditions (A) and (B) in Lemma~\ref{lem:pfofDlimbasic} hold, and therefore, $\gdet (\bzz)\to \gdetlm(\bfz)$.
\end{lm}

If the absolute values in (i) are removed, then (i) will imply (ii). 
However, due to the discontinuity of the branch cuts of the square-root functions, $\tKsL(w,w')$ may converge to $-\tlmKL(\zeta, \zeta')$ if the points are at the branch cuts. 
Nevertheless, the branch cuts do not affect the determinants as we discussed before. 
To emphasize this point, we state (ii) separately. 

\begin{proof}
It is direct to check that due to the term $\sfs_i$ (see~\eqref{eq:def_sfs}) the kernels\footnote{If $\tau_i=\tau_{i+1}$ and $x_i<x_{i+1}$ for some $i$, then we have $O(e^{-c_1|\zeta|})$ and $O(e^{-c_1|\zeta'|})$, which are enough for the analysis.}  $\tlmKL(\zeta, \zeta')=O(e^{-c_1 |\zeta|^3})$ for some positive constant $c_1$ as $|\zeta|\to \infty$ along $\zeta\in\SSSL$ uniformly for $\zeta'\in \SSSR$, and also $\tlmKL(\zeta, \zeta')=O(e^{-c_1 |\zeta'|^3})$ as $|\zeta'|\to \infty$ along $\zeta'\in\SSSR$ uniformly for $\zeta\in \SSSL$. 
There are similar estimates for $\tlmKR(\zeta', \zeta)$. 
This implies, in particular, that $\tlmKL(\zeta, \zeta')$ and $\tlmKR(\zeta', \zeta)$ are bounded for $\zeta\in \SSSL$ and $\zeta'\in \SSSR$. 

Since $\mrootL$ and $\mrootR$ has $O(N^m)$ number of points, the assumption (iii) implies that 
\beq \label{eq:fKsqsuarfeboundec}
	\sum_{w'\in\mrootR }|\tKsL(w,w')|^2 = O(e^{-\frac12cN^{\alpha\epsilon}} ), 
	\qquad
	\sum_{w'\in\mrootR }|\tKsR(w',w)|^2 = O(e^{-\frac12cN^{\alpha\epsilon}} ) 
\eeq
uniformly for $w\in \mrootL\cap \isk^c$, and similarly, 
\beq \label{eq:fKsqsuarfeboundec2}
	\sum_{w\in\mrootL }|\tKsL(w,w')|^2 = O(e^{-\frac12cN^{\alpha\epsilon}} ), 
	\qquad
	\sum_{w\in\mrootL }|\tKsR(w',w)|^2 = O(e^{-\frac12cN^{\alpha\epsilon}} ) 
\eeq
uniformly for $w'\in \mrootR\cap \isk^c$.

We now show that there is a positive constant $C_1$ such that 
\beq \label{eq:Ksquareboundec}
	\sum_{w\in \mrootL} \sqrt{ \sum_{w'\in\mrootR }|\tKsL(w,w')|^2}  \le C_1 
\eeq
for all large enough $L$. 
The inequality is obtained if 
\beqq
	\sum_{w\in \mrootL\cap\isk^c} \sqrt{ \sum_{w'\in\mrootR}|\tKsL(w,w')|^2},
	\quad
	\sum_{w\in \mrootL\cap\isk} \sqrt{ \sum_{w'\in\mrootR\cap\isk^c}|\tKsL(w,w')|^2},
	\quad
	\sum_{w\in \mrootL\cap\isk} \sqrt{ \sum_{w'\in\mrootR\cap\isk}|\tKsL(w,w')|^2}
\eeqq
are all $O(1)$. 
The first two terms are bounded from the assumption (iii) and the fact that $\mrootL$  and $\mrootR$ have $O(N^m)$ points. 
For the third term, we use the assumption (i).  
It is direct to check that there are $O(N^{m\epsilon})$ number of points in $\mrootL\cap\isk$ and $\mrootR\cap\isk$. 
Hence the the third term is bounded by 
\beqq
	2\sum_{\zeta\in \SSSL} \sqrt{ \sum_{\zeta'\in \SSSR}  |\tlmKL(\zeta, \zeta')|^2 } 
	+ O(N^{(m+1/2)\epsilon-1/2}\log N).
\eeqq
This is bounded since $|\tlmKL(\zeta, \zeta')|$ decays fast as $|\zeta|\to \infty$ or $|\zeta'|\to\infty$ on $\SSSL, \SSSR$ and  $\epsilon< 1/(1+2m)$. 
Hence we proved~\eqref{eq:Ksquareboundec}. 
Similarly, we have 
\beq \label{eq:Ksquareboundec2}
	\sum_{w'\in \mrootR} \sqrt{ \sum_{w\in\mrootL }|\tKsR(w',w)|^2} \le C_1.
\eeq

We now show that 
(B) in Lemma~\ref{lem:pfofDlimbasic} holds. Consider the formula~\eqref{eq:aux_2017_03_22_01aa} of $\gdet_{\bn}(\bzz)$.
As we mentioned before, we change $\KsL$ and $\KsR$ to $\tKsL$ and $\tKsR$ without changing the determinants.
From the Hadamard's inequality, for all different $w_j'$'s, 
\begin{equation} \label{eq:Hdmdi1}
\begin{split}
	\left|\det\left[\tKsL(w_i,w'_j)\right]_{i,j=1}^{|\bn|}\right|
	&\le \prod_{i=1}^{|\bn|} \sqrt{\sum_{j=1}^{|\bn|} |\tKsL(w_i,w'_j)|^2} 
	\le \prod_{i=1}^{|\bn|} \sqrt{\sum_{w'\in \mrootR} |\tKsL(w_i,w')|^2} 
\end{split}
\end{equation}
and, similarly,  for all different $w_j$'s,
\begin{equation} \label{eq:Hdmdi2}
\begin{split}
	\left|\det\left[ \tKsR(w'_i,w_j)\right]_{i,j=1}^{|\bn|}\right|
	&\le  \prod_{i=1}^{|\bn|}  \sqrt{\sum_{j=1}^{|\bn|} |\tKsR(w'_i,w_j)|^2} 
	\le  \prod_{i=1}^{|\bn|}  \sqrt{\sum_{w\in\mrootL}|\tKsR(w'_i,w)|^2} .
\end{split}
\end{equation}
Hence, 
\beqq
\bes
	|\gdet_\bn (\bzz)| 
	 &\le \sum_{\substack{w_i\in \mrootL \\ w_i'\in\mrootR \\ i=1,\cdots, |\bn|}} \left( \prod_{i=1}^{|\bn|} \sqrt{\sum_{w'\in \mrootR} |\tKsL(w_i,w')|^2}  \right)
	 \left( \prod_{i=1}^{|\bn|}  \sqrt{\sum_{w\in\mrootL}|\tKsR(w'_i,w)|^2} \right) \\
	 &= \left( \sum_{w\in \mrootL} \sqrt{\sum_{w'\in \mrootR} |\tKsL(w,w')|^2}  \right)^{|\bn|}
	 \left( \sum_{w'\in\mrootR}  \sqrt{\sum_{w\in\mrootL}|\tKsR(w' ,w)|^2} \right)^{|\bn|}.
\end{split}	
\eeqq
Using~\eqref{eq:Ksquareboundec} and~\eqref{eq:Ksquareboundec2}, we obtain (B) with $C=C_1^2$. 

We now prove that (A) in Lemma~\ref{lem:pfofDlimbasic} holds. 
Fix $\bn$. 
We divide the sum in the formula of $\gdet_{\bn}(\bzz)$ into two parts: the part that all $u_j^{(\ell)}, v_j^{(\ell)}$ are in $\isk$ and the rest. 
By the assumption (ii) and Lemma~\ref{lm:limit_nodes} (c), the first part converges, as $L\to \infty$, to 
\beqq
	(-1)^{|\bn|}\sum_{\substack{\UUlm^{(\ell)}\in (\inodesL_{z_\ell})^{n_\ell} \\ \VVlm^{(\ell)}\in (\inodesR_{z_\ell})^{n_\ell} \\ \ell=1,\cdots,m}}
	\det\left[\tlmKL(\sw_i, \sw'_j)\right]_{i,j=1}^{|\bn|}
	\det\left[\tlmKR(\sw'_i, \sw_j)\right]_{i,j=1}^{|\bn|}
\eeqq
which is equal to $\gdetlm_{\bn}(\bfz)$ in~\eqref{eq:Frdse1}. 
On the other hand, for the second part, note that~\eqref{eq:Ksquareboundec} and~\eqref{eq:Ksquareboundec2} imply, in particular, that there is a positive constant $C_2$ such that 
\beq \label{eq:Ksquareboundec001}
	\sqrt{ \sum_{w'\in\mrootR }|\tKsL(w,w')|^2}  \le C_2, 
	\qquad 
	\sqrt{ \sum_{w\in\mrootL }|\tKsR(w',w)|^2} \le C_2
\eeq
uniformly for $w\in \mrootL$ for the first inequality, and for $w'\in \mrootR$ for the second inequality. 
Now, by the Hadamard's inequality (see~\eqref{eq:Hdmdi1} and~\eqref{eq:Hdmdi2}) and the estimates~\eqref{eq:fKsqsuarfeboundec},~\eqref{eq:fKsqsuarfeboundec2}, and~\eqref{eq:Ksquareboundec001}, we find that
for the second part, 
\begin{equation}
\begin{split}
	\left|\det\left[\tKsL(w_i,w'_j)\right]_{i,j=1}^{|\bn|} \det\left[ \tKsR(w'_i,w_j)\right]_{i,j=1}^{|\bn|}\right|
	\le C_2^{2|\bn|} C_3 e^{-cN^{\alpha}} 
\end{split}
\end{equation}
for a positive constant $C_3$, 
since one of the variables is in $\isk^c$. 
Since there are only $O(N^m)$ points in $\mrootL$ and $\mrootR$, we find that the second part converges to zero. 
Hence, we obtain (A). 
\end{proof}

\subsubsection{Asymptotics of $\qr_\bz$, $\ql_{\bz}$, and $\hfs_j$}

In the remainder of this section, we verify the assumptions (i), (ii), and (ii) of Lemma~\ref{lem:K1Kw2estmfo}. 
The kernels contain $\ql_\bz(w)$, $\qr_\bz(w)$, and $\hfs_j$. 
We first find the asymptotics of these functions. 

The following asymptotic result was proved in \cite{Baik-Liu16}.

\begin{lm}[\cite{Baik-Liu16}] \label{lem:qnearcpw}
Let $z, z'$ be complex numbers satisfying $0<|z|, |z'|<1$. 
Let $\bz$ and $\bz'$ are two complex numbers such that 
$\bz^L=(-1)^N \rr_0^L z$ and $\bz'^L=(-1)^N \rr_0^L z'$.
For a complex number $\ww$, set 	
\begin{equation}
	w= -\rho+ \frac{\rho\sqrt{1-\rho}}{N^{1/2}} \ww.
\end{equation} 
There is a positive constant $C$ such that the following holds. 
\begin{enumerate}[(a)]
\item If $\Re(\ww)> c$ for some $c>0$, then, uniformly in $\ww$, 
	\begin{equation}
	\label{eq:qLforul}
	\ql_\bz(w)=e^{\hftn(\ww, z)}(1+O(N^{\epsilon/2-1}\log N)) \quad  \text{for $|\ww|\le N^{\epsilon/4}$}
	\end{equation}
and
	\begin{equation}
	e^{-CN^{-\epsilon/4}}\le |\ql_\bz(w)|\le e^{CN^{-\epsilon/4}} \quad \text{for $|\ww|\ge N^{\epsilon/4}$.}
	\end{equation}
\item If $\Re(\ww)< -c$ for some $c<0$, then, uniformly in $\ww$,  
		\begin{equation}
		\label{eq:qRRforul}
		\qr_\bz(w)=e^{\hftn(\ww, z)}(1+O(N^{\epsilon/2-1}\log N)) \quad  \text{for $|\ww|\le N^{\epsilon/4}$}
		\end{equation}
and
		\begin{equation}
		e^{-CN^{-\epsilon/4}}\le |\qr_\bz(w)|\le e^{CN^{-\epsilon/4}} \quad \text{for $|\ww|\ge N^{\epsilon/4}$.}
		\end{equation}
	\end{enumerate}
The errors are uniform for $z, z'$ in a compact subset of $0<|z|, |z'|<1$.
The function $\hftn$ is defined in~\eqref{eq:def_h_R} and~\eqref{eq:def_h_L}.
\end{lm}

\begin{proof}
The case when $|\ww|\le N^{\epsilon/4}$ is in Lemma 8.2 (a) and (b) in \cite{Baik-Liu16}, 
where we used the notations $q_{\bz,\LL}(w)= (w+1)^{L-N}\ql_\bz(w)$, $q_{\bz, \RR}(w)=w^N \qr_{\bz}(w)$,  
$\hftn_\LL(\ww, z)=\hftn(\ww,z)$ for $\ww<0$, and $\hftn_\RR(\ww, z)=\hftn(\ww,z)$ for $\ww>0$.

The upper bounds of $\ql_\bz(w)$ and $\qr_\bz(w)$ when $|\ww|\ge N^{\epsilon/4}$ were computed in the proof Lemma 8.4 (c) of the same paper:  (9.57) shows $|\qr_\bz(w)|=|\frac{q_{\bz, \RR}(w)}{w^N}|= O(e^{CN^{-\epsilon/4}})$. (There is a typo in that equation: the denominator on the left hand side should be $u^N$ instead of $u^{L-N}$.)
This upper bound was obtained from an upper bound of $\log |\qr_\bz(w)|$. 
Indeed, in the analysis between (9.52) and (9.57) in \cite{Baik-Liu16}, we first write (there is another typo in (9.52); $(L-N)\log (-u)$ should be $N\log(-u)$)
\begin{equation*}
	\log |\qr_\bz(w)|
	= \Re\bigg[ \sum_{v\in\rootsR_{z}}\log(-w+v)-N\log (-w)\bigg] 
	=\Re \left[ -L\bz^L\int_{-\rho+\ii\realR}\log\left(\frac{v-w}{-\rho-w}\right)\frac{(v+\rho)}{v(v+1)q_\bz(v)}\ddbarr{v} \right].
\end{equation*}
After estimating the integral in an elementary way, we obtain $|\log |\qr_\bz(w)||\le CN^{-\epsilon/4}$. 
This implies both the upper and the lower bounds, $e^{-N^{-\epsilon/4}}\le |\qr_\bz(w)|\le e^{N^{-\epsilon/4}}$. 
The estimate of $\ql_\bz(w)$ is similar.
\end{proof}

Now we consider $\hfs_j(w)$.   
The following result applies to  $\fftn_j(w)$ (recall~\eqref{eq:CapfF}), and hence $\hfs_j$,  for $w$ near $-\rho$. 

 \begin{lm}[\cite{Liu16}] \label{lem:gnnec}
 	Assume the same conditions for $k,t, a$ as Lemma~\ref{lem:katzlt}. 
 	Set 
 	\begin{equation}
 	\gn(w)= w^{-k+N+1} (w+1)^{-a+k-N} e^{tw}. 
 	\end{equation}
 	Then, for 
 	\begin{equation}
 	w= -\rho+ \frac{\rho\sqrt{1-\rho}}{N^{1/2}} \ww
 	\end{equation}
 	with $\ww=O(N^{\epsilon/4})$, we have 
 	\begin{equation} \label{eq:lecon}
 	\frac{\gn(w)}{\gn(-\rho)} = e^{x\ww +\frac12\gamma\ww^2 -\frac{1}{3}\tau \ww^3 } (1+O(N^{\epsilon-1/2})).
 	\end{equation}
 \end{lm}
 
 \begin{proof}
 	If we set $k'=k-N+1$ and $\ell'=a+2$, then $\gn(w)$ is same as $\tilde{g}_2(w)$ in (4.41) of \cite{Liu16} with $k'$ and $\ell'$ instead of $k$ and $\ell$ (and $\tau^{1/3}x$ replaced by $x$.)
 Since $\ell'=a+2=\ell+3$ satisfies the condition~\eqref{eq:blfm}, we see that 
 	$\frac{\gn(w)}{\gn(-\rho)}$ is equal to $g_2(w)$ in (4.40) of \cite{Liu16} with $j=0$. 
 	The asymptotics of $g_2(w)$ was obtained in (4.46) of \cite{Liu16}, which is same as~\eqref{eq:lecon}, 
 	under the conditions on $t$, $k'$, and $\ell'$ satisfying (4.1), (4.2), and (4.3) of \cite{Liu16}. 
 	From~\eqref{eq:akfm} and~\eqref{eq:blfm}, we find that these conditions are satisfied, and hence we obtain the lemma. 
 \end{proof} 
 
The next lemma is about $\hfs_j(w)$ when $w$  is away from $-\rho$.

 \begin{lm}
 	\label{lm:tail_1}
 	Suppose $x_j\in\realR,\gamma_j\in[0,1]$ and $\tau_j\in\realR_{\ge 0}$ are fixed constants, and assume that\footnote{If $\tau_{j-1}=\tau_{j}$ and $x_{j-1}<x_{j}$, then the errors change to $O(e^{-cN^{\epsilon/4}})$. In the proof, this change comes when we conclude the error term from~\eqref{eq:rhsfpr}.}  $\tau_1<\cdots<\tau_m$. 
Let $t_j$, $\ell_j$ and $b_j$ satisfy~\eqref{eq:scaling_t},~\eqref{eq:scaling_ell}, and~\eqref{eq:scaling_b}. Let
 	\begin{equation}
 	\label{eq:scaling_ak}
 	a_j=\ell_j+1,\qquad k_j=N-\frac{b_j-\ell_j}{2}+1
 	\end{equation}
as defined in~\eqref{eq:scaling_a} and~\eqref{eq:scaling_k}. 
Then, for $w\in\rootsL_{\bz}\cup \rootsR_{\bz}$ satisfying $|w+\rho|\ge \rho\sqrt{1-\rho}N^{\epsilon/4-1/2}$,  
\begin{equation} \label{eq:aux_2017_04_04_03}
 	\hfs_j(w)=O(e^{-cN^{3\epsilon/4}}).
\end{equation}
The  errors are uniform for $z$ in a compact subset of $0<|z|<1$. 
\end{lm}

\begin{proof}
The proof is similar to the first part of the proof of Lemma 8.4(c) in \cite{Baik-Liu16}, which is given in Section 9.4 of the same paper. 
We prove the case when $w\in \rootsL_{\bz}$. The case when $w\in \rootsR_{\bz}$ is similar. 
Recall from~\eqref{eq:newfhfs} and~\eqref{eq:CapfF} that for $w\in \rootsL_{\bz}$, $\hfs_j(w) = \frac{\fftn_j(w)\fftn_{j-1}(-\rho)}{\fftn_{j-1}(w)\fftn_j(-\rho)}$ and 
\beqq
	\frac{\fftn_j(w)}{\fftn_{j-1}(w)} =  w^{k_{j-1}-k_j} (w+1)^{-a_j+k_j+a_{j-1}-k_{j-1}} e^{(t_j-t_{j-1})w}. 
\eeqq 
We start with the following Claim. This is similar to a claim in Section 9.4 part (c). 
 	
\begin{claim}\label{claim:01}
Suppose $m$, $n$ and $\ell$ are positive integers and $\rho\in (0,1)$ satisfying
 		\begin{equation}
 		\label{eq:aux_2017_03_29_02}
 		-\frac{m}{\rho}+\frac{n}{1-\rho}+\ell\ge 0.
 		\end{equation}
 		Then the function
 		\begin{equation}
 		\label{eq:aux_2017_03_29_03}
 		\left|w^m(w+1)^ne^{\ell w}\right|
 		\end{equation}
increases as $\Re(w)$ increases along any fixed contour 
 		\begin{equation}
 		\label{eq:aux_2017_03_29_01}
 		\left|w^\rho(w+1)^{1-\rho}\right|=const.
 		\end{equation}
If~\eqref{eq:aux_2017_03_29_02} is not satisfied, then~\eqref{eq:aux_2017_03_29_03} increases as $\Re(w)$ increases along the part of the contour~\eqref{eq:aux_2017_03_29_01} satisfying
 		\begin{equation}
 		\label{eq:aux_2017_03_29_04}
 		|w+\rho|^2\ge -\frac{\rho(1-\rho)}{\ell}\left(-\frac{m}{\rho}+\frac{n}{1-\rho}+\ell\right).
 		\end{equation}
\end{claim}
	
 	\begin{proof}[Proof of Claim] 
 		Set 
 		\begin{equation*}
 		c=-\frac{\rho n-(1-\rho)m}{\rho \ell}.
 		\end{equation*}
 		We have $c\le 1-\rho$ from the condition~\eqref{eq:aux_2017_03_29_02}.
 		Note that
 		\begin{equation*}
 		\left|w^m(w+1)^ne^{\ell w}\right|=const\cdot \left|(w+1)^{n-\frac{1-\rho}{\rho}m}e^{\ell w}\right|=const\cdot\left|(w+1)^{-c}e^{w}\right|^{\ell}
 		\end{equation*}
 		by using the condition that $w$ is on the contour~\eqref{eq:aux_2017_03_29_01}. It is direct to check by parameterizing the contour and taking the derivatives that the function $\left|(w+1)^{-c}e^{w}\right|$ increases as $\Re(w)$ increase along the contour~\eqref{eq:aux_2017_03_29_01}: see Claim in Section 9.4 part (c) of \cite{Baik-Liu16}. 
If~\eqref{eq:aux_2017_03_29_02} is not satisfied, then $c>1-\rho$, and in this case, again it is direct to check that the function $\left|(w+1)^{-c}e^{w}\right|$ increases as $\Re(w)$ increases  along the contour~\eqref{eq:aux_2017_03_29_01} if $w$ is restricted to $|w+\rho|^2\ge \rho(c-1+\rho)$. 
The last condition is~\eqref{eq:aux_2017_03_29_04}.
\end{proof}

We continue the proof of Lemma~\ref{lm:tail_1}.	
We prove that the function $\left| \frac{\fftn_j(w)}{\fftn_{j-1}(w)} \right|$ increases as $\Re(w)$ increases along parts of the contour
 	\begin{equation}
 	\label{eq:aux_2017_03_29_06}
 	\left|w^{\rho}(w+1)^{1-\rho}\right|=\rr_0 |z|^{1/L}
 	\end{equation}
 	with the restriction
 	\begin{equation}
 	\label{eq:aux_2017_03_29_07}
 	|w+\rho|\ge \rho\sqrt{1-\rho}N^{\epsilon/4-1/2}.
 	\end{equation}
Recall that $\rootsL_{\bz}\cup \rootsR_\bz$ is a discrete subset of the contour~\eqref{eq:aux_2017_03_29_06}.
By Claim, it is sufficient to show that
    \begin{equation}
    \label{eq:aux_2017_03_29_13}
    \left(\rho\sqrt{1-\rho}N^{\epsilon/4-1/2}\right)^2 \ge 
    -\frac{\rho(1-\rho)}{t_j-t_{j-1}}\left(-\frac{-k_j+k_{j-1}}{\rho}+\frac{-a_j+k_j+a_{j-1}-k_{j-1}}{1-\rho}+t_j-t_{j-1}\right).
    \end{equation}
From~\eqref{eq:scaling_ak}, and~\eqref{eq:scaling_t},~\eqref{eq:scaling_ell} and~\eqref{eq:scaling_b}, the right-hand side is equal to
    \begin{equation*}
    -\frac{1}{2\rho(1-\rho)(t_j-t_{j-1})}\left((b_j-b_{j-1})-2\rho(1-\rho)(t_j-t_{j-1})-(1-2\rho)(\ell_j-\ell_{j-1})+O(t_j^{2/3})\right) .
    \end{equation*}
This is $O(N^{-1})$.  Thus~\eqref{eq:aux_2017_03_29_13} holds for sufficiently large $N$, and we obtain the monotonicity of $\left| \frac{\fftn_j(w)}{\fftn_{j-1}(w)} \right|$.
    
The monotonicity implies that 
\begin{equation}
   \left|\frac{\fftn_j(w)}{\fftn_{j-1}(w)} \right|\le \left|\frac{\fftn_j(u_c)}{\fftn_{j-1}(u_c)}\right|
   \end{equation}
with $u_c$ on the contour satisfying 
\begin{equation} 
    u_c=-\rho+\rho\sqrt{1-\rho}\xi_c N^{-1/2} 
\end{equation}
for a complex number $\xi_c$ satisfying $|\xi_c|=N^{\epsilon/4}$. 
Now by   Lemma~\ref{lem:gnnec}, 
   \begin{equation} \label{eq:rhsfpr}
   \left|\frac{\fftn_j(u_c)\fftn_{j-1}(-\rho)}{\fftn_{j-1}(u_c)\fftn_j(-\rho)}\right|=\left|e^{(x_{j}-x_{j-1})\xi_c+\frac{1}{2}(\gamma_{j}-\gamma_{j-1})\xi^2_c-\frac{1}{3}(\tau_{j}-\tau_{j-1})\xi^3_c}\right|\left(1+O(N^{\epsilon-1/2})\right).
   \end{equation}
The point $u_c$ is on the contour~\eqref{eq:aux_2017_03_29_06} and the contour is close to the contour $\left|w^{\rho}(w+1)^{1-\rho}\right|=\rr_0$ exponentially, which is self-intersecting at $w=-\rho$ like an x. 
This implies that  $\arg(\xi_c)$ converges to either $\frac{3\pi}{4}$ or $-\frac{3\pi}4$.
Since $\tau_{j}-\tau_{j-1}$ is positive, we find that the right-hand side of~\eqref{eq:rhsfpr} decays fast and is of order $e^{-cN^{3\epsilon/4}}$.
This proves~\eqref{eq:aux_2017_04_04_03} when $w\in \rootsL_{\bz}$.
The case when $w\in \rootsR_{\bz}$ is similar. 
\end{proof}

From Lemma~\ref{lem:gnnec} and~\ref{lm:tail_1}, we obtain the following asymptotics of $\hfs_j(w)$. 

\begin{lm} \label{lm:asym_fs}
Recall $\hfs_j(w)$ introduced in~\eqref{eq:newfhfs} and $\sfs_j(w)$ defined in~\eqref{eq:def_sfs}. 
Assume the same conditions on the parameters as in Lemma~\ref{lm:tail_1}.
Then\footnote{If $\tau_{j-1}=\tau_{j}$ and $x_{j-1}<x_{j}$, then the error changes to $O(e^{-cN^{\epsilon/4}})$ for $|\ww|\ge N^{\epsilon/4}$.} there is a positive constant $c$ such that for $w\in \rootsL_{\bz}\cup \rootsR_\bz$, 
\begin{equation}
	\label{eq:asym_fs}
	\hfs_j(w) 
	=\begin{dcases}
	\sfs_j(\ww)
	(1+O(N^{\epsilon-1/2}))& \text{if $|\ww|\le N^{\epsilon/4}$}\\
	O(e^{-cN^{3\epsilon/4}})& \text{if $|\ww|\ge N^{\epsilon/4}$}
	\end{dcases}
	\end{equation}
where for given $w$, $\ww$ is defined by the relation 
\begin{equation}
	w=-\rho+\frac{\rho\sqrt{1-\rho}}{N^{1/2}}\ww.
\end{equation}
In particular, $\hfs_j(w)$ is bounded uniformly for $w\in \rootsL_{\bz}\cup \rootsR_\bz$ as $N\to\infty$. 
\end{lm}

\subsubsection{Verification of conditions (i) and (ii) of Lemma~\ref{lem:K1Kw2estmfo}}

Let $w\in\roots_{\bz_i}\cap\mrootL$ and $w'\in\roots_{\bz_j}\cap\mrootR$.  
Let $\zeta$ and $\zeta'$ be the image of $w$ and $w'$ under either the map $\mathcal{M}_{N,\LL}$ or $\mathcal{M}_{N,\RR}$ in Lemma~\ref{lm:limit_nodes}, depending on whether the point is on $\rootsL_{z}$ or $\rootsR_{z}$.
We also set 
\beq
	\ww:= \frac{N^{1/2}}{\rho\sqrt{1-\rho}} (w+\rho), 
	\qquad
	\ww':= \frac{N^{1/2}}{\rho\sqrt{1-\rho}} (w'+\rho). 
\eeq
Then, by Lemma~\ref{lm:limit_nodes}, 
\beq	\label{eq:wwtozt}
	|\ww-\zeta|, \quad |\ww'-\zeta'| \le  N^{3\epsilon/4-1/2}\log N.
\eeq

We have 
\begin{equation} 
\label{eq:aux_2017_07_01_02}
 \begin{split}
 	\tKsL(w,w') &= -\left(\delta_i(j)+\delta_i\left(j+(-1)^i\right)\right) 
	\frac{\jac(w) \sqrt{\hfs_i(w)}\sqrt{ \hfs_j(w')} (\sH_{\bz_i}(w))^2 }{\sH_{\bz_{i-(-1)^{i}}}(w) \sH_{\bz_{j-(-1)^{j}}}(w')(w-w')} \sQL(j). 
\end{split} 
\end{equation}
Clearly, $\sQL(j)= \QWL(j)$.
Assume that $w, w'\in\Omega$.
Then 
\begin{equation}
\label{eq:aux_2017_9_27_01}
	\jac(w)=-\frac{\rho\sqrt{1-\rho}}{\ww N^{1/2}}\left(1+O(N^{\epsilon/4-1/2})\right).
\end{equation}
The overall minus sign in $\tKsL(w,w')$ cancels the minus sign from $\jac(w)$.
For other factors in~\eqref{eq:aux_2017_07_01_02}, we use 
Lemma~\ref{lem:qnearcpw} and Lemma~\ref{lm:asym_fs}. Here, we recall from~\eqref{eq:defofHszw} that 
$\sH_z(w)=\qol_z(w)$ for $\Re(w)<-\rho$ and $\sH_z(w)=\qor_z(w)$ for $\Re(w)>-\rho$. 
We obtain 
\beq
	|\tKsL(w,w')| = |\tlmKL(\ww, \ww')| (1+O(N^{\epsilon-1/2}\log N)).
\eeq
We then take the approximate $\ww$ and $\ww'$ by $\zeta$ and $\zeta'$ using~\eqref{eq:wwtozt}. 
Since the derivatives of $\tlmKL(\ww, \ww')$ are bounded (which is straightforward to check), 
we find that
\beq
	|\tKsL(w,w')| = |\tlmKL(\zeta, \zeta')| (1+O(N^{\epsilon-1/2}\log N))
	+ O(N^{3\epsilon/4-1/2}\log N).
\eeq
Recall that $|\tlmKL(\zeta, \zeta')|$ are bounded; see the first paragraph of the proof of Lemma~\ref{lem:K1Kw2estmfo}.  
Therefore, we obtain the first equation of~\eqref{eq:aux_2017_07_01_01}.
The estimate of $|\tKsR(w,w')|$ is similar, and we obtain (i) of Lemma~\ref{lem:K1Kw2estmfo}.
The part (ii) of the lemma is similar.

\subsubsection{Verification of condition (iii) of Lemma~\ref{lem:K1Kw2estmfo}} \label{sec:asymlasiii}

Consider $|\tKsL(w,w')|$ when $w$ or $w'$ is in $\Omega^c$. 
Here $w\in\mrootL$ and $w'\in\mrootR$.
We estimate each of the factors on the right hand side of~\eqref{eq:aux_2017_07_01_02}. 
We have the following. 
The estimates are all uniform in $w$ or $w'$ in the domain specified. 
The positive constants $C$, $C'$, and $c$ may be different from a line to a line. 
\begin{enumerate}[(1)]
	\item $N^{1/2} |w+\rho|\ge C$ for $w \in\mrootL\cup \mrootR$. This follows from Lemma~\ref{lm:limit_nodes}.
	\item $N^{1/2} |w-w'|\ge C$ for $w\in\mrootL$ and $w'\in\mrootR$.
	\item $|\jac(w)|\le CN^{-1/2}$  for $w \in\mrootL\cup \mrootR$. Recall~\eqref{eq:jac} for the definition of $\jac(w)$. 
	\item $|\hfs_i(w)|\le C$  for $w \in\mrootL\cup \mrootR$ by Lemma~\ref{lm:asym_fs}.
	\item $|\hfs_i(w)|\le Ce^{-cN^{3\epsilon/4}}$ for $w\in \Omega^c\cap (\mrootL\cup \mrootR)$ by Lemma~\ref{lm:asym_fs}.
	\item $C \le \left|\sH_{\bz_i}(w)\right|\le C'$ for  $w\in\mrootL\cup\mrootR$ by Lemma~\ref{lem:qnearcpw}. For the upper bound, we also use the decay property of $\hftn(\ww, z)=O(\ww^{-1})$ from~\eqref{eq:hftn_estimate}.  Recall~\eqref{eq:defofHszw} for the definition of $\sH_{\bz}$.
	\item $|\sQL(i)|\le C$ and $|\sQR(i)|\le C$
\end{enumerate}
By combining these facts we obtain\footnote{If $\tau_{j-1}=\tau_{j}$ and $x_{j-1}<x_{j}$, then the error terms in (5) and the kernels are $O(e^{-cN^{\epsilon/4}})$.} that $|\tKsL(w,w')|=O( e^{-cN^{3\epsilon/4}})$ 
and $|\tKsR(w',w)|=O( e^{-cN^{3\epsilon/4}})$ if  $w$ or $w'$ is in $\Omega^c$.
Hence we obtain (iii) of  Lemma~\ref{lem:K1Kw2estmfo}. 
This completes the proof of Theorem~\ref{thm:htfntasy}.

\section{Properties of the limit of the joint distribution} \label{sec:consistency}

In this section, we discuss a few properties of the function
\begin{equation}
	 \FS(x_1, \cdots, x_m; \spp_1, \cdots, \spp_m)
\end{equation}
introduced in Section~\ref{sec:limitdisformula}, which is a limit of the joint distribution. 
In order to emphasize that this is a function of $m$ variables with $m$ parameters, let us use the notation
\begin{equation}
	\FSm (\bx; \bp)= \FS(x_1, \cdots, x_m; \spp_1, \cdots, \spp_m) 
\end{equation}
where 
\beq
	\bx=(x_1, \cdots, x_m), \qquad \bp=(\spp_1, \cdots, \spp_m).
\eeq
We also use the notations
\beq
	\bxk=(x_1, \cdots, x_{k-1}, x_{k+1}, \cdots, x_m), \qquad 
	\bpk=(\spp_1, \cdots, \spp_{k-1}, \spp_{k+1}, \cdots, \spp_m)
\eeq
for vectors of size $m-1$ with $x_k$ and $\spp_k$ removed, respectively.

Recall the formula 
\beq \label{eq:Sfformulaa}
\bes
	\FSm(\bx; \bp)=	
	\oint\cdots\oint \cccm(\bfz; \bx, \bp) \gdetlmm(\bfz; \bx, \bp) \ddbar{z_m}\cdots\ddbar{z_1}
\end{split}
\end{equation}
where $\bfz=(z_1, \cdots, z_m)$, and the contours are nest circles satisfying 
\beq
	0<|z_m|<\cdots<|z_1|<1.
\eeq
We wrote $\ccc(\bfz)$ and $\gdetlm(\bfz)$ by $\cccm(\bfz; \bx, \bp)$ and $\gdetlmm(\bfz; \bx, \bp)$ to emphasize that they are functions of $m$ variables $z_1, \cdots, z_m$ and depend on $\bx$ and $\bp$. 
Recall from Property (P1) in Subsubsection~\ref{sec:defFS}, which we proved in Subsubsection~\ref{sec:cccsz}, that for each $i$, $\cccm(\bfz)$ is a meromorphic function of $z_i$ in the disk $|z_i|< 1$ and the only simple poles are $z_i=z_{i+1}$ for $i=1, \cdots, m-1$.
The part (i) of the next lemma shows the analytic property of $\gdetlmm(\bfz; \bx, \bp)$ and proves Property (P2) in Subsubsection~\ref{sec:defFS}. 
The part (ii) is used later.

\begin{lm} \label{lem:Dgetlmmdcov}
\begin{enumerate}[(i)]
\item $\gdetlmm(\bfz; \bx, \bp)$ is analytic in each of the variables $z_k$ in the deleted disk $0<|z_k|<1$. 
\item For $1\le k\le m-1$,
\begin{equation} \label{eq:Dintheconstce}
	\lim_{z_k\to z_{k+1}}  \gdetlmm(\bfz; \bx, \bp)
	= \gdetlmmo(\bfzk; \bxk, \bpk)
\end{equation}
where $\bfzk= (z_1, \cdots, z_{k-1}, z_{k+1}, \cdots, z_m)$. 
\end{enumerate}
\end{lm}

\begin{proof} 
Recall the series formula~\eqref{eq:aux_2017_03_29_14}:
\begin{equation} \label{eq:formrecsef}
	\gdetlmm(\bfz; \bx, \bp)
	=\sum_{\bn\in (\intZ_{\ge 0})^m} \frac{1}{(\bn !)^2}
\gdetlmm_{\bn}(\bfz; \bx, \bp), 
\qquad
	\gdetlmm_\bn(\bfz; \bx, \bp)= \sum 
	\gdetlmddm_{\bn, \bfz} (\bUUlm, \bVVlm; \bx, \bp) 
\end{equation}
where the last sum is over $\UUlm^{(\ell)}\in\left(\inodesL_{z_\ell}\right)^{n_\ell}$ and $\VVlm^{(\ell)}\in\left(\inodesR_{z_\ell}\right)^{n_\ell}$, $\ell=1,\cdots,m$. 
Here, (recall the notational convention in Definition~\ref{def:notc})
\begin{equation} \label{eq:duvfoforn}
\begin{split}
	&\gdetlmddm_{\bn, \bfz} (\bUUlm, \bVVlm; \bx, \bp)
	=
	\left[ \prod_{\ell=1}^m \frac{\Delta(\UUlm^{(\ell)})^2\Delta(\VVlm^{(\ell)})^2}{\Delta(\UUlm^{(\ell)};\VVlm^{(\ell)})^2}
	\fslm_\ell(\UUlm^{(\ell)}) 
	\fslm_\ell(\VVlm^{(\ell)})\right]\\
	&\times\left[ \prod_{\ell=2}^m  \frac{\Delta(\UUlm^{(\ell)};\VVlm^{(\ell-1)})\Delta(\VVlm^{(\ell)};\UUlm^{(\ell-1)}) e^{-\hftn(\VVlm^{(\ell)}, z_{\ell-1}) - \hftn(\VVlm^{(\ell-1)}, z_{\ell})} }
	{\Delta(\UUlm^{(\ell)};\UUlm^{(\ell-1)})\Delta(\VVlm^{(\ell)};\VVlm^{(\ell-1)}) 
	 e^{\hftn(\UUlm^{(\ell)}, z_{\ell-1}) + \hftn(\UUlm^{(\ell-1)}, z_{\ell})} }
	\left(1-\frac{z_{\ell-1}}{z_{\ell}}\right)^{n_\ell} \left(1-\frac{z_\ell}{z_{\ell-1}}\right)^{n_{\ell-1}} \right]
\end{split}
\end{equation}
where $\bUUlm=(\UUlm^{(1)}, \cdots, \UUlm^{(m)})$, $\bVVlm= (\VVlm^{(1)}, \cdots, \VVlm^{(m)})$ with $\UUlm^{(\ell)}=(\su_1^{(\ell)}, \cdots, \su_{n_\ell}^{(\ell)})$, $\VVlm^{(\ell)}=(\sv_1^{(\ell)}, \cdots, \sv_{n_\ell}^{(\ell)})$. 
Note that we may take the components of $\UUlm^{(\ell)}$ and $\VVlm^{(\ell)}$ to be all distinct due to the factors $\Delta(\UUlm^{(\ell)})$ and $\Delta(\VVlm^{(\ell)})$.
Recall that $\fslm_\ell(\zeta):=\frac{1}{\zeta} \sfs_\ell(\zeta) e^{2\hftn(\zeta, z_\ell)}$
where $\hftn$ and $\sfs_j$ are defined in~\eqref{eq:def_h_R},~\eqref{eq:def_h_L}, and~\eqref{eq:def_sfs}. 

The points $\su_j^{(\ell)}$ and $\sv_j^{(\ell)}$ are roots of the equation $e^{-\zeta^2/2}=z_\ell$, and hence they depend on $z_\ell$ analytically (if we order them properly). 
Note that the only denominators in~\eqref{eq:duvfoforn} which can vanish are 
$\Delta(\UUlm^{(\ell)};\UUlm^{(\ell-1)})$ and $\Delta(\VVlm^{(\ell)};\VVlm^{(\ell-1)})$. 
They vanish only when $z_{\ell-1}=z_{\ell}$. 
Hence the only possible poles of $\gdetlmm(\bfz; \bx, \bp)$ are $z_k=z_{k+1}$ for $k=1, \cdots, m-1$. 
Hence (ii) implies (i).

We now prove (ii). We may assume that $z_1, \cdots, z_m$ are {all distinct and take the limit as $z_k\to z_{k+1}$; otherwise, we} may take successive limits.
Let us consider which terms in~\eqref{eq:duvfoforn} vanish when $z_k= z_{k+1}$.
Clearly, 
$(1-\frac{z_{k}}{z_{k+1}})^{n_{k+1}} (1-\frac{z_{k+1}}{z_{k}})^{n_{k}}$ vanishes.
On the other hand, when $z_k=z_{k+1}$, $\UUlm^{(k)}$ and $\UUlm^{(k+1)}$ are from the same set. 
If there is a non-zero overlap between the pair of vectors, then $\Delta(\UUlm^{(k+1)}; \UUlm^{(k)})$, which is in the denominator in~\eqref{eq:duvfoforn}, vanishes. 
Similarly, $\Delta(\VVlm^{(k+1)}; \VVlm^{(k)})$ may also vanish. 
Hence $\gdetlmddm_{\bn, \bfz} (\bUUlm, \bVVlm; \bx, \bp)$ is equal to 
\beq \label{eq:zdivbiyuva}
	\frac{(z_{k+1}-z_k)^{n_{k+1}+n_k}}{\Delta(\UUlm^{(k+1)}; \UUlm^{(k)}) \Delta(\VVlm^{(k+1)}; \VVlm^{(k)})}
\eeq
times a non-vanishing factor as $z_k\to z_{k+1}$. 
Now, note that if $\su_i^{(k)}\to \su_j^{(k+1)}$ for some $i$ and $j$, then from $e^{-(\su_j^{(k+1)})^2/2}=z_{k+1}$, 
\beq
	\lim_{\su_i^{(k)}\to \su_j^{(k+1)}} \frac{z_{k+1}-z_k}{\su_j^{(k+1)}-\su_i^{(k)}} 
	= \frac{\dd z_{k+1}}{\dd \su_j^{(k+1)}} 
	= -\su_j^{(k+1)}z_{k+1}. 
\eeq
Since $\Delta(\UUlm^{(k+1)}; \UUlm^{(k)})$ has at most $\min\{n_{k+1}, n_k\}$ vanishing factors $\su_j^{(k+1)}-\su_i^{(k)}$ (because the components of $\UUlm^{(k)}$ are all distinct, and so are $\UUlm^{(k+1)}$) and similarly, $\Delta(\VVlm^{(k+1)}; \VVlm^{(k)})$ has at most $\min\{n_{k+1}, n_k\}$ vanishing factors $\sv_j^{(k+1)}-\sv_i^{(k)}$, 
we find, by looking at the degree of the numerator, that~\eqref{eq:zdivbiyuva} is non-zero only if (a) $n_k=n_{k+1}$, (b) $\UUlm^{(k)}$ converges to a vector whose components are a permutation of the components of $\UUlm^{(k+1)}$, and (c) $\VVlm^{(k)}$ converges to a vector whose components are a permutation of the components of $\VVlm^{(k+1)}$. 
When $n_k=n_{k+1}$, $\UUlm^{(k)}\to \UUlm^{(k+1)}$, and $\VVlm^{(k)}\to \VVlm^{(k+1)}$, the term~\eqref{eq:zdivbiyuva} converges to 
\beqq
	\frac{z_{k+1}^{2n_{k+1}} \prod_{i=1}^{n_{k+1}} \su_i^{(k+1)} \sv_i^{(k+1)}}{\Delta(\UUlm^{(k+1)})^2\Delta(\VVlm^{(k+1)})^2}.
\eeqq
Hence, using the fact that (see~\eqref{eq:def_sfs})
$\sfs_k(w)\sfs_{k+1}(w)$ is equal to $e^{-\frac13 (\tau_{k+1}-\tau_{k-1}) \zeta^3 + \frac12 (\gamma_{k+1}-\gamma_{k-1}) \zeta^2 + (x_{k+1}-x_{k-1}) \zeta}$ 
for $\Re(w)<0$ and is equal to $e^{\frac13 (\tau_{k+1}-\tau_{k-1}) \zeta^3 - \frac12 (\gamma_{k+1}-\gamma_{k-1}) \zeta^2 - (x_{k+1}-x_{k-1}) \zeta}$ 
for $\Re(w)>0$, we find that $\gdetlmddm_{\bn, \bfz} (\bUUlm, \bVVlm; \bx, \bp)$ converges to 
\beqq
	\gdetlmddmo_{\bnk, \bfzk} (\bUUlmk, \bVVlmk; \bxk, \bpk) 
\eeqq
where $\bnk$, $\bUUlmk$, $\bVVlmk$ are same as $\bn, \bUUlm, \bVVlm$ with $n_k$, 
$\bUUlm^{(k)}$, $\bVVlm^{(k)}$ removed. 
We obtain the same limit if $n_k=n_{k+1}$, $\UUlm^{(k)}$ converges to a vector whose components are a permutation of the components of $\UUlm^{(k+1)}$, and $\VVlm^{(k)}$ converges to a vector whose components are a permutation of the components of $\VVlm^{(k+1)}$. 
Hence,
\begin{equation*}
	\lim_{z_k\to z_{k+1}} \gdetlmm_\bn(\bfz; \bx, \bp)
	= (n_k!)^2 \sum
	\gdetlmddmo_{\bnk, \bfzk} (\bUUlmk, \bVVlmk; \bxk, \bpk)
\end{equation*}
where the sum is over $\UUlm^{(\ell)}\in\left(\inodesL_{z_\ell}\right)^{n_\ell}$ and $\VVlm^{(\ell)}\in\left(\inodesR_{z_\ell}\right)^{n_\ell}$, $\ell=1,\cdots,k-1, k+1, \cdots, m$.  
This implies~\eqref{eq:Dintheconstce}. 
This completes the proof of (ii), and hence (i), too.
\end{proof}

Next lemma shows what happens if we interchange two consecutive contours in the formula~\eqref{eq:Sfformulaa} of $\FSm(\bx; \bp)$.  
Although we do not state explicitly, we can also obtain an analogous formula for the finite-time joint distribution from a similar computation. 

\begin{lm} \label{lem:FresidueF}
For every $1\le k\le m-1$, 
\beq \label{eq:Fresintm}
\bes
	\FSm(\bx; \bp)=	 \FSmo(\bxk; \bpk) + 
	\oint\cdots\oint \cccm(\bfz; \bx, \bp) \gdetlmm(\bfz; \bx, \bp) \ddbar{z_m}\cdots\ddbar{z_1}
\end{split}
\end{equation}
where the contours are nested circles satisfying  
\beq \label{eq:Fsneocon}
	\text{$0<|z_m|<\cdots<|z_{k+1}|<|z_{k-1}|<\cdots<|z_1|<1$ 
	and $0<|z_k|<|z_{k+1}|$.}
\eeq
\end{lm}

\begin{proof}
We start with the formula~\eqref{eq:Sfformulaa}. 
We fix all other contours and deform the $z_k$-contour so that $|z_k|$ is smaller than $|z_{k+1}|$. 
Then the integral~\eqref{eq:Sfformulaa} is equal to a term due to the residue plus the same integral with the contours changed to satisfy
\beqq
	0< |z_m|<\cdots<|z_{k+2}|<|z_k|<|z_{k+1}|<|z_{k-1}|<\cdots<|z_1|<1.
\eeqq
Since the integrand have poles only at $z_i=z_{i+1}$, $i=1, \cdots, m-1$, 
it is analytic at $z_k=z_i$ for $i\ge k+2$. 
Hence the conditions that $|z_k|>|z_i|$, $i\ge k+2$, are not necessary, and we can take the contours as~\eqref{eq:Fsneocon}. 

It remains to show that the residue term is $\FSmo(\bxk; \bpk)$. 
It is direct to check from the definition~\eqref{eq:costcdefaq} that
\beq \label{eq:Cintheconstce}
	\lim_{z_k\to z_{k+1}} (z_k-z_{k+1}) \cccm(\bfz; \bx, \bp)
	= z_{k+1} \cccmo(\bfzk; \bxk, \bpk)
\eeq
where we set $\bfzk=(z_1, \cdots, z_{k-1}, z_{k+1}, \cdots, z_m)$.
On the other hand, from Lemma~\ref{lem:Dgetlmmdcov} (ii), $\gdetlmm(\bfz; \bx, \bp)$ converges to 
$\gdetlmm(\bfzk; \bxk, \bpk)$. 
Thus, noting $\frac{\dd z_k}{2\pi i z_k}$ in~\eqref{eq:Sfformulaa}, the residue term 
is equal to $\FSmo(\bxk; \bpk)$. 
\end{proof}

The multiple integral in~\eqref{eq:Fresintm} has a natural probabilistic interpretation.
The following theorem gives an interpretation for more general choices of contours. 

\begin{thm} \label{thm:pofEpm}
Assume the same conditions as Theorem~\ref{thm:htfntasy}.
Let $E_j^{\pm}$ be the events defined by 
\beq
	E_j^{-}= \left\{\frac{ \height(\mr p_j)-(1-2\rho)s_j-(1-2\rho+2\rho^2)t_j}{-2\rho^{1/2}(1-\rho)^{1/2} L^{1/2}}\le x_j\right\}
\eeq
and
\beq
	E_j^{+}= \left\{\frac{ \height(\mr p_j)-(1-2\rho)s_j-(1-2\rho+2\rho^2)t_j}{-2\rho^{1/2}(1-\rho)^{1/2} L^{1/2}}> x_j\right\}.
\eeq
Then
\begin{equation}
	\label{eq:001}
	\lim_{L\to\infty}\prob\left(E_1^\pm\cap\cdots\cap E_{m-1}^\pm \cap E_m^{-} \right)=(-1)^\#\oint\cdots\oint \ccc(\bfz) \gdetlm(\bfz) \ddbar{z_m}\cdots\ddbar{z_1}
\end{equation}
where $\#$ denotes the number of appearances of $+$ in $E_1^\pm\cap\cdots\cap E_{m-1}^\pm$, and the contours are circles of radii between $0$ and $1$ such that for each $1\le j\le m-1$, 
\beq \label{eq:contrs}
	\text{$|z_j|>|z_{j+1}|$ if we have $E_j^-$ and $|z_j|<|z_{j+1}|$ if we have $E_j^+$.}
\eeq
\end{thm}

\begin{proof} 
Theorem~\ref{thm:htfntasy} is the special case when we take $E_j^-$ for all $j$. 
The general case follows from the same asymptotic analysis starting with a different finite-time formula. 
We first change Theorem~\ref{prop:multipoint_distribution_origin}. 
Let $\tilde E_j^+:=\left\{\xx_{k_j}(t_j)\ge a_j\right\}$ and $\tilde E_j^-:=\left\{\xx_{k_j}(t_j)< a_j\right\}$ be events of TASEP in $\conf_N(L)$. 
Then we claim that 
\begin{equation}
\label{eq:aux_2017_07_01_03}
\begin{split}
\prob_Y\left(\tilde E_1^\pm\cap\cdots\cap \tilde E_{m-1}^\pm \cap \tilde E_m^{-}\right)
=  
(-1)^{\#}\oint \cdots\oint 
\constc(\bbz, \bk)  \detv_Y(\bbz, \bk, \tilde\ba, \bt) \ddbar{z_m}\cdots \ddbar{z_1}
\end{split}
\end{equation}
where 
the contours are circles of radii between $0$ and $1$ such that for each $1\le j\le m-1$, 
$|z_j|>|z_{j+1}|$ if we have $\tilde E_j^-$ and $|z_j|<|z_{j+1}|$ if we have $\tilde E_j^+$.
Here $\#$ denotes the number of appearances of $+$ in $\tilde E_1^\pm\cap\cdots\cap \tilde E_{m-1}^\pm$
and 
$\tilde \ba=(\tilde a_1,\cdots,\tilde a_m)$ with $\tilde a_j=a_j$ or $a_j-1$ depending on whether we have $\tilde E_j^-$ or $\tilde E_j^+$, respectively. 
The identity~\eqref{eq:aux_2017_07_01_03} follows from the same proof of Theorem~\ref{prop:multipoint_distribution_origin} with a small change. 
The contour condition is used in the proof when we apply Proposition~\ref{thm:DRL_simplification}. 
Now,  in Proposition~\ref{thm:DRL_simplification},  if we replace the assumption\footnote{See the paragraph including the equation~\eqref{eq:wordtep} for a discussion how the condition $\prod_{j=1}^N|w'_j+1|<\prod_{j=1}^N|w_j+1|$ is related to the condition $|z'|<|z|$.} $\prod_{j=1}^N|w'_j+1|<\prod_{j=1}^N|w_j+1|$ by $\prod_{j=1}^N|w'_j+1|>\prod_{j=1}^N|w_j+1|$, then the conclusion changes to  
\begin{equation} \label{eq:popral}
\begin{split}
	&\sum_{\substack{X\in \conf_N(L)\cap\{x_{k}< a\} }} \DR_{X}\left(W\right)\DL_{X}\left(W'\right)\\
	&=-\left(\frac{z}{z'}\right)^{(k-1)L}\left(1-\left(\frac{z'}{z}\right)^L\right)^{N-1}
\left[ \prod_{j=1}^N \frac{w_j^{-k}(w_j+1)^{-a+k+2}}{(w'_j)^{-k}(w'_j+1)^{-a+k+1}} \right] \det\left[\frac{1}{w_i-w'_{i'}}\right]_{i,i'=1}^N. 
\end{split}
\end{equation}
Note the change in the summation domain from $x_k\ge a$ to $x_k<a$. 
This identity follows easily from Proposition~\ref{prop:01} and the geometric series formula. 
The probability of the event $\tilde E_1^\pm\cap\cdots\cap \tilde E_{m-1}^\pm \cap \tilde E_m^{-}$ is then obtained from the same calculation as before using,  for each $i$,  either Proposition~\ref{thm:DRL_simplification} or~\eqref{eq:popral} depending on whether we have $\tilde E_i^+$ or $\tilde E_i^-$.  

For the periodic step initial condition, Theorem~\ref{thm:multi_point_formula} showed that 
$\constc(\bbz, \bk)  \detv_Y(\bbz, \bk, \tilde\ba, \bt)$ is equal to $\const(\bbz) \gdet(\bbz)$. 
This proof does not depend on how the $z_i$-contours are nested. 
Hence we obtain 
\begin{equation}
\label{eq:aux_2017_07_01_04}
	\prob\left(\tilde E_1^\pm\cap\cdots\cap \tilde E_{m-1}^\pm \cap \tilde E_m^{-}\right)=(-1)^\#\oint\cdots\oint \const(\bbz) \gdet(\bbz) \ddbar{z_m}\cdots\ddbar{z_1}
\end{equation}
where $\const(\bbz)$ and $ \gdet(\bbz)$ are the same as that in Theorem~\ref{thm:multi_point_formula} except that the parameter $\ba$ is replaced by $\tilde{\ba}$. 
The contours are the same as that in~\eqref{eq:aux_2017_07_01_03}.
Now, the asymptotic analysis follows from Section~\ref{sec:pfastp}. Recall that the analysis of Section~\ref{sec:pfastp} does not depend on the ordering of $|z_i|$; see Remark~\ref{rmk:zord}. 
\end{proof}

We now prove the consistency of $\FSm(\bx; \bp)$ when one of the variables tends to positive infinity. 

\begin{prop}[Consistency] \label{prop:consy}
We have 
\begin{equation} \label{eq:conscy}
	\lim_{x_k \to +\infty} \FSm(\bx; \bp) 
	= \FSmo (\bxk; \bpk). 
\end{equation}
\end{prop}

\begin{proof}
We consider the case when $k=m$ first and then the case when $k<m$. 

(a) Assume $k=m$. When $m=1$, we showed in Section 4 of \cite{Baik-Liu16} that $\FS^{(1)}(x_1; \spp_1)$ is a distribution function. Hence, 
\beq \label{eq:taillimone1}
	\lim_{x_1 \to +\infty} \FS^{(1)}(x_1; \spp_1)=1.
\eeq
Now we assume that $m\ge 2$ and take $x_m\to +\infty$. 
Recall the formula~\eqref{eq:Sfformulaa} in which the $z_m$-contour is the smallest contour. 
From the definition (see~\eqref{eq:costcdefaq}), 
\beqq
	\cccm(\bfz; \bx, \bp)=\cccmo(\bfz^{[m]}; \bx^{[m]}, \bp^{[m]}) \frac{z_{m-1}}{z_{m-1}-z_m} 
	\frac{e^{x_m A_1(z_{m})+ \tau_m A_2(z_m)}}{e^{x_{m-1} A_1(z_{m}) + \tau_{m-1} A_2(z_{m})}}
	e^{2B(z_{m})-2B(z_{m},z_{m-1})}.
\eeqq
We choose the contour for $z_m$ given by $z_m=\frac1{x_m^2}e^{\ii\theta}$, $\theta\in [0, 2\pi)$. 
Since $A_i(z)=O(z)$, $B(z)=O(z)$, and $B(z, w)=O(z)$ as $z\to 0$, we have 
\begin{equation}
\label{eq:aux_2017_07_01_06}
	\cccm(\bfz; \bx, \bp)=\cccmo(\bfz^{[m]}; \bx^{[m]}, \bp^{[m]})+O(x_m^{-1})
\end{equation}
as $x_m\to \infty$.

We now show that 
\begin{equation}
\label{eq:aux_2017_07_01_07}
	\gdetlmm(\bfz; \bx, \bp)=\gdetlmmo(\bfz^{[m]}; \bx^{[m]}, \bp^{[m]})+O(e^{-cx_m})
\end{equation}
as $x_m\to \infty$, where we had chosen $|z_m|=\frac1{x_m^2}$. 
From the series formula~\eqref{eq:Fdmddser}, 
\begin{equation}
\label{eq:aux_2017_09_26_02}
	\gdetlmm(\bfz; \bx, \bp) - \gdetlmmo(\bfz^{[m]}; \bx^{[m]}, \bp^{[m]})= \sum_{n_m\ge 1}\sum_{n_1,\cdots,n_{m-1}\ge 0}\frac{1}{(\bn!)^2}\gdetlm_\bn(\bbz).
\end{equation}
We need to show that the sum is exponentially small. 
From~\eqref{eq:aux_2017_03_29_14} and~\eqref{eq:aux_2017_07_01_05},
\begin{equation} \label{eq:Dnwimesp}
	\gdetlm_\bn(\bfz)  
	= \sum_{\substack{\UUlm^{(\ell)}\in\left(\inodesL_{z_\ell}\right)^{n_\ell} \\ \VVlm^{(\ell)}\in\left(\inodesR_{z_\ell}\right)^{n_\ell} \\ \ell=1,\cdots,m-1}}
	\bigg[ 	\sum_{\substack{\UUlm^{(m)}\in\left(\inodesL_{z_m}\right)^{n_m} \\ \VVlm^{(m)}\in\left(\inodesR_{z_m}\right)^{n_m}}}
	\gdetlmdd_{\bn, \bfz} (\bUUlm, \bVVlm) \bigg] 
\end{equation}
where for each $\bUUlm$ and $\bVVlm$, $\gdetlmdd_{\bn, \bfz} (\bUUlm, \bVVlm)$ is equal to 
\begin{equation}
\label{eq:aux_2017_09_26_01}
\begin{split}
	&\frac{\Delta(\UUlm^{(m)})^2\Delta(\VVlm^{(m)})^2}{\Delta(\UUlm^{(m)};\VVlm^{(m)})^2} 
\fslm_m(\UUlm^{(m)})  
\fslm_m(\VVlm^{(m)})\\
&\cdot
\frac{\Delta(\UUlm^{(m)};\VVlm^{(m-1)})\Delta(\VVlm^{(m)};\UUlm^{(m-1)}) e^{-\hftn(\VVlm^{(m)}, z_{m-1}) - \hftn(\VVlm^{(m-1)}, z_{m})} }
{\Delta(\UUlm^{(m)};\UUlm^{(m-1)})\Delta(\VVlm^{(m)};\VVlm^{(m-1)}) 
	e^{\hftn(\UUlm^{(m)}, z_{m-1}) + \hftn(\UUlm^{(m-1)}, z_{m})} }
\left(1- \frac{z_{m-1}}{z_{m}} \right)^{n_m} \left(1-\frac{z_m}{z_{m-1}} \right)^{n_{m-1}}
\end{split}
\end{equation}
times 
\beq \label{eq:ddwithoutmtm}
	\gdetlmdd_{\bn^{[m]}, \bfz^{[m]}} (\bUUlm^{[m]}, \bVVlm^{[m]}), 
\eeq
a factor which does not depend on $z_m$, and hence also on $\UUlm^{(m)}$ and $\VVlm^{(m)}$ (and $x_m$ and $\mr p_m$.)
The term~\eqref{eq:ddwithoutmtm} is same as $\gdetlmdd_{\bn, \bfz} (\bUUlm, \bVVlm)$ when $m$ is replaced by $m-1$.
Since $n_m\ge 1$ in the sum~\eqref{eq:aux_2017_09_26_02}, 
the inside sum in~\eqref{eq:Dnwimesp} is not over an empty set. 
We show that for each $\bUUlm$ and $\bVVlm$,~\eqref{eq:aux_2017_09_26_01} is 
\beq \label{eq:estofdimqp}
	O(e^{-c x_m|\UUlm^{(m)}|-cx_m|\VVlm^{(m)}|})
\eeq
for a constant $c>0$, where $|\UUlm^{(m)}|$ is the sum of the absolute values of $\UUlm^{(m)}$, and $|\VVlm^{(m)}|$ is similarly defined.
This proves~\eqref{eq:aux_2017_07_01_07}.

To show the decay of~\eqref{eq:aux_2017_09_26_01}, we first note that every component $u$ of $\UUlm^{(m)}$ is a solution of the equation $e^{-u^2/2}=z_m$ satisfying $\Re(u)<0$. 
As $|z_m|=\frac1{x_m^2}\to 0$, $\Re(u^2)\to \infty$. 
Since $\Re(u^2)=(\Re(u))^2-(\Im(u))^2$, this implies that $\Re(u)\to -\infty$, and hence $|u|\to\infty$. 
Similarly, every component $v$ of $\VVlm^{(m)}$ satisfies $\Re(v)\to \infty$ and $|v|\to \infty$. 
It is also easy to check (see Figure~\ref{fig:Scontour}) that the solutions of the equation $e^{-\zeta^2/2}=z$ lies in the sectors $-\pi/4<\arg(\zeta)<\pi/4$ or $3\pi/4<\arg(\zeta)< 5\pi/4$ for any $0<|z|<1$. 
Hence $|u|\le \sqrt{2}\Re(u)$, $|v|\le \sqrt{2}\Re(v)$, and $\sqrt{2}|u-v|\ge |u|+|v|$. 
We now consider each term in~\eqref{eq:aux_2017_09_26_01}. 
Considering the degrees, 
\beqq
	\frac{\Delta(\UUlm^{(m)})^2\Delta(\VVlm^{(m)})^2}{\Delta(\UUlm^{(m)};\VVlm^{(m)})^2}=O(1),\qquad \frac{\Delta(\UUlm^{(m)};\VVlm^{(m-1)})\Delta(\VVlm^{(m)};\UUlm^{(m-1)})  }
{\Delta(\UUlm^{(m)};\UUlm^{(m-1)})\Delta(\VVlm^{(m)};\VVlm^{(m-1)})  }=O(1).
\eeqq
From the formula of $\hftn$ in~\eqref{eq:aux_2017_09_16_01}, and using~\eqref{eq:def_h_RminwithL}, 
\beqq
	\hftn(\VVlm^{(m)}, z_{m-1}), \, \hftn(\UUlm^{(m)}, z_{m-1}), \, 
	\hftn(\VVlm^{(m-1)}, z_{m}), \,  \hftn(\UUlm^{(m-1)}, z_{m}) =O(1).
\eeqq
Recall from~\eqref{eq:aux_2017_04_05_02} that $\fslm_m(\zeta)= \frac1{\zeta} \sfs_\ell(\zeta) e^{2\hftn(\zeta, z_{m})}$.
As above, 
\beqq
	\hftn(\VVlm^{(m)}, z_{m}), \,  \hftn(\UUlm^{(m)}, z_{m})=O(1).
\eeqq
On the other hand, from the definition~\eqref{eq:def_sfs}, 
\beqq
	| \sfs_\ell(u)|\le e^{-cx_m|u|}, \qquad | \sfs_\ell(v)|\le e^{-cx_m|v|}
\eeqq
for the components $u$ and $v$ of $\UUlm^{(m)}$ and $\VVlm^{(m)}$, implying that 
\beqq
	|\fslm_m(\UUlm^{(m)}) \fslm_m(\VVlm^{(m)})| = O(e^{-c x_m|\UUlm^{(m)}|-cx_m|\VVlm^{(m)}|}).
\eeqq
This term dominates the factor
\beqq
	\left(1- \frac{z_{m-1}}{z_{m}} \right)^{n_m} =O(x_m^{2n_m}).
\eeqq
Combining together, we obtain the decay~\eqref{eq:estofdimqp} of~\eqref{eq:aux_2017_09_26_01}, and hence we obtain~\eqref{eq:aux_2017_07_01_07}. 

We insert~\eqref{eq:aux_2017_07_01_06} and~\eqref{eq:aux_2017_07_01_07} in~\eqref{eq:Sfformulaa}, and integrating over $\theta$, we obtain
\begin{equation*}
\begin{split}
&\FSm(\bx; \bp) 
=\FS^{(m-1)}(\bx^{[m]}, \bp^{[m]})+O(x_m^{-1})
\end{split}
\end{equation*}
for all large $x_m$. Hence we proved~\eqref{eq:conscy} for $k=m$.

(b) Assume that $k<m$.  
Let us denote the integral in~\eqref{eq:001} with the contours~\eqref{eq:contrs} by 
\beqq
	\FgSm(x_1^{\pm}, \cdots, x_{m-1}^{\pm}, x_m^{-}; \bp).
\eeqq
Note that 
\beqq
	\FgSm(x_1^{-}, \cdots, x_{m-1}^{-}, x_m^{-}; \bp)= \FSm(x_1, \cdots, x_{m-1}, x_m ;  \bp).
\eeqq
Fix $k$ such that $1\le k\le m-1$. 
By Lemma~\ref{lem:FresidueF}, we obtain~\eqref{eq:conscy} if we show that 
\beq \label{eq:implsh}
	\lim_{x_k\to +\infty} \FgSm(x_1^{-}, \cdots, x_{k-1}^{-},  x_k^+, x_{k+1}^-, \cdots x_m^{-}; \bp) =0.
\eeq
Now, from the joint probability function interpretation stated in Theorem~\ref{thm:pofEpm},  
\begin{equation} \label{eq:004}
 \begin{split}
	 \FgSm(\cdots, x_{i-1}^{\pm},  x_i^{\pm}, x_{i+1}^{\pm}, \cdots ; \bp) 
	 \le \FgSmo( \cdots, x_{i-1}^{\pm},   x_{i+1}^{\pm}, \cdots; \bp^{[i]}) 
	 \end{split}
\end{equation}
for any $1\le i\le m$ (for any choice of $\pm$-sign for $x_i$, $1\le i\le m-1$, and the choice of $-$ sign for $x_m$.)
Using~\eqref{eq:004} $m-1$ times, we find that
\begin{equation}
 \begin{split}
	 \FgSm(x_1^{-}, \cdots, x_{k-1}^{-},  x_k^+, x_{k+1}^-, \cdots x_m^{-}; \bp) 
	 \le \FgSmoo(x_k^+; \spp_k) 
	 = 1- \FS^{(1)}(x_k; \spp_k).
\end{split}
\end{equation}
The one-point function $\FS^{(1)}(x_k; \spp_k)$ converges to $1$ as $x_k\to +\infty$ from~\eqref{eq:taillimone1}.
Hence we obtain~\eqref{eq:implsh}, and this completes the proof of~\eqref{eq:conscy} when $1\le k\le m-1$. 
\end{proof}

In the opposite direction, we have the following result. 

\begin{lm} \label{lm:Flimni}
We have
\begin{equation} \label{eq:lefttailf}
	\lim_{x_k \to -\infty} \FSm(\bx; \bp) =0.
\end{equation}
\end{lm}

\begin{proof}
Since a joint probability is smaller than a marginal distribution and $\FSm(\bx; \bp)$ is a limit of joint probabilities,
\beq
	\FSm(\bx; \bp)\le \FS^{(1)}(x_k; \spp_k).
\eeq 
As mentioned before, the function $\FS^{(1)}(x_k; \spp_k)= \FS(x_k; \spp_k)$ is shown to be a distribution function in Section 4 of \cite{Baik-Liu16}. 
This implies the lemma. 
\end{proof}


\section{Infinite TASEP} \label{sec:inftasep}

If we take $L\to\infty$ while keeping all other parameters fixed, the periodic TASEP becomes the infinite TASEP with $N$ particles. 
In terms of the joint distribution, this is still true if $L$ is fixed but large enough.


\begin{lm}
\label{thm:multi-time_TASEP}
Consider the infinite TASEP on $\intZ$ with $N$ particles and let $\tilde{x}_i(t)$ denote the location of the $i$th particle (from left to right) at time $t$. 
Assume that the infinite TASEP has the initial condition given by $\tilde{x}_i(0)=y_i$,
where  $y_1<\cdots<y_N$. 
Also consider the TASEP in $\conf_N(L)$ and denote by $\xx_i(t)$ the location of the $i$th particle.
Assume that 
\beq \label{eq:lcyna}
	L>y_N-y_1
\eeq 
and let $\xx_i(t)$ have the same initial condition given by $\xx_i(0)=y_i$. 
Fix a positive integer $m$.
Let $k_1,\cdots,k_m$ be integers in $\{1,\cdots,N\}$, let $a_1,\cdots,a_m$ be integers, 
and let $t_1,\cdots,t_m$ be positive real numbers. 
Then for any integer $L$  satisfying (in addition to~\eqref{eq:lcyna})
	\begin{equation}
	\label{eq:aux_2017_07_01_08}
		L > \max\{a_1-k_1,\cdots, a_m-k_m\} - y_1+ N+1,  
	\end{equation}
we have  
	\begin{equation}
	\label{eq:aux_2017_07_01_09}
	\prob\left( \tilde{x}_{k_1}(t_1)\le  a_1,\cdots, \tilde{x}_{k_m}(t_m)\le  a_m\right)
	=\prob\left( \xx_{k_1}(t_1)\le a_1,\cdots, \xx_{k_m}(t_m)\le a_m\right) .
	\end{equation} 
\end{lm}

\begin{proof}\footnote{This lemma can be seen easily from the directed last passage percolation interpretation of the TASEP.} 
We first observe that the particles $\xx_i(t)$ are in the configuration space $\conf_N(L)$, while the particles $\tilde x_i(t)$ are in the configuration space $W_N:=\{(x_1,\cdots,x_N)\in\intZ^N: x_1<\cdots<x_N\}$. 
The only difference between these two configuration space is that $\conf_N(L)$ has an extra restriction $x_N\le x_1+L-1$. 
Therefore, {if this restriction does not take an effect before} time $t$, i.e, $\xx_N(s)<\xx_1(s)+L-1$ for all $0<s< t$, then the dynamics of TASEP on $\conf_N(L)$ is the same as that of infinite TASEP (with the same initial condition) before time $t$.
Furthermore, if we focus on the $i$-th particle in TASEP in $\conf_N(L)$, there exists a smallest random time $T_i$ such that the dynamics of this particle are the same in both TASEP $\conf_N(L)$ and infinite TASEP before time $T_i$. 
{The times $T_i$ are determined inductively as follows.}
First, $T_N$ is the smallest time such that $\xx_N(t)=\xx_{1}(t)+L-1$.
Next,  $T_{N-1}$ is the smallest time that satisfies $t\ge T_N$ and $\xx_{N-1}(t)=\xx_N(t)-1$.
For general index $1\le i\le N-1$, $T_i$ is the smallest time that satisfies $t\ge T_{i+1}$ and $\xx_i(t)=\xx_{i+1}(t)-1$. 
Note that $T_1\ge T_2\ge \cdots \ge T_N$ and for $1\le i\le N-1$, 
\beq \label{eq:xxtordd}
	\xx_i(T_i)=\xx_{i+1}(T_i)-1\ge \xx_{i+1}(T_{i+1})-1.
\eeq

The same consideration shows that if we consider $m$ particles $\xx_{k_1}(t_1), \cdots , \xx_{k_m}(t_m)$ of the TASEP in $\conf_N(L)$ at possibly different times, their joint distribution is same as that of the infinite TASEP if $t_i< T_{k_i}$ for all $i$. 
Therefore, we obtain~\eqref{eq:aux_2017_07_01_09} if we show that under the condition~\eqref{eq:aux_2017_07_01_08}, the event that $\xx_{k_i}(t_i)\le a_i$ for all $1\le i\le m$, is a subset of the event that 
$t_i< T_{k_i}$ for all $1\le i\le m$.
Now, suppose that $t_{i} \ge T_{k_i}$ for some $i$. 
Then, writing $\ell=k_i$ and using~\eqref{eq:xxtordd}, 
\begin{equation*}
	\begin{split}
	\xx_{\ell}(t_i)&\ge \xx_{\ell}(T_{\ell}) \ge \xx_{\ell+1}(T_{\ell+1})-1
	\ge \cdots \ge \xx_{N}(T_N)-(N-\ell).
	\end{split}
	\end{equation*}
Since $x_{N}(\tau_N)=\xx_1(T_N)+L-1$ and $\xx_1(T_N)\ge \xx_1(0)=y_1$, this implies that (recall that $\ell=k_i$)
\begin{equation*}
	\begin{split}
	\xx_{\ell}(t_i) \ge y_1+L-1 -(N-\ell) > a_{\ell}
	\end{split}
\end{equation*}
using the condition~\eqref{eq:aux_2017_07_01_08}. 
Hence we are not in the event that $\xx_{k_i}(t_i)\le a_i$ for all $1\le i\le m$. 
This completes the proof. 
\end{proof}

The above result implies, using the inclusion-exclusion principle, 
\begin{equation} 
	\prob\left( \tilde{x}_{k_1}(t_1)\ge a_1,\cdots, \tilde{x}_{k_m}(t_m)\ge a_m\right)
	=\prob\left( \xx_{k_1}(t_1)\ge a_1,\cdots, \xx_{k_m}(t_m)\ge a_m\right) 
\end{equation} 
for $L$ satisfying 
\begin{equation}
	\label{eq:aux_2017_10_09_01}
	L > \max\{a_1-k_1,\cdots, a_m-k_m,y_N-N\} - y_1+ N.  
	\end{equation}

Therefore, Theorem~\ref{prop:multipoint_distribution_origin} implies that 
\begin{equation} \label{eq:aux_2017_07_01_10}
	\prob\left( \tilde{x}_{k_1}(t_1)\ge a_1,\cdots, \tilde{x}_{k_m}(t_m)\ge a_m\right)=\mbox{ the right hand side of~\eqref{eq:multipoint_distribution_origin}} 
\end{equation} 
for any $L$ satisfying~\eqref{eq:aux_2017_10_09_01}. 
In particular, for the initial condition $\tilde{x}_i(0)=i-N$, $i=1, \cdots, N$, by Theorem~\ref{thm:multi_point_formula}, we find that 
\begin{equation}
\label{eq:aux_2017_08_14_01}
	\prob\left( \tilde{x}_{k_1}(t_1)\ge a_1,\cdots, \tilde{x}_{k_m}(t_m)\ge a_m\right)
	=\mbox{ the right hand side of~\eqref{eq:multi_point_formula}}
\end{equation}
for any integer $L$ satisfying  
\begin{equation} \label{eq:Lmaxf}
	L \ge  2N+ \max\{a_1-k_1,\cdots, a_m-k_m, -N\} . 
\end{equation}
Note that since the particles move only to the right, the above joint probability is same as that of infinite TASEP (with infinitely many particles) with the step initial condition. 
Hence we obtained a formula for the finite-time joint distribution in multiple times and locations of the infinite TASEP with the step initial condition. 
Actually we have infinitely many formulas, one for each $L$ satisfying~\eqref{eq:Lmaxf}. 
Since the infinite TASEP does not involve the parameter $L$, all these formulas should give an equal value for all $L$ satisfying~\eqref{eq:Lmaxf}.

Now, if we want to compute the large time limit of the joint distribution of the infinite TASEP under the KPZ scaling, we need to take  $a_i=O(t)$. 
The above restriction on $L$ implies that $L\ge O(t)$. 
This implies that $t\ll L^{3/2}$, which corresponds to the \emph{sub-relaxation} time scale.   
Hence the large-time limit of the joint distribution of the infinite TASEP is equal to the large-time limit, if exists, of the joint distribution of the periodic TASEP in the sub-relaxation time scale. 
However, it is not immediately clear if the formula~\eqref{eq:multi_point_formula} 
is suitable for 
the sub-relaxation time scale when $m\ge 2$.
In particular, the kernels $\KsL(w,w')$ and $\KsR(w,w')$ do not seem to converge. 
We leave the analysis of the multi-point distribution of the infinite TASEP as a future project.



\def\cydot{\leavevmode\raise.4ex\hbox{.}}

\end{document}